\documentclass[11pt,letterpaper]{amsart}
\usepackage{times,amsfonts,mathabx,amssymb,amsrefs,mathrsfs,tikz}
\usetikzlibrary{decorations,decorations.pathmorphing}
\usepackage{todonotes}

\newcommand{\N}{{\mathbb N}}
\newcommand{\Q}{{\mathbb Q}}
\newcommand{\Z}{{\mathbb Z}}
\newcommand{\R}{{\mathbb R}}

\newcommand{\acts}{\curvearrowright}

\newcommand{\Span}{{\rm span}}

\DeclareMathOperator{\supp}{supp}

\DeclareMathOperator{\Norm}{Norm}
\DeclareMathOperator{\Tushev}{Krull}
\DeclareMathOperator{\HAKquotient}{K}
\DeclareMathOperator{\HNMsubgroup}{H}
\DeclareMathOperator{\A'}{A_0}

\newtheorem{thm}{Theorem}[section]
\newtheorem{prop}[thm]{Proposition}
\newtheorem{notation}[thm]{Notation}
\newtheorem{claim}[thm]{Claim}
\newtheorem{cor}[thm]{Corollary}
\newtheorem{lem}[thm]{Lemma}
\theoremstyle{definition}
\newtheorem{defn}[thm]{Definition}
\newtheorem{question}{Question}
\newtheorem{rem}[thm]{Remark}

\newtheorem{exa}[thm]{Example}
\newtheorem*{thm*}{Theorem}
\newtheorem*{cor*}{Corollary}
\newcommand{\thistheoremname}{}

\newtheorem*{genericthm*}{\thistheoremname}
\newenvironment{namedthm*}[1]
  {\renewcommand{\thistheoremname}{#1}%
   \begin{genericthm*}}
  {\end{genericthm*}}
\newtheorem*{genericcor*}{\thistheoremname}
\newenvironment{namedcor*}[1]
  {\renewcommand{\thistheoremname}{#1}%
   \begin{genericthm*}}
  {\end{genericthm*}}

\title[]{Poisson boundary of nilpotent-by-abelian groups}

\author{Anna Erschler}
\address{A.E.: C.N.R.S., \'{E}cole Normale Superieur, PSL Research University, 45 rue d'Ulm, 75005, Paris,  France}
\email{anna.erschler@ens.fr}
\author{Josh Frisch}
\address{J.F.: UC San Diego S9500 Gilman Drive 
La Jolla CA 92093, USA}
\email{joshfrisch@gmail.com}

\date{\today}
\begin{document}

\begin{abstract}
Given a finitely generated group, the well-known Stability Problem
asks whether the non-triviality of the Poisson-Furstenberg boundary (which is equivalent to the existence of non-constant bounded harmonic functions) depends on the choice of 
simple random walk on the group.  This question was far from being understood even in the class of
linear groups. Given an amenable group, e.g. a solvable group, there is no  known characterisation, even a conjectural one, of when it admits a simple random walk with non-trivial boundary.
We provide a characterisation of groups with  non-trivial boundary for  finitely generated linear groups of positive characteristic. We prove in particular that the Stability Problem has a positive answer in this class
of groups. For linear groups of characteristic $0$, we prove a sufficient condition for the triviality of the boundary which does not depend on the choice of a simple random walk. We conjecture that our sufficient condition is also necessary.
Our arguments are based on a new comparison criterion for group extensions, on new $\Delta$-restriction entropy estimates and a criterion for boundary non-triviality,  and on a new "cautiousness" criterion for triviality of the boundary.

\end{abstract}

\maketitle


\section{Introduction and statement of main results}


Given a group of matrices over a field, Tits alternative \cite{tits1972} implies  that the group is non-amenable if and only if it contains a free subgroup. An argument  of Milnor \cite{milnor68} and Wolf \cite{wolf68} about the growth of solvable groups, and its refinement in \cite{rosenblatt74} implies that the growth of such a finitely generated group is exponential if and only if the group contains a free sub-semigroup. A well-known property weaker than subexponential growth and stronger than amenability is the Liouville property of simple random walks on a group, stating that all bounded harmonic functions are constant (for this and other basic definitions related to random walks see Section \ref{sec:background}). For fields of positive characteristic, we characterize non-Liouville solvable groups as those that contain the three-dimensional lamplighter group as a subgroup of one of its blocks (naturally defined metabelian groups associated to our group in a manner which we describe below). We show that the characteristic zero case is essentially different. For this case, we give a sufficient criterion for triviality of the boundary and conjecture that this criterion is also necessary.

Our arguments are based on a new comparison
criterion for 
the Liouville property of group extensions.
A probability measure on a group $G$ is said to be {\it non-degenerate}
if its support generates $G$ as a semi-group.

\begin{namedthm*}{Theorem A} (= Thm \ref{thm:compcrit}, Comparison criterion)
Let $1\to F\to G \to H\to 1$
and $1\to F\to G_2\to H\to 1$
be short exact sequences where the group $F$ is abelian, and suppose that the induced action of $H$ on $F$ is the same for both extensions above.  Let $\mu_1$ and $\mu_2$ be finite entropy  measures on $G,G_2$ which
have the same  projections to $H$. Let $f\in F$, assume that  $f$ acts non-trivially on the Poisson boundary of $(G,\mu_1)$. Then, if $\mu_2$ is non-degenerate, $f$ acts non-trivially on the Poisson boundary of $(G_2,\mu_2)$.
\end{namedthm*}
The assumption of finite entropy in the formulation of the theorem is essential, see Remark \ref{rem:finiteentropyimportant}. The assumption that $\mu_2$ is non-degenerate cannot be replaced by the weaker condition of {\it irreducibility} (that the support of $\mu_2$
generates $G_2$ as a group),
see Remark \ref{rem:nondegenerateimportant}.

We now explain our application of the  comparison criterion to linear groups.
Since the Poisson boundary
is non-trivial for any non-degenerate measure on a non-amenable  group,
and since
any finitely generated amenable linear group is virtually solvable and in particular virtually nilpotent-by-abelian (Malcev -- Kolchin theorem, \cite{malcev}), the question of which linear groups have the Liouville property is essentially about virtually nilpotent-by-abelian groups.


For linear groups, the number of the afore-mentioned metabelian blocks is finite, 
and they can be described explicitly.
Throughout the paper $\theta$ denotes the following matrix
$$
\theta=
\left( \begin{array}{ccc}

1 & 1  \\

0 &  1
 \end{array} \right),
$$ 

Given a group $G$ of  $n\times n$ upper triangular matrices and
$(i,j)$, $1 \le i < j \le n$, we consider all
 matrices of the form
$$
G_{i,j}=
\left( \begin{array}{ccc}
g_{i,i} & 0  \\
0 &  g_{j,j}
 \end{array} \right), 
$$
defined when there exists $g\in G$ with the entries 
$g_{i,i}$ and $g_{j,j}$ at $(i,i)$ and $(j,j)$.

We say that  $(i,j)$, $1 \le i < j \le n$, is {\it a validating pair} of indexes if 
there exists an element $g$ in $G$ such that its matrix entries have the three following properties
\begin{itemize}
\item $g_{i,i}=1$ for all $i: 1 \le i \le n$.
\item $g_{i,j} \ne 0$.
 \item $g_{i',j'}=0$  for all $(i', j')\ne (i,j)$ such that $i \le i' < j' \le j$
\end{itemize}

We also say that $g$ is an $(i,j)$-{\it validating element}.
In this case
we consider the group generated
by $\theta$ and the matrices $G_{i,j}$.
We call this group the $(i,j)$-basic block of $G$
and denote it by  $B_{i,j}$.  If $(i,j)$ is a validating pair of indexes, we say that this $(i,j)$-basic block is valid.
Otherwise, if  $(i,j)$ is not a validating pair of indexes,
 we say  the 
$(i,j)$-basic block is trivial.
See Section \ref{sec:moreonreductiontoblocks} and Definition \ref{def:blocks} for details.

An example  of the matrix below is an illustration of the fact that 
for a linear group, containing this matrix,  the pair of indexes $(i,j)$ is validating (for $i=2$ and $j=5$). The symbol
$\ast$ as a matrix entry on the picture
denotes a non-zero element of our field.

\[ 
    \bordermatrix{ &  &  &    &  & j &          \cr
       & 1 & ? & ?   & ? & ? & ?      \cr
      i &  & 1 & 0  & 0 & * &  ?       \cr
       &  &  & 1 &  0 & 0 & ?        \cr
       &  &  &   & 1 & 0 & ?       \cr
       &  &  &  &   & 1 & ?        \cr
       &  &  &   &  &  & 1         } \qquad
\]

Observe that a group of upper-triangular matrices admits a natural homomorphism to the (abelian) group of diagonal matrices. We restrict this homomorphism to $G$. We have $G \to k^n$,  $g \to (g_{1,1}, g_{2,2}, \dots, g_{n,n})$.
For given $i,j: 1 \le i,j \le n$ consider a mapping $G \to k^2$, where
$G \to (g_{1,1}, g_{2,2}, \dots, g_{n,n}) \to (g_{i,i}, g_{j,j})$.
We show that  our random walk $(G,\mu)$ is not Liouville, if and only if  for some
$(i, j)$ the random walk on the basic metabelian block
$B_{i,j}$ is not Liouville, for some probability measure $\mu_{i,j}$ on $B_{i,j}$  having the same projection  as $\mu$ to $k^2$.
The measures on $B_{i,j}$ with this property will be called {\it associated} measures.

More precisely, we prove


\begin{namedthm*}{Theorem B} ( (1) and (2) of Thm  \ref{thm:lineargroups})

Let $G$ be a group of upper-triangular $n\times n$ matrices over a field.

\begin{enumerate}

\item The Poisson boundary of $(G,\mu)$
is non-trivial if and only if there exist $i,j: 1 \le i<j \le n$ such that
  for any associated measures $\mu_{i,j}$ on the $(i,j)$-block $B_{i,j}$ 
the Poisson boundary of $(B_{i,j}, \mu_{i,j)}$ is non-trivial.

\item 
An element $g\in G$ acts non-trivially on the
boundary $(G,\mu)$ if and only if there exists an unipotent
$h \in \Norm_G (g)$  and $i,j: 1 \le i<j \le n$ such
that any associated measure on the block $B_{i,j}$ has  non-trivial boundary and
 $h$ is $(i,j)$-validating element.


\end{enumerate}
\end{namedthm*}

Both claims of the theorem also hold if we replace "any associated measure" by  "some  associated measure".

In Theorem \ref{thm:nilpotentbyAbelian}
we explain an analogous reduction statement 
in a general context of nilpotent-by-abelian groups.
We recall that there are uncountably many nilpotent-by-abelian groups (even if we consider these groups up to quasi-isometry, in particular they have a continuum of distinct F{\o}lner functions, see \cite{erschlerzheng}). In contrast with the linear case studied in Theorem B, the number of blocks of a nilpotent-by-abelian group is not necessarily finite.

\subsection{Metabelian groups and metabelian blocks}

After applying Theorem B, we are left to determine which random walks on metabelian blocks are Liouville. To address this question, we prove a new criterion for triviality of the boundary and a new criterion for non-triviality of the boundary. 

Our first new criterion is a 
 {\it cautiousness criterion} for the Liouville property, see Theorem \ref{prop:cautious}. The argument uses in an essential
way the Shannon -- McMillan -- Breiman type theorem (\cites{kaimanovichvershik, derriennic}), and its  
  idea is sketched in the beginning of Section \ref{sec:cautious}. As one of the applications, 
we obtain Corollary \ref{cor:rank2rank1}, a version of which we formulate below.
\begin{namedcor*}{ Corollary C}
Let $G$ be an (amenable) linear group over $k$ which has a finite index upper-triangular subgroup $G_2$.
Suppose that none of  the basic blocks of $G_2$ contains a three-dimensional wreath product as a subgroup.
If the characteristic of $k$ is zero, 
we assume additionally  
that any basic block of $G_2$ which contains $\mathbb{Z}^2\wr \mathbb{Z}$ as a subgroup is isomorphic to this wreath product.
Then for any simple random walk (and any symmetric finite second moment random walk) on $G$ the boundary is trivial. 
\end{namedcor*}

One of the corollaries of Corollary $C$  is discussed in Corollary \ref{cor:variablesfield}.
If $F$ is either a function field over at most $2$ variables in positive characteristic or $F$ is a function field over at most $1$ variables in characteristic $0$
then we prove that any virtually solvable group which is linear over F is Liouville: 
any finite second moment centered measure on this group has trivial
boundary. Observe that the case of transcendence degree $0$ recovers the previously known case of linear  groups over $\mathbb{Q}$ (see \cite{schapira}; see also \cite{brofferioschapira} for the description of the boundary for  random walks on $GL(n, \mathbb{Q})$).

More generally, using a notion of Krull dimension for metabelian groups (see the beginning of Section \ref{sec:section7} for a discussion of this notion),
 we give a complete characterization of linear amenable  linear groups over a field of positive characteristic such that
 the boundary of simple random walks is non-trivial (see Corollary \ref{cor:1234charp} in 
Section \ref{sec:nontrivial}, one of its claims is
 Corollary D  below).  We prove that,  passing to a finite index subgroup, these groups correspond to upper-triangular linear groups such that at least one of the basic blocks 
 has
 dimension $\ge 3$.
 We mention that characteristic $p$ groups of dimension $0$ are virtually abelian. 
Triviality of the boundary for simple random walks on linear groups, positive characteristics  of dimension 1 follows by combining  Thm 1.1 in
\cite{jacoboni} with \cite{saloffzheng}. 
The main positive characteristic application of our triviality of the boundary criterion is 
in dimension  $2$. Here we get new examples of metabelian groups with trivial boundary for simple random walks (such as the Baumslag group for $d=2)$. As before, any linear group with such blocks also provides an example.
In dimension $3$
we get many new examples of groups with non-trivial boundary.

Now in characteristic $0$ the situation was known for simple random walks in  dimension $1$ (this follows from the already mentioned result of Brofferio -- Schapira \cite{brofferioschapira} for linear groups over $\mathbb{Q}$. For linear groups over $\mathbb{Z}$ this result was known earlier due to Kaimanovich \cites{kaimanovichsolvable, kaimanovichpolycyclic}).
Our result shows that in dimension $2$ the boundary of simple (or centered finite second moment) random walks is trivial, and we get some new examples.
In dimension $\ge 4$,  we also get many new examples of non-trivial boundary, in particular among metabelian groups containing a three-dimensional wreath product as a subgroup. The only case not covered by our results is therefore dimension $3$ and characteristic $0$.

We give examples (see Proposition \ref{prop:Galpha}) of groups of dimension $3$ in characteristic $0$ having non-trivial boundary of random walks. Furthermore, we believe (see Question \ref{question:rank2lamplighter}  below) that except for the well-known examples of two-dimensional wreath products, and groups that can be reduced to these wreath products as applications of our Theorem B, all simple (or finite entropy non-degenerate) random walks on such groups of dimension $3$ have non-trivial boundary.
If the answer to Question \ref{question:rank2lamplighter} is positive, one would get a complete characterization of metabelian groups with the Liouville property and in view of our theorem, thus for all linear groups of characteristic $0$.

More precisely, we conjecture that the assumption of  Corollary C  above is not only sufficient, but necessary. First we ask this question in characteristic $0$. Note that every finitely generated linear group over a field of characteristic $0$ is virtually torsion-free by Selberg's lemma, so it suffices to analyze torsion-free groups. 

\begin{question}\label{question:rank2lamplighter}
Let $k$ be a field of characteristic $0$ and let $G$ be a 
linear group over $k$.
Assume that every  finite index upper-triangular subgroup $G_2$ of $G$ admits  a basic block containing a two-dimensional wreath product but not isomorphic to such a wreath product. Is it true that every non-degenerate finite entropy random walk on $G$ has  non-trivial boundary?
\end{question}

In the case where the field $k$ has positive characteristic, we prove that the sufficient conditions of Corollary C are also necessary.

\begin{namedcor*}{Corollary D} Let $G$ be an amenable linear group over a field of  characteristic $p$. Then a simple random walk on $G$ has non-trivial boundary if and only if one of the basic blocks of $G$ contains $  \mathbb{Z}/p \mathbb{Z}  \wr \mathbb{Z}^3$ as a subgroup.
\end{namedcor*}

For a more general result, we refer once more to Corollary \ref{cor:1234charp}. 

We also prove a partial result in this direction in characteristic 0.
If some basic block of $G$ (or of $G_2$) contains a three-dimensional wreath product, then the boundary of any non-degenerate finite entropy random walk is non-trivial (this is a part of the claim of Corollary \ref{cor:metablockrank3}). 
We also provide further examples in support of our conjectured positive answer to Question
\ref{question:rank2lamplighter} 
in Proposition \ref{prop:Galpha}.

In view of our results, we believe that the Liouville property of finitely generated linear groups is determined by the asymptotics of the F{\o}lner function of their blocks.
For general finitely generated groups, it is known that
a group with arbitrarily quick growing
F{\o}lner  function can have the Liouville property. Arbitrarily quick growing F{\o}lner functions can even be found among groups of
intermediate growth \cite{erschlerpiecewise}. For a class of elementary amenable groups with prescribed
F{\o}lner  function and trivial boundary see \cite{brieusselzheng}.
We conjecture that for linear groups of $2\times 2$ matrices
simple random walks are not Liouville if and only if the F{\o}lner  function is $\ge \exp(Cn^3)$
for some $C>0$.

We prove our results on the non-triviality of the above mentioned boundaries by developing a $\Delta$-restriction entropy criterion, Theorem \ref{prop:rank3} (and combining this with our comparison criterion). While transience of some associated random walks appears in many examples (including non-degenerate random walks on wreath products) as a sufficient condition for boundary non-triviality, Lemma \ref{lem:mainlemma}
in the proof of Theorem \ref{prop:rank3} uses in an essential way a notion of uniform strong transience (see Definition \ref{def:uniformlysr}), and the claim is false under 
the weaker assumption of transience (see Remark \ref{rem:notsimplytrans}).

We remind the reader that a well-known open problem asks whether the non-triviality of the Poisson boundary depends on the choice of measure for simple random walks on a finitely generated group $G$. See for example the survey \cite{saloffsurvey}, where Saloff-Coste attributes this question to Varopoulos
in the early '80s.
It is well-known that a closely related question about quasi-isometric graphs has a negative answer  (see T.Lyons \cite{lyons} for examples of quasi-isometric graphs, one of which is Liouville, and the other is not. Moreover, such examples can even be found among graphs of polynomial growth, see  \cite{Benjamini}). 
For the moment, this stability was known only for some  
classes where all  groups have trivial boundary (i.e. subexponential growth, as a well-known corollary of the entropy criterion \cite{avez2}; polycyclic groups or solvable groups or finite Pruefer rank,
groups with (quite) slow decay of return probability (\cite{saloffzheng}, \cite{pereszheng})
and
in very particular examples (such as wreath products).
Our results in Corollary C and its partial converse in Corollary D show stability for any linear groups of positive characteristics (as well for a large class of groups of zero characteristics). Our conjecture, if proven, would show stability for all linear groups.  We also mention that Kaimanovich  has expressed recently an opinion that, despite the lack of known candidates for a counterexample, the  answer for the Stability Problem might be negative in the class of all finitely generated groups (\cite{thompson}, 1.D).

{\bf Acknowledgements.} 
We are grateful to Vadim Kaimanovich and Olga Kharlampovich for helpful discussions.
This project has received funding from the European Research Council (ERC) under
the European Union's Horizon 2020 
research and innovation program (grant agreement
No.725773). The work of the second named author was also supported by NSF grant 2102838. 

\section{Definitions, notations  and basic facts}\label{sec:background}

First we recall terminology about group extensions.
Given a short exact sequence
$$
1 \to \HNMsubgroup \to G \to \HAKquotient \to 1,
$$
one says that $G$ is  an extension of $\HNMsubgroup$ by a group $\HAKquotient$.
However, if $\HNMsubgroup$ and $\HAKquotient$ have some properties, for example if $\HNMsubgroup$ is nilpotent and $\HAKquotient$ is abelian, it is said  that $G$ is "nilpotent-by-abelian" (and not "abelian-by-nilpotent"). See \cite{tao}, Definition 2.4.2 and a footnote for this definition, 
for  reflections about possible origins of this inconsistent terminology.

Now we recall definitions and background related to random walks.
Given a probability measure $\mu$ on a group $G$, a (right) {\it  random walk} $(G,
\mu)$ is a Markov chain with state space $G$
and with transition probabilities $p(g|gh)=
\mu(h)$, for all $g,h \in G$. We will assume that the random walk starts at the identity $e_G \in G$.
We have already mentioned that a measure $\mu$ on $G$ is said to be irreducible
(also called adapted) if its support generates $G$
as a group.
A measure
$\mu$ on a group $G$ is  {\it non-degenerate} if the support of $\mu$ generates $G$ as a semi-group. More intuitively, non-degeneracy means that the random walk, starting at the identity,  visits any point $g\in G$ with positive probability.
We now recall the definition of the Poisson boundary.


\begin{defn}[Poisson boundary]

Consider the space of one-sided infinite trajectories $G^\infty$.
For two  trajectories
$X$ and $Y$ in this space,  we say  that they are equivalent if they coincide after some instant, possibly up to the
time shift. This means
there exist $N,k\ge 0$ such that $X_i=Y_{i+k}$ for all $i>N$. Consider the measurable
hull of this equivalence relation in $G^\infty$. The quotient
by the obtained equivalence relation is called the {\it Poisson boundary} (also called the {\it Poisson-Furstenberg boundary}).
\end{defn}
Given an element $g\in G$, we say that $g$ acts on the space of trajectories by
$$g(x_1,x_2,x_3,\dots) = (gx_1,gx_2,gx_3,\dots).$$ This action preserves the equivalence relation in the definition of the Poisson boundary and thus induces an action on the boundary.

The Poisson boundary can also be defined for any Markov chain. 
In the case of random walks on groups, there are several equivalent definitions. In particular, if we do not allow the shift by $k$ in the definition above, and say that $X$ and $Y$ are equivalent if $X_i=Y_{i}$ for all $i>N$, then we obtain in this way the definition of the {\it tail boundary}.
It is not difficult to see that for a general Markov chain this boundary can be larger than the Poisson boundary, and a random walk can have a trivial Poisson boundary and a non-trivial tail boundary. In the case of random walks on groups, this cannot happen and the two definitions are equivalent. For these and other definitions and basic facts about boundaries see  \cites{kaimanovichvershik,furman, erschlerkaimanovich2019}.

We also mention that in the case when the Markov chain is a random walk on a group,  there are structural results where the group structure of the space is essential.

A function $F:G \to \R$ is called $\mu$-{\it harmonic}, if for all $g\in G$ it holds
$F(g) = \sum_{h\in G} F(gh) \mu(h)$. The Poisson boundary 
 can be equivalently defined in terms
of bounded harmonic functions
on the subgroup of $G$,
generated by the support of $\mu$ and with values in $\mathbb{R}$. In particular 
the non-triviality of the Poisson boundary is equivalent to the  existence of  non-constant bounded harmonic functions.

Given a probability 
measure $\mu$,
its {\it entropy} $H(\mu)$
is defined as

$$H(\mu)=-\sum_{g\in G} \mu(g) \log(\mu(g)).$$

\begin{defn}
We recall that the {\it entropy} of a random walk $(G,\mu)$, also called the {\it asymptotic entropy},
is defined as  the limit, as $n$ tends to infinity,  of $H(n)/n$, where $H(n)=H(\mu^{*n})$. $H(n)$ is called the entropy function of $(G,\mu)$.
\end{defn}

In the definition above $\mu^{*n}$ denotes the $n$-th convolution power  of $\mu$, and the above-mentioned limit always exists
rem:dimsew of the sub-additivity  $H(\mu^{*(n+k)})\le H(\mu^{*n})+H(\mu^{*k})$. The
latter is a consequence of the fact that $H(\nu_1 * \nu_2) \le H(\nu_1)+ H(\nu_2)$, for any probability measures $\nu_1$ and $\nu_2$ on a group $G$.
Moreover, a Shannon-McMillan-Breiman type theorem
holds which states that with probability
$1$ the limit $\lim_{n\to \infty} - \frac{1}{n}\ln(\mu^{*(n)}(g_n))$ exists and is equal to the asymptotic entropy 
(Thm $2.1$  in \cite{kaimanovichvershik}, \cite{derriennic}).

The notion of entropy of random walks was introduced by Avez \cite{avez}, who showed that for finitely supported measures $\mu$
if the asymptotic entropy is zero then the Poisson boundary is trivial \cite{avez2}. We know now that the statement is if and only if and, moreover, holds true for any finite entropy measure, regardless of the size of the support, because of the following Entropy criterion of Derriennic and Kaimanovich -- Vershik.

\begin{thm}[Entropy criterion, \cites{kaimanovichvershik1, kaimanovichvershik, derriennic }] 
Let $\mu$ be a probability measure on a group $G$.
If the measure $\mu$ has  finite entropy, then the Poisson boundary of the random walk  $(G,\mu)$ is trivial
if and only if the entropy of the random walk is $0$.
\end{thm}

The Entropy criterion, in particular, implies that if the growth of the group $G$ is subexponential, then
any finitely supported random walk (and more generally any random walk of finite first moment) has a trivial Poisson boundary.

Let $S$ be a generating set of $G$.
The word length $l_S(g)$ is the shortest number of terms needed to write $g$
as a product of elements of $S$.
The growth function  $v_{G,S}(n)$ counts the number of group elements of word length at most $n$. 
\begin{defn}
Suppose that a probability measure $\mu$ on $G$ has finite first moment, that is, $\sum_g \mu_g l_S(g)<\infty$. (It is clear that the finiteness of the first moment does not depend on the choice of finite generating set $S$ of $G$).
The drift function $L(n)$ is the expected length of the position at time instant $n$
$$
L(n) = E[l(x_n)].
$$
It is clear again that $L(n+m) \le L(n)+L(m)$, and thus there exists
a limit $l=\lim_{n\to \infty}L(n)/n$.
This limit is called the {\it drift} (or {\it rate of escape}) of the random walk $(G,\mu)$.
\end{defn}

It is not difficult to see that $l=0$
implies $h=0$ (see \cite{guivarch}).
A random walk is said to be {\it symmetric} if the defining measure $\mu$ satisfies $\mu(g)=\mu(g^{-1})$,
for all $g$.
If the random walk is symmetric, the converse statement is true: if $h=0$ then $l=0$ (This was shown by Varopoulos \cite{varopoulos} under the additional assumption that the support of $\mu$ is finite and by  Karlsson -- Ledrappier for the general finite first moment case
\cite{karlssonledrappier}).

Among groups of exponential growth 
there are many known examples with trivial boundary of simple random walks, and there are many known examples with non-trivial boundary.
For many groups, it is a challenging problem to understand whether the boundary is non-trivial.

\begin{defn}
Given groups $A$ and $B$, 
{\it a wreath product} $B\wr A$ is a semi-direct product of $A$ and $\sum_A B$,
where $A$ acts on $\sum_A B$ by shifting the index set.
\end{defn}

\begin{rem}[Notations for wreath products]
Some papers use $A\wr B$ notation for what we denote now $B\wr A$. The convention $A\wr B$,
for the acting group $A$, goes back to works of P\'{o}lya (see \cite{polya}, definition
of "kranz" on p. 178 and examples throughout that paper
, who used the notation $A[B]$, $A$ is the acting  group, for the wreath product of $A$
and $B$; of Kaloujnine \cite{kaloujnine45},  Kaloujnine and Krasner \cite{kaloujninekrasner48}
, the notation $\circ$ for the wreath product sign in their first papers on the subject, and the notation $A {\rm WR } B$ and  $A\wr B$, in this order, in later works of Kaloujnine and his collaborators and students (see e.g. \cite{KaluzninBeleckijFejnberg}). That notation was used in many
papers of the first named author.
\end{rem}

If $B=\mathbb{Z}/2\mathbb{Z}$, wreath products  $B\wr A$ are also called {\it lamplighter} groups.
The  elements of lamplighter groups are pairs $(a,f)$,
$f:A \to \mathbb{Z}/2\mathbb{Z}$ is a finitely supported function, which we can view as "the lamp configuration". In the sequel, we denote by $\theta$ the element $(e_A,\chi_{e_A})$, "the lamp" at the identity of $A$.

One-dimensional and two-dimensional lamplighter groups
$\mathbb{Z}/2\mathbb Z \wr \mathbb{Z} $ and $ \mathbb{Z}/2 \wr \mathbb{Z}^2$
provide examples of groups of exponential growth with trivial boundary for simple random walks \cite{kaimanovichvershik}.
On the other hand
$\mathbb{Z}/2\mathbb{Z}  \wr \mathbb{Z}^d$ where $d>2$ provides examples of amenable groups where simple random walks with non-trivial boundary \cite{kaimanovichvershik}.
 Observe that $\mathbb{Z}/2\mathbb{Z} \wr \mathbb{Z}^d$
 are examples of (torsion-abelian)-by-abelian
 groups. For some groups in this class (e.g. Baumslag groups which will be discussed in Example \ref{ex:Baumslag}) the question was open and we give a complete characterization of groups in this class.

The results of \cite{kaimanovichvershik} are also useful to determine which elements act trivially or not trivially on the Poisson boundary.
A consequence of Theorem 3.2 (together with Proposition 3.5) of \cite{kaimanovichvershik} is that the differential entropy of any non-trivial quotient of the Poisson boundary is strictly smaller than that of the Poisson boundary itself. 
We apply it to the quotient of the Poisson boundary by the action of the normal subgroup generated by an element $h$, and taking into account that differential entropy coincides with the asymptotic entropy (see \cite{kaimanovichvershik}, Thm 3.1) we get the following claim.
Given $f\in G$, 
$\Norm_G(f)$ denotes the normal subgroup of $G$, generated  as a normal subgroup by $f$. If $G$ is clear from context, then we omit $G$ in the notation and write $\Norm(f)$.

\begin{claim}  \label{lem:strictinequality}
Let $G$ be a countable group, and 
 $\mu$ be a probability measure on $G$. Then $f \in G$ acts trivially on the Poisson 
boundary of $(G,\mu)$ if and only if
$$
h(G,\mu) = h (G/\Norm_G(f) , \mu),
$$
where we use the same notation $\mu$ for its image in the quotient group $G/\Norm_G(f)$.
\end{claim}

Now we mention that equivariant quotients of Poisson-Furstenberg boundaries (in the category of measure spaces) are called {\it $\mu$-boundaries}
of the random walk $(G,
\mu)$.

We also recall the definition of conditional entropy.
Let $\bf P$ be 
the space of one-sided infinite
trajectories $G^\infty$. 
Consider a $\mu$-boundary $B$.
Since $B$ is a quotient of this space $\bf  P$,
the points of $B$
define conditional measures on $\bf P$.
Let us say that a probability measure in the space of one-sided infinite trajectories $G^\infty$ has asymptotic entropy $h$ if its one-dimensional distributions $\lambda_n$ ($n$ step distribution of the corresponding conditional process) satisfy 
$$
- \ln \lambda_n(x_n)/n \to h
$$
for almost all $x_i \in G^\infty$.

We recall that the Entropy criterion  of Kaimanovich \cite{kaimanovichhyperbolic}, Theorem 4.6 states the following.
 
\begin{thm}[Conditional entropy criterion]\label{kaimanovich:conditional}
Let $\mu$ be a probability measure on a group $G$.
If the measure $\mu$ has a finite entropy, then a $\mu$-boundary $B $
is maximal (equal to the Poisson boundary) if and only if the asymptotic entropy is $0$ for almost all $b\in B$.
\end{thm}

\section{Examples}\label{sec:examples}
In this section, we describe several examples illustrating the application
of Theorem B and the notion of metabelian blocks.

Throughout the paper $F_d$ denotes the free group on $d$ generators.
In many examples below we consider matrices 
over $\mathbb{Z}[X_1^{\pm 1}, \dots,  X_d^{\pm 1}]$. Alternatively, by choosing $X_i$ to be mutually transcendental numbers in $\mathbb{C}$, we can  consider these matrices to be over $\mathbb{C}$. 

A group is said to be metabelian if it is solvable of step $2$, that is, if it is an extension of an abelian group by an abelian group.
We recall that any metabelian group is linear over a commutative ring
\cite{remeslennikov68}.
If the commutator subgroup of a finitely generated metabelian group is torsion-free then it can furthermore be represented by triangular matrices over a field of characteristic zero; if the commutator subgroup is of exponent $p$, for $p$ prime, 
then the group can be represented by triangular matrices over a field of characteristic $p$ 
\cite{remeslennikov69}; the statement moreover holds for metabelian  (and virtually metabelian) groups
with the commutator being a $p$-group,
see \cite{wehrfritz75}, Thm 1.1.


\begin{exa} [Lamplighter on $\mathbb{Z}^d$] \label{exa:deflamplighter}
Consider the wreath product $G_d = \mathbb{Z} \wr \mathbb{Z}^d$.
This group is isomorphic to the group generated by the matrix $\theta$ and  the following  matrices over $\mathbb{Z}[X_1^{\pm 1}, \dots,  X_d^{\pm 1}]$  
$$
%
%
 M_{x_i}=
\left( \begin{array}{ccc}
1 & 0  \\
0 & X_i\\
 \end{array} \right)
$$
Let $G_{d,p}= \mathbb{Z}/p \mathbb{Z} \wr \mathbb{Z}^d$.
This group is generated by  the above-mentioned matrices when we consider them over $\mathbb{Z}/p\mathbb{Z}[X_1^{\pm 1}, \dots X_d^{\pm 1}]$.
It is straightforward that for this matrix realization the block of each among these wreath products 
is equal to the group itself.
\end{exa}

The representation of free metabelian groups in the example below was studied in Lemma 2 of  \cite{wehrfritz69}.


\begin{exa} [Free metabelian and free $p$-metabelian groups]
Let $\rm{Met}_d$ denote the  free metabelian group  on $d\ge 2$ generators,
that is,
$F_d/[[F_d, F_d],[F_d,F_d]]$.
This group is generated by $d$ matrices of the form
$$
\rm{FM}_{x_i}=
\left( \begin{array}{ccc}
X_i^{-1} & S_i  \\
0 & X_i \\
 \end{array} \right)
$$ 
over $\mathbb{Z}[X_1^{\pm1}, \dots X_d^{\pm 1}, S_1, \dots, S_d]$.
Observe that each non-trivial element of the commutator group satisfies the  condition in the definition of a valid block, hence the $(1,2)$ block is valid and $\theta$ is one of its generators. 
Other generators are by definition
$$
\left( \begin{array}{ccc}
X_i^{-1} & 0  \\
0 & X_i \\
\end{array} \right).
$$
It is clear that the block is isomorphic to $\mathbb{Z} \wr \mathbb{Z}^d $.

Let $\rm{Met}_d(p)$ be the free $p$-metabelian group; this group is by definition
\newline
$\rm{Met}_d/([\rm{Met}_d,\rm{Met}_d])^p$.
It is generated by the same matrices as $\rm{Met}_d$ if we consider them
over 
$\mathbb{Z}/p\mathbb{Z}[X_1^{\pm1}, \dots X_d^{\pm 1}]$.
With a similar argument as for the free 
metabelian group, we can conclude that the only block of this matrix is
$\mathbb{Z}/p\mathbb{Z} \wr \mathbb{Z}^d$.
\end{exa}

It is known that for some aspects free metabelian groups are analogous to 
wreath products, and in particular the Poisson boundary for simple random
walks on a free-metabelian group is non-trivial if and only if when $d\ge 
3$ \cite{Vershik2000}. This can now be deduced as one of the simplest manifestations of our general result: wreath products with the base group $\mathbb{Z}^d$ have the same blocks as free metabelian groups
on $d$-generators. Hence the boundary is trivial on a wreath product if and only if it is trivial for the corresponding free metabelian group.

In the example below the
dimension is at most two and in the next example the dimension is  three, see \ref{lem:rankdrops}.

\begin{exa}[Liouville quotients of Lamplighter groups]
Any proper quotients of the $3$-dimensional Lamplighter over $\mathbb{Z}/p\mathbb{Z}$, (with $p$ prime), provide examples of groups with trivial Poisson boundary for simple random walks (or any symmetric measure of finite second moment).
\end{exa}

\begin{exa}[Non-Liouville quotients of Lamplighter groups]
Quotients of the $d$-dimensional lamplighter group, $d\ge 4$, over a single relation are examples of groups 
where any non-degenerate finite entropy measure has a non-trivial boundary.
\end{exa}

A result of Baumslag \cite{Baumslag} states that any finitely generated metabelian group can be embedded into a finitely presented metabelian group.
A basic example of his construction is the following embedding of Lamplighter
groups $  \mathbb{Z}/ p \mathbb{Z} \wr \mathbb{Z}^d$.

\begin{exa}[Baumslag groups] (= Example \ref{ex:Baumslag})
Consider 
$B_d(\mathbb{Z}/p\mathbb{Z})$
to be a subgroup of $GL_{2}(R)$,
$R=\Z/p\Z[X_1^{\pm 1},\dots, X_d^{\pm 1}]$ ,
generated by the matrix $\theta$ and   $2d$ matrices of the form 
$$
%
%
 M_{x_i} =
\left( \begin{array}{ccc}
1 & 0  \\
0 & X_i\\
 \end{array} \right), 
\medspace \medspace
M_{x_i+1} =
 \left( \begin{array}{ccc}
1 & 0  \\
0 & X_i+1\\
 \end{array} \right).
$$

For $D=2$ all finite second moment centered random walks on these groups have a trivial boundary.
\end{exa}

We mention below two two-dimensional examples in characteristic $0$. By our results, these groups have  trivial boundary for centered finite second moment random walks. 
\begin{exa}(= Example \ref{exa:g23x})
The group $G_{2,3,X}$ is the group generated by the matrix $\theta$ and
matrices $M_2$, $M_3$, $M_x$ defined below
$$
 M_{2} =
\left( \begin{array}{ccc}
1 & 0  \\
0 & 2\\
 \end{array} \right),
 \medspace \medspace
 M_{3} =
\left( \begin{array}{ccc}
1 & 0  \\
0 & 3\\
 \end{array} \right),
  \medspace \medspace
 M_{x} =
\left( \begin{array}{ccc}
1 & 0  \\
0 & x\\
 \end{array} \right).
\medspace \medspace
$$
\end{exa}

Another one-dimensional example:
\begin{exa}
The group $G_{X,X+1,X+2}$ is generated by the matrix $\theta$ and
the following matrices over $\mathbb{Q}[X]$:
$$
 M_{x} =
\left( \begin{array}{ccc}
1 & 0  \\
0 & X\\
 \end{array} \right),
 \medspace \medspace
 M_{x+1} =
\left( \begin{array}{ccc}
1 & 0  \\
0 & X+1\\
 \end{array} \right),
  \medspace \medspace
 M_{x+2} =
\left( \begin{array}{ccc}
1 & 0  \\
0 & X+2\\
 \end{array} \right).
\medspace \medspace
$$
\end{exa}

If we consider the matrices in the example above over  $\mathbb{Z}/p \mathbb{Z}[X]$, the group we obtain is linear in characteristic $p$. 
But even in characteristic $0$ case, as in the example above, it provides an example of the triviality of the boundary.

The examples below for $d=2$
 provide (torsion-free) metabelian examples without three-dimensional wreath products as a subgroup and
with non-trivial boundary for simple random walks; as well as examples of  symmetric finite first moment
measures with non-trivial boundary on  groups without two-dimensional wreath products as subgroups.

\begin{exa}[Lamplighter-Baumslag-Solitar  group]\label{exa:defLampBS}
We consider the group $\rm{LBS}_d \subset GL(2,R)$ , $R=\Z/p\Z[X_1^{\pm 1},\dots, X_d^{\pm 1}]$,
generated by  $M_{X_1}, \dots M_{X_d}$, $\theta$ and 
$$
M_2 =
\left( \begin{array}{ccc}
1 & 0  \\
0 & 2\\
 \end{array} \right) 
$$
We call this  group "Lamplighter Baumslag-Solitar" since $\theta$
and $M_{X_1}, \dots M_{X_d}$ generate $\mathbb{Z}\wr \mathbb{Z}^d$, while $M_2$ and $\theta$ generate the solvable Baumslag-Solitar group since
$M_2^{-1} \theta M_{2} = \theta^2$.
A more general example of this type can also be considered if instead of $2$ we consider an algebraic number $\alpha$ which is not a root of unity (see Proposition \ref{prop:Galpha}). These groups have nontrivial boundary for finite entropy non-degenerate random walks whenever $d\geq2$. 
\end{exa}

Below we mention an  example of a metabelian group, containing the $3$-dimensional wreath product as a subgroup. By our results, all such groups have non-trivial boundary for simple random walks.
\begin{exa}\label{exa:defx1+x2+x3} Consider the group generated by $2 \times 2$ matrices $\theta$, $M_{x_1}$, $M_{x_2}$, $M_{x_3}$ and
$$
 M_{x_1+x_2+x_3} =
\left( \begin{array}{ccc}
1 & 0  \\
0 & X_1+X_2+X_3\\
 \end{array} \right).
$$
\end{exa}

We can also consider the
the following group (and many other examples) which also has nontrivial boundary for the same reason.
\begin{exa}
The group, generated by $\theta$, $M_{X_1}$, $M_{X_2}$, $M_{X_3}$ and 
$$
\left( \begin{array}{ccc}
1 & 0  \\
0 &  \sqrt{X_1+X_2+X_3+1}  -\sqrt{X_1+X_2+X_3} \\
 \end{array} \right).
$$

\end{exa}

Another general class of $2 \times 2 $ matrices containing three-dimensional wreath products (and having thus non-trivial Poisson boundary) is described in the following example.


The examples below will follow from Lemma \ref{exa:twobytworank} and Corollary \ref{cor:metablockrank3}.

\begin{exa}\label{exa:a/btrans3}
Let $G$ be a non-abelian upper triangular linear group over a field $k$. Consider the homomorphism $\phi$ from $G \to k$  which sends a matrix 
$$
\left( \begin{array}{ccc}
a & c  \\
0 & b \\
 \end{array} \right).
$$
to $a/b$. Let $k'$ be the minimal field which contains the image of $\phi(G)$. Then if the transcendence degree  of  $k'$ is $\ge  3$, then the boundary is nontrivial for any finite entropy random walk.
\end{exa}

We have mentioned in the introduction (Corollary \ref{cor:variablesfield}) that
if the transcendence degree of  $k'$ is $\leq 1$, then the boundary is trivial for centered finite second moment random walks. Furthermore, by the same corollary, if the transcendence degree of $k'$ is $2$ and $k'$ is of positive characteristic, the boundary is also trivial for centered finite second moment random walks.

In the example below, the group is defined over a field of transcendence degree $3$, while its blocks are two-dimensional lamplighters. This is among many examples where the criterion about the transcendence degree of the blocks implies the triviality of the boundary of the random walk. 

\begin{exa}[$X$ - $Y$ -$Z$ group]
\label{exa:defXYZ}
Consider the group generated by  matrices 
$$
\bar{M}_x=
\left( \begin{array}{ccc}
X & 0 & 0   \\
0 & 1 & 0  \\
0 & 0 & 1  \\
 \end{array} \right), 
\medspace \medspace
\bar{M}_y=
\left( \begin{array}{ccc}
1 & 0 & 0   \\
0 & Y & 0  \\
0 & 0 & 1  \\
\end{array} \right),
\medspace \medspace
\bar{M}_z=
\left( \begin{array}{ccc}
1 & 0 & 0   \\
0 & 1 & 0  \\
0 & 0 & Z  \\
\end{array} \right)
$$
and
$$
\theta_{1,1,1}=
\left( \begin{array}{ccc}
1 & 1 & 1   \\
0 & 1 & 1  \\
0 & 0 & 1  \\
\end{array} \right)
$$

Then this group has a trivial Poisson boundary for any simple random walk (or more generally any centered finite second moment random walk). 

\end{exa}

\section{Comparison Criterion for group extensions}
The main goal of this section is to prove the following theorem. 
\begin{thm}[Comparison Criterion: abelian case]\label{thm:compcrit}
Let $1\to F\to G \to H\to 1$
and $1\to F\to G_2\to H\to 1$
be short exact sequences where the group $F$ is abelian, and the induced action of $H$ on $F$ is the same for both extensions.  Let $\mu_1$ and $\mu_2$ be finite entropy  measures on $G$,$G_2$ which
have the same  projections to $H$. Let $f\in F$, assume that  $f$ acts non-trivially on the Poisson boundary of $(G,\mu_1)$. Then under the assumption that $\mu_2$ is non-degenerate,  $f$ acts non-trivially on the Poisson boundary of $(G_2,\mu_2)$. 
\end{thm}

\begin{rem}\label{rem:finiteentropyimportant}
The assumption that the entropy of the measures is finite is essential in the theorem, as well as in its Corollary \ref{cor:compcrit} below.
Indeed, consider $G=G_2=\mathbb{Z}/2\mathbb{Z} \wr \mathbb{Z}^3 $.
We have
$$
1 \to \sum_{\mathbb{Z}^3}\mathbb{Z}/2\mathbb{Z} \to G \to \mathbb{Z}^3 \to 1.
$$
The wreath product $G$ is solvable, and in particular amenable. As we have mentioned,  by a theorem of Rosenblatt 
\cite{rosenblatt1981} and Kaimanovich-Vershik \cite{kaimanovichvershik} we know that $G$ admits non-degenerate measures $\mu$ (which can be chosen to have full support $\supp \mu =G$), such that the boundary  $(G,\mu)$ is trivial. Consider a measure $\mu_2$, which has the same projection to
$G$ as $\mu$, with the support of $\mu_2$ being equal to $\mathbb{Z}^3 \cup \theta$.
The projection of $\mu_2$ (and $\mu$) to $\mathbb{Z}^3$ is a non-degenerate random walk on $\mathbb{Z}^3$, hence this random walk is transient. A well-known argument from \cite{kaimanovichvershik} implies that the value of the lamp at $0$ stabilizes to a non-trivial limit along the trajectories of the random walk $(G,\mu_2)$,  that the boundary of $(G,\mu_2)$ is thus non-trivial and also that $f=\theta$ acts non-trivially on this boundary.
\end{rem}

\begin{rem}\label{rem:nondegenerateimportant}
 The assumption of Theorem \ref{thm:compcrit}
 that the support of $\mu_2$ generates $G$ as a semi-group (the
 measure is non-degenerate) cannot be weakened to ask that it generates $G$ as a group. Consider two central extensions, one of which is a direct product:
$$
1 \to C \to G \to H \to 1,
$$
$$
1 \to C \to C\times H \to H \to 1.
$$
For any measure  $\mu_2$ on $C \times H$ its Poisson boundary is equal to the boundary of its projection to $H$, and any element of $C$ acts trivially on this boundary. However, there exist central extensions and (not non-degenerate) measures on them where elements of the center act non-trivially. Such examples can be chosen among Hall's group in such a way that the center is infinite \cite{erschlerkaimanovich2019}[Cor 5.14]. For earlier non-discrete examples with finite center see \cite{furstenberg63}, $SL(2, \mathbb{R})$ example discussed after Thm 5.3  and remarks after  \cite{erschlerkaimanovich2019}[Cor 5.14].
\end{rem}

As a consequence of the above comparison criterion for actions of an element, we get the following corollary relating the triviality of the boundary between two different measures:

\begin{cor}[Comparison criterion: triviality of the boundary]\label{cor:compcrit}
Let $1\to F\to G \to H\to 1$
and $1\to F\to G_2\to H\to 1$
be short exact sequences with $F$ abelian, and suppose that the actions of $H$ on $F$ for both sequences are the same.  Let $\mu$ and $\mu_2$ be finite entropy measures on $G$, $G_2$ which
have the same  projections to $H$,
and where $\mu_2$ is non-degenerate.  Then if the Poisson boundary of $(G,\mu)$ is non-trivial,  the boundary of  $(G,\mu_2)$ is also non-trivial. 
\end{cor}
\begin{proof}
To see that Corollary \ref{cor:compcrit} follows from Theorem \ref{thm:compcrit}, observe that if the projection of $\mu$ to $H$ has a non-trivial Poisson boundary, then the same is true for the projection of $\mu_2$ (this is the same projection), and then clearly the boundary of $(G,\mu_2)$ is non-trivial.
It is sufficient therefore to consider the case when  the random walk on $H$, defined by the projection of $\mu$ (and of $\mu_2$) has a trivial boundary. In this case note that the Poisson boundary of $\mu$ or $\mu_2$ is non-trivial if and only if there exists an element of $F$ that acts non-trivially on the boundary. (Indeed, if all elements of $F$ act trivially, observe
that a non-constant harmonic function on $(G,\mu)$ is $F$-invariant and induces a non-constant harmonic function on $H$).
Hence the claim of the corollary follows from the claim of the Theorem.
\end{proof}

It seems natural to ask whether
one can replace the assumption that the group $F$
is abelian by the assumption that this group is hyper-FC central.

\begin{rem}
 
Assume that $F$ is not hyper-FC-central and that $F$ admits finite entropy non-degenerate measures with trivial boundary.
Then the statement of Corollary \ref{cor:compcrit} is not true 
for the following extension:
$$
1\to F \to F \to 1 \to 1.
$$
In this case the assumption that
$\mu$ and $\mu_2$ have the same projection to the trivial group is obviously verified. But the claim of the Corollary is not true.
\end{rem}

Indeed, by a result  of \cite{4authors} we know  that any not hyper-FC-central group admits a finite entropy measure with non-trivial boundary.

We also mention that 
any amenable group admits
a non-degenerate measure with trivial boundary (\cites{rosenblatt1981, kaimanovichvershik}). 
This measure can not, however, necessarily be chosen to have finite entropy
(see \cite{erschlerliouv} and many new examples in this paper).




\begin{rem} Let $\mu$ be a measure on a countable group $G$. Consider $\mu'
= \alpha \mu + \beta \chi_e$, $\alpha+\beta=1$,
$\alpha>0$. Then the Poisson boundary of $(G,\mu)$ is trivial if and   only if the Poisson boundary of  $(G,\mu')$ is trivial.
Moreover, $g\in G$ acts trivially on the boundary  of $(G,\mu)$ if and only if $g$ acts trivially on 
the boundary of $(G,\mu')$
\end{rem}
This follows by noting that the space of harmonic functions for the two measures is the same. 

\begin{rem}\label{rem:comparemeasuresinZ3}
We consider a lamplighter group $G= \mathbb{Z}/2\mathbb{Z} \wr \mathbb{Z}^d $. Consider a finite entropy non-degenerate measure $\mu$ on $G$, $\mu(e) \ne 0$. 
Then there exists $\mu_2$ on $G$, such that  $\mu_2$  is supported on the union of $\mathbb{Z}^d$ and $\theta$ and which has the same projection to $\mathbb{Z}^d$ as $\mu$ and  such that
the following holds. For any such $\mu_2$ the action of $\theta$ on the boundary of $(G,\mu)$ is non-trivial if and only if the action of $\theta$ on the boundary of $(G,\mu_2)$ is non-trivial. In particular, $\theta$ acts non-trivially if and only if the projection to the base group is transient.
\end{rem}
This provides us in particular a new way to see why all finite entropy non-degenerate random walks on $G= \mathbb{Z}/2\mathbb{Z} \wr \mathbb{Z}^d$, $d\ge 3$
have non-trivial boundary (\cite{erschlerliouv}).
\begin{proof} 
We consider $\mu_2$ with the same projection to $\mathbb{Z}^d$ and with the support as we claim. By comparison criterion, the action of $\theta$ is either trivial in both cases or non-trivial. Finally, observe that for $\mu_2$ the original argument of Kaimanovich and Vershik can be used to show that the action of $\theta$
is non-trivial whenever the projected random walk is transient, and another argument from \cite{kaimanovichvershik} shows the triviality of the boundary in the case of recurrent projection.
\end{proof}

\subsection{ Proof of  Theorem
\ref{thm:compcrit}} 





Now we start  proving  the theorem.  We assume that  $f\in F$ acts non-trivially on the boundary of  $(G,\mu)$. We want to prove that it acts non-trivially
on the boundary of $(G_2, \mu_2)$.

First observe that Claim \ref{lem:strictinequality} implies

\begin{claim}\label{cor:technicalentropycriterion}
Suppose $f$ acts non-trivially on the Poisson boundary of $(G,\mu)$, where $\mu$ has finite entropy. Then for large enough $N$ it holds $$
H(\mu^{*N}, G/\Norm(f)) < N  h(G,\mu)$$ \label{basicentropy}
\end{claim}
\begin{proof}
If $N$ is large enough, then 
$H(\mu^{*N}, G/\Norm(f))/N$ is close to
$h (G/\Norm_G(f) , \mu)$, and by
Claim \ref{lem:strictinequality} we know that $h (G/\Norm_G(f) , \mu)< h(G,\mu)$.

\end{proof}

To prove the Theorem, we will apply the corollary above for $\mu$ and we need to show that the same inequality holds for some $N$ for our second measure $\mu_2$. This will be done by studying the $\Delta$-restriction entropy, defined below.





\begin{defn}[$\Delta$-restriction  entropy]\label{def:deltarestrictionentropy}
Given a group $G$, a probability measure $\mu$ on $G$ and
a finite set $\Delta \subset \supp (\mu)$,
we define the $\Delta$-restriction  entropy
$H_{\Delta}(n)$
as follows.
We consider an $n$-step trajectory of $(G,\mu)$.
Then we take the conditional entropy of $X_n$, conditioned on prescribing all increments except those that are in $\Delta$.
\end{defn}


For brevity, we will often say $\Delta$-entropy instead of  $\Delta$-restriction entropy. 
We note that a special case of  $\Delta$-entropy 
was used
in \cite{erschlerliouv} to provide lower estimates of the asymptotic entropy for wreath products and $3$-dimensional Baumslag groups.

For our applications, we will be interested in $\Delta$ such that its elements  belong to the same coset  $G/F$.

If $\mu$ is  some measure on a countable space, not necessarily a probability measure, we can also speak about the entropy of $\mu$ defined by
$$
H(\mu) = \sum - \mu(x) \ln \mu(x).
$$

\begin{lem}\label{lem:subadditivity}
For any group $G$, any  probability measure $\mu$, and any finite $\Delta \subset \supp (\mu)$
the function
$H_{\Delta}(n)$ is subadditive: 
$$
H_{\Delta}(n+m) \le H_{\Delta}(n)+ H_{\Delta}(m).
$$
In particular, there is a limit $H_{\Delta}(n)/n$ and if this function $\ge Cn$
for some positive $C$ and an infinite subsequence of $n$, then it holds for all $n$.
\end{lem}

\begin{proof}

Indeed, conditioning for $n$ first increments and for $m$ last increments is the same as our conditioning for $n+m$ increments.
Observe also that the conditional entropy of the convolution of two measures is not greater than the sum of the conditional entropies (all conditioned on the same event). The last two claims then follow from Fekete's lemma.

\end{proof}

\begin{claim} \label{lem:easyinequality}
Let $\mu$ be a finite entropy probability measure on a group $G$,
and $\Delta$ is a subset of the support of $\mu$.
Then the $\Delta$-entropy of the random walk $(G,\mu)$
is at most the entropy of the restriction $\nu$ of $\mu$ to $\Delta$. 
\end{claim}
\begin{proof}
Observe that 
$H_\Delta(1)$ is equal to $H(\nu)$.
By subadditivity of $\Delta$-entropy the claim follows.
\end{proof}

\begin{lem}\label{lem:easydirection}
Consider a probability measure $\mu$ on a group $G$, and 
a finite set $\Delta= \{g_1, \dots, g_k \} \subset  \supp (\mu)$.
Consider $f\in G$ and assume that the normal subgroup  $\Norm(f)$
generated by $f$ is abelian.
We also assume that all the $g_i$ have the same projection to $G/ Norm(f)$.
If the $\Delta$-restriction entropy is linear, then $f$ acts  non trivially
on the Poisson boundary of $(G,\mu)$.
\end{lem}

\begin{proof}
To prove this lemma, observe that the $\Delta$-restriction entropy
gives a lower bound for the conditional entropy of $(G,\mu)$, conditioned on the quotient in $G/\Norm(f)$.
Indeed, the conditional information of our $n$ step position from the definition of $\Delta$-restriction
entropy already entirely determines the projections to $G/\Norm(g)$.
In other words we are conditioning on more information in the definition of $\Delta$-restriction entropy than on the full $n$-step trajectory on the quotient group  $G/\Norm(f)$ which is in turn more information than the position at step $n$ of the induced random walk on  $G/\Norm(f)$.  
We can conclude that the asymptotic  entropy of $(G,\mu)$ is strictly greater than the asymptotic
entropy of the induced random walk $G/\Norm(f)$. Hence we can use
Claim \ref{lem:strictinequality}  and claim that $f$ acts non-trivially on the boundary of $(G,\mu)$.
\end{proof}

Lemma \ref{lem:efentropy} below shows that, under appropriate assumptions, the converse of the statement  in Lemma \ref{lem:easydirection} is true.

First we make two straightforward remarks about abelian normal subgroups, which we will use several times in the sequel.
\begin{claim} \label{claim:firstobvious}
If $F$ is an abelian normal subgroup of $G$, $f_i\in F$, then
$$
g_1f_1g_2f_2 \dots g_k f_k=
g_1 g_2\dots  g_k f'_1 \dots f'_k
$$
where $f'_i$ is the conjugate  of $f_i$ by $g_{i+1}g_{i+2}\dots  g_k$. That is $f_i^{'}=g_k^{-1}g_{k-1}^{-1}\dots g_{i+1}^{-1}f_ig_{i+1}\dots g_k$
\end{claim}
This also implies the following.
\begin{claim} \label{claim:secondobvious}
Consider a partition of $S_1$, \dots, $S_l$ of the numbers $1$, $2$, \dots , $k$.
Then 
$$
g_1f_1g_2f_2 \dots g_k f_k=g_1 g_2 \dots  g_k f''_1  f''_2 \dots f''_l
$$
$f''_i$  is a product of conjugates of $f_j$, $j \in S_i$.
\end{claim}

We now give a formula evaluating the $\Delta$-restriction entropy as the entropy of a random variable taking values in $F$.

\begin{lem}\label{lem:efentropy}[Main Lemma for comparison criterion]

Consider an extension 
$$
1 \to F \to G \to K \to 1.
$$
Let $f\in F$, $F$ be abelian. Let $\mu$ be a finite entropy probability measure on $G$ satisfying
$h(\mu) > H(G/\Norm(f), \mu)$.
(In other words, the asymptotic entropy of
$(G,\mu)$ is strictly larger than the 
entropy of the projection of $\mu$ to the normal subgroup generated by $f$.)
Then there exists a finite set $\Delta=\{g_1, \dots, g_k\}$,
such that $g_i$, $1 \le i \le k$ have the same projection to $G/\Norm(f)$
and such that $\Delta$-entropy of $(G,\mu)\ge Cn$, for some positive constant $C$.
\end{lem}

For the proof of Lemma \ref{lem:efentropy} above, we start with  the following observation:

\begin{lem} \label{lem:Deltadecomposition}
Let $\mu$ be a probability measure of $G$
of finite entropy.
Let $F$ be  an abelian normal subgroup of $G$. Then
$$
H(X_n) = H(\mu^{*n}) \le n H_{\mu_{G/F}} + \sum_{\alpha} H_{\Delta_\alpha}(n),
$$
where $\Delta_\alpha$ are subsets in the partition of elements in the support
of $\mu$ to their projection to $G/F$.
\end{lem}

\begin{proof}

For each $\alpha$, fix an element $g_\alpha \in \Delta_{\alpha}$.
Consider the elements $f_{\alpha,j} \subset F$, where $j$ takes values from $1$ to $|\Delta_{\alpha}|$, such that
the elements of $\Delta_\alpha$ are $g_\alpha f_{\alpha,j}$.  
When we write $X_n$ as a product of increments of the form $g_\alpha f_{\alpha,j}$, we see 
a product as in the left-hand side in the  
Claim \ref{claim:secondobvious}. Thus we can use this Claim, and write
$$
X_n = U_n \Pi_{j=1}^m   W^n_{\alpha}
$$
where $U_n$ is the product of the corresponding $g_i$. Here instead
of increments in $\Delta_\alpha$ we write $g_\alpha$, and $W^n_{\alpha}$ is a product of conjugates
of $f_{\alpha,j}$'s with the same value of $\alpha$. 

We denote by $Y_i$ increments of the trajectory  $X_n$.

For fixed $U_n$, we have that
\begin{equation}{\label{eq:sumwi}}
H(X_n|U_1, \dots, U_n)\leq \Sigma_\alpha H(W^n_\alpha|U_1,\dots ,U_n)
\end{equation}

Note that $H(W^n_{\alpha}|U_1,\dots,U_n)= H_{\Delta_\alpha}(X_n)$.

If we fixed  the conditioning from the definition of
$\Delta_\alpha$-entropy, then we have fixed $U_1, U_2, \dots, U_n$, and the distribution we get  is isomorphic to the distribution of $W_{\alpha,n}$, conditioned on $U_1,U_2,\dots,U_n$. Indeed, the conjugates of $f_{\alpha,j}$ are defined by $f_{\alpha,j}$ and $U_1,U_2,\dots,U_n$ (the conjugate
$yxy^{-1}$ of $x\in F$ only depends on the coset of $y$ in
$G/F$ and the value of $x$).

Now note that the entropy of the sequence ${U_1,U_2,\dots,U_n}$ is equal to $nH(\mu,G/F)$. Thus, adding this equality to  $\ref{eq:sumwi}$ and substituting the second half with $H(W_\alpha^n|U_1,\dots,U_n)= H_{\Delta_\alpha}(X_n)$ 
yields the desired inequality. 
\end{proof}

Now we prove Lemma \ref{lem:efentropy}.
\begin{proof}
We observe that Lemma  \ref{lem:Deltadecomposition} above and the assumption that $h(\mu,G)>H(\mu,G/F)$ imply that $\sum_\alpha H_{\Delta_\alpha}(n)$ grows linearly in $n$.
Therefore, if the index set is finite (and this is the case when the projection of $\mu$ to $G/F$ is finitely supported), then 
there exists $\alpha$ such that, for some subsequence of
$n$,
$\Sigma_\alpha$ grows linearly. Here we use the notation  $\Sigma_\alpha$ for the $\Delta_\alpha$-entropy, for the set $\Delta_\alpha$
described in  Lemma \ref{lem:Deltadecomposition}.
So we see that $\Delta_\alpha$ entropy is linear on some subsequence. By Lemma \ref{lem:subadditivity}
$\Delta_\alpha$ entropy is linear.

Now, even if the projection is infinitely supported,
we argue as follows. We write $\mu= \mu'+ \mu''$, where $\mu''$ has small entropy.
By Lemma \ref{lem:Deltadecomposition},
$$
\sum_\alpha H_{\Delta_\alpha}(X_n)
$$
is linear.

And, using Lemma \ref{lem:easyinequality}, we claim that  given $\varepsilon$, we  can choose a co-finite set $A_\varepsilon$, such that the sum
$$
\sum_{\alpha\in A_{\varepsilon}} H_{\Delta_\alpha}(X_n) \le \varepsilon n.
$$
Therefore, we can observe that
$\sum_{\alpha \notin  A_{\varepsilon}} H_{\Delta_\alpha}(X_n)
$ is linear. And thus for some $\alpha \notin A_\varepsilon$.
$\Delta_\alpha(X_n)$ entropy is linear on some subsequence.
Hence, again by Lemma \ref{lem:subadditivity}
$\Delta_\alpha$ entropy is linear.

Now we have found $\Delta \subset  \supp \mu$ such that the $\Delta$-entropy is linear. Observe that we can moreover claim the existence of a finite set $\Delta$ with this property. Indeed, we write $\Delta= \Delta' \cup \Delta''$, such that $\Delta'$ is finite, and the entropy of the non-normalized restriction $\mu"$ of $\mu$ on $\Delta''$
is smaller than $\varepsilon$.
Using again Lemma \ref{lem:easyinequality}, we observe that 
$$
H_\Delta (X_n) \le H_{\Delta'}(X_n)+n*H(\mu''). 
$$
Therefore, if $\varepsilon$ is small enough, we conclude that the $\Delta'$ entropy grows linearly.
\end{proof}

Before we continue the proof of the theorem,   we discuss the concept of colored increments. 


\begin{defn}[Colored $\Delta$ restriction entropy]
Given a subset $S$ of $G$ and $p$,  $0<p<1$ we say that the colored  $\Delta, p$ restriction entropy is the conditional entropy of $X_n$ conditioned on the following event. For each $n\in \mathbb{N}$ we choose independently and identically at random a Bernoulli random variable which is $0$ with probability $1-p$ and $1$ with probability p. We condition on all increments where the Bernoulli random variable is $0$ and on all increments which are not in $\Delta$. 
\end{defn}


We are specifically interested in the case when we have several elements $\Delta= \left( g_1, \dots, g_m \right)$, which have the same
projection to $G/F$, where $F$ is in an abelian normal subgroup of $G$, and we will usually take $p=1/k$, $k \in \mathbb{Z}$. In this case we can informally think about this as conditioning on all increments except those of one of $k$ possible colors (as well as all increments not in $\Delta$). 
Our main interest is to understand whether the $\Delta,p$-entropy is linear. 

\begin{rem}
The definition of colored $\Delta$-entropy and the question of its linearity can be reduced to the question about the linearity of the usual, not colored,  $\Delta$-entropy for a random walk on the product 
of $G$ with $\mathbb{Z}$, or some other groups with the Liouville property.
\end{rem}

\begin{lem}[Changing charges for $\Delta$-entropy]\label{lem:changingcharges}
Let $F$ be an abelian normal subgroup of $G$, $f\in F$. Let $\mu$ and $\mu_2$ be probability measures on $G$, and  let 
a finite subset $\Delta$ of $G$ belong to the support of both of them.
We also assume that $\Delta$ belongs to the same coset $G/F$, and that $\mu$ and $\mu_2$ have the same projection to
$G/F$.

If the $\Delta$-entropy of $\mu$ is linear, then the $\Delta$-entropy of $\mu_2$ is also linear.
\end{lem}

\begin{proof}

Choose a large enough $k$ such that for 
 each $g\in \Delta$
$$
\mu(g)/k < \mu_2(g).
$$
Consider the coloring of $\Delta$ increments for the random walk $(G,\mu)$ in $k$ colors.
Fix some element $g\in \Delta$ and
write elements of $\Delta$ as $gf_j$. We know that $f_j \in F$.
 
Use Claim \ref{claim:firstobvious}  to rewrite
the product of (first $n$) increments, putting all $f_j$'s on the right.
We get a product  $Z_n$ multiplied by conjugations of $f_j$.
More precisely we use Claim \ref{claim:secondobvious} to group the conjugations of $f_j$'s "of a given color" together (for each fixed $j$). For a color $i$ we denote by $\Sigma^n_{i,j}$
the product of conjugates of $f_j$ of color $i$.
Then the product of $n$ increments  is equal to
$$
Z_n \Pi_{i,j; 1\le i \le k} \Sigma^n_{i,j}
$$
It follows from the definition that the $\Delta$-entropy of $\mu$ is the entropy
of $\Pi_{1\le i \le k,j} \Sigma^n_{i,j} $. This is at most
$$
\Sigma_{i,j:1 \le i \le k} H(\Sigma^n_{i,j})
$$
Observe that the distribution of $\Sigma^n_{j,i}$ does not depend on $i$ (since all colors have the same role), and in particular the entropy 
of $\Sigma^n_{j,i}$ is equal to that of $\Sigma^n_{j,l}$ for all $i,l:  1 \le i, l \le k$.
Thus $\Delta$-entropy is at most
$$
k \Sigma_j   H (\Sigma_{j,1}^n).
$$

Since the  $\Delta$-entropy of $\mu$ is linear in $n$, we conclude that
there exists $j$ such that 
$H (\Sigma_{j,1}^b) \ge C/k n$, for $C$ not depending on $n$  (for some
subsequence of $n$'s and hence by subadditivity of colored $\Delta$-entropy,
for all $n$). 

Let us call the first color (among $k$ colors) "white". So we know that the $\Delta$-entropy of the white color is linear.

Now we look at our random walk $(G,\mu_2)$ and we want to claim that the
$\Delta$-entropy of the random walk $(G,\mu_2)$ is bounded from below by the white $\Delta$-entropy of $(G,\mu)$.

We color increments of the random walk $(G,\mu_2)$  in two colors, white and black. Each increment  $g \in \Delta \subset G_2$ we color in white with probability
$\mu(g)/k$ and in black with probability $ - \mu_2(g  - \mu(g)/k)$. 

We consider any coupling between trajectories $(G,\mu)$ and $(G,\mu_2)$,
where corresponding increments have the same projection on $G/F$ and white increments for $(G,\mu)$ are sent to white increments for $(G,\mu_2)$ with the same value in $\Delta$. 

Observe that the  $\Delta$-entropy after $n$ steps  of $(G,\mu_2)$ is greater than or equal to the white entropy of this random walk. (Since the white entropy is conditioned on strictly more information). 

Finally, observe that the white $\Delta$ entropy of $(G,\mu)$  is equal to
the white $\Delta$-entropy of $(G,\mu_2)$. By our assumption, the projections onto $G/F$ are equal for coupled trajectories. Hence, we use the fact that for our coupling both increments at some time $i$ (for $(G,\mu)$ and $(G,\mu_2)$ trajectories), are the same ${\rm mod} F$, the fact that $F$ is abelian and we use Claim  \ref{claim:secondobvious}. Indeed, when we bring  $\Delta$-increments of white color on the right using this remark, observe that
the product on the right, corresponding to
conjugations of these white elements,  depends only on the value of these increments and projection  to $G/F$ of the corresponding position of the random walk. 
\end{proof}

\begin{cor} \label{cor:changingchargesandgroups}
If $G$ and $G_2$ are extensions of an abelian group $F$ where $G/F$, $G_2/F'$ are isomorphic and they induce isomorphic conjugation actions on $F$ and $F'$.
Then if measures $\mu$ on $G$ and $\mu_2$ on $G_2$ have the same projection on
$K=G/F$, $\Delta \in \supp \mu$ and $\Delta' \in \supp \mu_2$, $\Delta = g \Omega$, $\Delta'\subset g' \Omega$ where
$g$ has the same projection to $K$ as $g'$, and $\Omega$
is some subset of $F$.
Then the $\Delta$-entropy of $(G,\mu)$ is linear if and only if the $\Delta'$ entropy
of $(G_2, \mu_2)$ is linear.
\end{cor}
\begin{proof}

Indeed, this follows from Lemma \ref{lem:changingcharges} in view of the following observation.
Under the assumption of our corollary, assume additionally
that $\mu(gw)= \mu(g'w)$
for any $w \in \Omega$. 

Observe that in this case the $\Delta$-entropy of $(G,\mu)$ is equal to
the $\Delta'$ entropy
of $(G_2,\mu_2)$
as follows from  \ref{claim:firstobvious}.
Knowing this, we can then change the values of $\mu(h)$,
 $h \in \Delta$ using 
 Lemma \ref{lem:changingcharges}.
\end{proof}

Now we are ready to prove Theorem \ref{thm:compcrit}.
Since $\mu_2$ is non-degenerate by  the 
assumption 
of the Theorem,
replacing the measures if necessary by affine combinations of their convolution powers, $\mu$  by $\sum_{i=1}^\infty a_i \mu^{*i}$
and analogously $\mu_2$ by
$\sum_{i=1}^\infty a_i \mu_2^{*i}$,
we can assume that 
$\supp \mu_2 =G_2$.
(Here we used the easy and well-known fact that the Poisson boundary is the same for $\mu$ and for a convex combination of its convolution powers \cite{kaimanovich83}.)

We assume that $f$ acts non-trivially
on the boundary of $(G,\mu)$. If necessary, we choose some $m$,
replace $\mu$ by $\mu^{*m}$ and $\mu_2$ by $\mu_2^{*m}$ so that we can assume 
\begin{equation} \label{eq:assumptionlemma}
h(G,\mu) > H (G/\Norm(f), \mu).
\end{equation}

In other words, we can assume that $\mu$ satisfies the technical  assumption of Lemma \ref{lem:efentropy}
(See Claim \ref{cor:technicalentropycriterion})
By Lemma \ref{lem:efentropy} we know therefore that for some $\Delta \subset G$,
$\Delta$ belonging to a coset of $G/F$, the $\Delta$-entropy is linear.
Write $\Delta = g \Omega$, $\Omega \subset F$.
Consider  $\Delta' = g' \Omega$,
$g$ has the  same projection on $K= G_2/F'$ as the projection of 
$g'$ to $K=G/F$.

Since the support of $\mu_2$ is  $G_2$, we know that $\Delta'$ belongs to the support of  $\mu_2$. We now use Corollary \ref{cor:changingchargesandgroups} to claim that $\Delta'$ entropy of $(G_2,\mu_2)$ is linear. 
Then we can use  Lemma \ref{lem:easydirection}
to claim that $f'$ acts non-trivially on
the boundary of $(G_2,\mu_2)$. Since $f'$ belongs to the normal subgroup generated by $f$,  we can conclude that $f$ acts non-trivially on the boundary of $(G_2,\mu_2)$,
and this concludes the proof of the comparison criterion.

\section{Nilpotent-by-Abelian groups}\label{section:nilpotentbyabelian}

In this section we study  nilpotent-by-abelian groups.  Given a nilpotent-by-abelian group $G$, we associate certain metabelian groups, which we call
metabelian components of $G$.
We will first check that triviality of the boundary is equivalent to triviality of the boundaries (for 
appropriate measures) on all metabelian components of $G$. Then for each metabelian component we define $p$-primary metabelian components, and reduce the problem to that of the corresponding measures
on (torsion $p$)-by-abelian   components as well as on the (torsion free)-by-abelian component. Torsion components 
can be in turn reduced to the case when the group is an extension of $p$-primary group by Abelian ones). Finally,  both the order $p$ extensions and the torsion free extension  case are reduced to associated blocks (some groups we associate to each such extension).

\begin{claim} \label{cl:products}
Given $G=A \times B$, and a measure $\mu$ on $G$ of finite entropy. 
The mapping from the boundary of $G$ to the product of boundaries of $A$ and $B$ is injective.
\end{claim}

\begin{proof}
By our assumption $H(\mu)$ is finite. Let $P(A)$ be the Poisson boundary of the projected random to $A$ ($(A,\mu)$) and $P(B)$ the Poisson boundary of the projection to $B$. By the conditional entropy criterion of Kaimanovich \ref{kaimanovich:conditional}, we know that the conditional entropy of $(A,\mu)$ conditioned on its boundary value along the infinite trajectory, is zero. And the same about $(B,\mu)$ and $P(B)$.
Now consider the conditional entropy of $(A \times B, \mu)$ conditioned on the limiting values in $P(A)$ and $P(B)$. We recall that the entropy of the measure on a product of two spaces $X$ and $Y$ is at most the sum of the entropies of the projection to $X$ and to $Y$. Applying this to  spaces with probability measures: the space $X$ of  trajectories of $(A,\mu)$ and the space $Y$ of trajectories of $(B,\mu)$,
we conclude 
that the conditional entropy is at most $0+0=0$.

\end{proof}

\begin{cor} \label{cor:products}
Let $G=A \times B$, $\mu$ on $G$ be a probability measure. 
Let $g=(g_a,g_b)$, $g\in G$, $g_a \in A$,
$g_b\in B$.
Then $g$ acts non-trivially on the Poisson boundary of $(G,\mu)$ if and only if  at least one of $g_a$ or $g_b$ acts non-trivially on the Poisson boundary of $(A,\mu)$ or $(B,\mu)$ respectively.
\end{cor}
\begin{proof}
First observe that if $g_a$ acts non-trivially on the boundary of $(A,\mu)$, then $g$ has the same action on this boundary, which is a quotient of the Poisson boundary of $G$.

Now we prove the other direction.
We know that  the Poisson boundary of $(G,\mu)$ injects into the Poisson boundary of $(A,\mu)$ times the Poisson boundary of $(B,\mu)$. Observe that  $g = (g_a, e) (e, g_b)$ acts non trivially on the boundary of $(G,\mu)$, then either $(g_a,e)$ or $(g_b, e)$ acts non-trivially on the boundary of $(G,\mu)$. Since $P(G)$ injects into $P(A)\times P(B)$, if for example $(g_a,e)$ acts non-trivially on the boundary, then it acts non trivially on $P(A) \times P(B)$, and it is clear that the action of such elements on $P(B)$ is trivial. Therefore, under this assumption $(g_a,e)$ acts non-trivially on $P(A) \times \{e\}$, and thus $g_a$ acts non-trivially on $P(A)$.
\end{proof}

In our reduction procedure
in this section
we will use several times a corollary from the Comparison Criterion (Thm \ref{thm:compcrit}), formulated in Lemma \ref{lem:mainforreduction} below.

\begin{lem}\label{lem:mainforreduction}[Reduction Lemma]
Let $G$ be a group, $H$ be a normal subgroup in $G$ and $A$  be a normal subgroup of $G$ in the center of $H$. Consider a short exact sequence
$1\to H\to G\to K\to 1$.
Let $C$ be the semi-direct product of $K$ with $A$, the action being the conjugation action of $G$ on $A$.
Consider a non-degenerate finite entropy probability measure  $\mu$
on $G$.
Let $\nu$ be a non-degenerate finite entropy   measure on $C$ with the same projection to $K$ as $\mu$. Then $a$ acts non-trivially on the boundary of $(C,\nu)$ if and only if it acts non-trivially on the boundary of $(G,\mu)$. 
\end{lem}
\begin{proof}
Let $\phi_1$ be the natural map from $C\to K$ and $\phi_2$ be the natural map from $G/A\to G/H=K$.
Consider the subgroup $W$ of $C\times G/A$ 
of elements $(c,x)$ such that $\phi_1(c)=\phi_2(x)$. 
Note that there is a natural injection from $A\to W$ coming from the natural injection from $A\to C$ on the first coordinate and trivial on the second coordinate.

$$
1 \to A \to W \to U \subset K \times G/A \to 1, 
$$
where $U = G/A$ consists of $(k, y)$, $y \in G/A$ such that the projection of $y$ to $K$ is $k$.
and we have
\begin{equation}\label{eq:exactse1}
1 \to A \to G \to G/A \to 1
\end{equation}
and
\begin{equation}\label{eq:exactse2}
1 \to A \to W \to G/A \to 1
\end{equation}

Consider a non-degenerate measure $\rho$ of finite entropy on $W$
which has the same projection to
$G/A$ (coming from the projection to the first coordinate of $U$) as $\mu$ (and thus as $\nu$).

We apply the comparison criterion
(Thm \ref{thm:compcrit}) to $(W, \rho)$
and $(G,\mu)$ (with respect to the short exact sequences in Equation \ref{eq:exactse1} and \ref{eq:exactse2}.
We conclude that an action of $a \in A$
on the boundary of $(G,\mu)$ is non-trivial if and only if its action on the boundary of $(W, \rho)$ is non-trivial.
We know that $W$ is a subgroup of $C \times G/A$, and we can consider measure $\rho$ as a measure on $C \times G/A$.
Now we apply Corollary \ref{cor:products} to $(C \times G/A,\rho)$, and claim that $a$
acts non-trivially on the boundary of 
$(C \times G/A,\rho)$ if and only if either $a$ acts non-trivially on the boundary of the  projected random walk on $C$ or on the boundary of the projected random walk on $G/A$. Clearly the second option cannot happen, so we know that $a$ acts non-trivially on the boundary of the projection of $\rho$  to $C$.
Now we apply once more the comparison criterion, this time for $(C, \rho)$
and $(C,\nu)$. We conclude that $a$  
acts non-trivially on the boundary of
$(C,\rho)$ if $a$ acts non-trivially on the boundary $(C,\nu)$.
\end{proof}

For our reduction procedure, we will also use several times the following not difficult property of actions on Poisson boundary:

\begin{lem}\label{lem:obsforreduction}
Let $G$ be a group. Let $\mu$ be a finite entropy measure on $G$. Let $g\in G$. Let $M$ be a normal subgroup of $G$. Then $g$ acts non-trivially on the boundary of  $(G,\mu)$ if and only if at least one of the following two conditions holds 
\begin{enumerate}
\item
 Some element of $M\cap \Norm_G (g)$ acts non-trivially on the boundary of $(G,\mu)$.
\item $g$ acts non-trivially on the boundary of $(G/M,\mu)$.
\end{enumerate}
\end{lem}

\begin{proof}

First observe that if (1) or (2) holds, then $g$ acts non-trivially on the boundary of $(G,\mu)$. Indeed, if (2) holds, then if an element acts non-trivially on the boundary of the quotient group, it also acts non-trivially on the boundary of our group.
It is also straightforward that if (1) holds, then there exists some element in $\Norm_G (g)$ that acts non-trivially, hence $g$
acts non-trivially.

Consider the map from $G\to G/M\times G/Norm_G (g)$. If (1) does not hold, then all elements of
$M \cup Norm_G (g)$ act trivially, and since $M \cup Norm_G (g)$ is the kernel of the map above, we can conclude, using Corollary \ref{cor:products}, that
if $g$ acts non-trivially on the boundary of $(G,\mu)$ then $g$ either acts non-trivially on the projected r.w. on $G/M$ or on
$G/\Norm_G (g)$. The last possibility clearly cannot happen, so in this case $g$ acts non-trivially on the boundary of $G/M$, and thus (2) holds.
\end{proof}

\subsection{Reduction to abelian-by-$K$ components}

Consider a nilpotent-by-Abelian group $G$ and a short exact sequence
$$
1 \to N \to G \to A \to 1,
$$
where $N$ is nilpotent of degree $d$. We will canonically associate to this short exact sequence 
certain metabelian groups. 

\begin{defn}[abelian-by-$K$ components]
Consider a short exact sequence
$$
1 \to N \to G \to K \to 1
$$
Assume that $N$ is hyper-central (for example nilpotent).
We define (upper) abelian-by-$K$ components of $G$ in the following way.
Consider the canonical central series of $N$: $C_{\alpha+1}$ is the center of $N/C_\alpha$.
Observe that $C_\alpha$ is a characteristic subgroup,
for all $\alpha$. We consider the  quotient group
$C_{\alpha+1}/C_\alpha$ and a short exact sequence
$$
1 \to N/C_\alpha \to G/C_{\alpha} \to K \to 1.
$$
Note that this sequence  induces an action of $K$ by conjugation on
$C_{\alpha+1}/C_\alpha$. We consider the  corresponding semi-direct  product $M_\alpha$. We call this group the
abelian-by-$K$ component of $G$.
\end{defn}

When we have a probability measure
$\mu$ on $G$, we associate to it (any) measure on $M_\alpha$ with the same projection to $K$. If $\mu$ is irreducible (or non-degenerate), then the projection to $K$ is clearly irreducible (respectively non-degenerate), and we can therefore choose an associated measure to be irreducible (or non-degenerate).

Note that if $K$ is abelian, then abelian-by-$K$ components are
metabelian. Observe also that if $H$ is nilpotent, we have associated in this case finitely many metabelian components.

\begin{lem}[Reduction to abelian-by-$K$ components] \label{lem:reductionabelianbyH}
Let $G$ be such that
$$
1 \to H \to G \to K \to 1,
$$
$H$ is hypercentral
and let $M^\alpha$ be  components
of $G$ corresponding to this short exact sequence.
Let $\mu$ be a finite entropy irreducible  measure on $G$. Choose some finite entropy irreducible  measure  $\mu^\alpha$ on
$M^\alpha$  which has the same projection to $K$ as $\mu$.
Consider $g\in H$.
Let $\alpha$ be the smallest ordinal such that $g \in M^\alpha$.
$g$ acts trivially on the Poisson boundary of $(G,\mu)$
if and only if the two following conditions hold.
\begin{enumerate}
\item every element $h$ of the normal subgroups in $G$ (generated by $g$) $h\in \Norm(f)$ which is also contained in $N_\beta$ where $\beta<\alpha$ acts trivially on the Poisson boundary of $(M^\beta, \mu^\beta)$. 

\item $g$ acts trivially on the Poisson boundary of $(M^\alpha, \mu^\alpha)$.
\end{enumerate}
\end{lem}
In particular, if the boundary of the projection to $K$ is trivial, then the
Poisson boundary of $(G,\mu)$ is trivial if and only if that of $(M^\alpha, \mu^\alpha)$ is trivial for all $\alpha$.

\begin{proof}
First apply Lemma \ref{lem:obsforreduction} for $G=G$, $g=g$ $M=C_{\alpha-1}$ where we suppose that $g \in C_\alpha$, and $\alpha$ is minimal with this property). Thus we see that $g$ acts non-trivially if and only if condition (1) fails or $g$ acts non-trivially on $G/C_{\alpha-1}$. Now applying Lemma \ref{lem:mainforreduction} to $G_2=G/C_{\alpha-1}$, $H_0=N/C_{\alpha-1}$ and $\A'=C_\alpha/C_{\alpha-1}$ implies that g acting non-trivially on $G/C_{\alpha-1}$ is the same as acting non-trivially on $M^{\alpha}$.
\end{proof}

Note that in the particular case  when $G$
is hypercentral-by-abelian, we have reduced the question of non-triviality of the action  of elements of $N$
on the boundary $(G,\mu)$ to the question about random walks on corresponding metabelian components.
And in a general case ($g$ is not necessarily in $N$) the reduction can be done using  Lemma \ref{lem:obsforreduction}.

As we have already mentioned, if $N$ is nilpotent, there are finitely many such components.

\begin{rem}
 In the case when $N$ is nilpotent,
 we can analogously define lower abelian-by-$K$
 components (metabelian if $K$ is abelian). And analogously, the non-triviality of the action of an element on the boundary of $(G,\mu)$ can be reduced to the same question about random walks on lower components.
\end{rem}

\begin{rem}\label{rem:notfinitelygeneratedcomponents}
Metabelian components $M_i$ (as well as upper metabelian components $M^i$)
are not necessarily finitely generated.
Consider for example $N$
to be  a free two-step nilpotent group on a countable set $n_i$, $i \in \mathbb{Z}$
and $A=\mathbb{Z}$ acts by shift of the index on the set $n_i$. 
Let $G$ be  the corresponding  extension. Then $G$ is finitely generated (and $N$ is finitely generated as a normal subgroup in $G$). But $[N,N]$
is freely generated by $[n_i, n_j]$, $i>j$.
And the action by $A$ preserves the difference $j-i$, so all $[n_i, n_0]$, $i\in \mathbb{Z}$ lie on distinct orbits and  are linearly independent in $G$. 
 So that the extension of $[N,N]$ by $A$ is not a finitely generated subgroup.
\end{rem}

\begin{rem}
As we mentioned, if $G$ is finitely generated, $M_i$ (and $M^i$) are not necessarily finitely generated. Thus, even if we study finitely supported measures on (nilpotent-by-abelian) 
$G$, to use the lemma above, we need to consider infinitely generated measures on $M_i$ (or $M^i$). There are other known situations where for the understanding of a Liouville property of a group $G$ one makes a reduction to some not finitely generated groups (and not finitely supported measures). As for example the argument of Kaimanovich and Vershik, that uses exit measure to some abelian group (e.g. $A \wr \mathbb{Z}$ or more generally $A \wr B$, $B$ is virtually $\mathbb{Z}^2$ or $\mathbb{Z}$. We also mention that Furstenberg-Lyons-Sullivan approximations \cites{lyonssullivan, kaimanovichFLS}, that reduce a question of the boundary on the covering, in particular for universal covers of compact manifolds, to a measure on the deck transformation group, have infinite support.
\end{rem}

In the following sections
we make
further reduction from components to blocks (in particular from metabelian 
components to metabelian blocks). We will later discuss sufficient conditions on $G$ (e.g. finitely generated linear group) which will guarantee that the reduction can be made for finitely many blocks

\subsection{Reduction to torsion-by-$K$ and torsion-free-by-$K$ groups}

Having a  group $G$
and a short exact sequence
$$
1 \to H \to G \to K \to 1,
$$
where $H$ is abelian,
we define {\it torsion-free by  $K$ abelian component} (in case when $K$ is abelian, these will be torsion-free-by-abelian component) $L$ of $G$ (with respect to this short exact sequence) to be $G/T$, where $T$ is the torsion subgroup of $H$ (this group is clearly an extension of $H/T$ by $K$).

We also define 
{\it torsion-by-$K$  component} $U$ of $G$ (with respect to this exact sequence) as follows: $U$ is a semi-direct product of $T$ by $K$.

\begin{lem}[Reduction to torsion-free-by -$K$ components and torsion-by-$K$ components. I: torsion case]
\label{lem:reductionI}
Consider a short exact sequence
$$
1 \to H \to G \to K \to 1,
$$ where $H$ is abelian and $T$ is the torsion subgroup of $H$. 
Let $\mu$ be a non-degenerate probability measure on $G$, and let $\mu_2$ be a non-degenerate associated  measure on the
torsion-by-$K$ component.
An element $g\in T$ acts non-trivially on the boundary of $(G,\mu)$ if and only if it acts non-trivially on the boundary of
the random walk defined by $\mu_2$.
\end{lem}
\begin{proof}
We apply Lemma \ref{lem:mainforreduction}
for $G=G$, $H=H$, $A=T$, $K=K$.
\end{proof}

\begin{lem}[Reduction to torsion-free-by-$K$ components and torsion-by-$K$ components. II: torsion-free  case]
\label{lem:reductionII}
$$
1 \to N \to G \to K \to 1,
$$
$T$ is a torsion subgroup of an abelian subgroup $N$,
$g \in N$.

Let $\Norm(g)$ be the normal subgroup generated by $g$ in $G$, and consider $C_g= \Norm(g) \cap T$.
$g$ acts non-trivially on the boundary of $(G,\mu)$
if and only if either $g$
acts non-trivially on the boundary of a torsion-free-by-$K$ component, or there is $c\in C_g$ that acts non-trivially on the boundary of a torsion-by-$K$ component.
\end{lem}
\begin{proof}
We apply  Lemma \ref{lem:obsforreduction}
for $G=G$, $M=T$.
\end{proof}

Consider a torsion-by-$K$ component $M$,
$$
 1 \to T \to M \to K \to 1,
$$ 
where $T$ is an abelian torsion group.
For a prime number $p$, let $T_p$ be $p$-part of $T$, that is
$$
T_p=\{t \in T: \mbox{the order of $t$ is a power of $p$}  \}
$$
It is easy to see that $T_p$ do not intersect and that $T$ is a direct sum
of $T_p$:
$$
T=\bigoplus_{p}T_p
$$
(for more on the structure of torsion Abelian groups see e.g. \cite{fuchs}[Ch.10].)
Let $M_p$ be an extension of $T_p$ by $K$. We say that $M_p$ is
{\it $p$-torsion-by-$K$ component of $M$}.
Given an element $t\in T$, we write it in the form
$t= \bigoplus_{\mbox{$p$ is prime}} (t_p)$.
Observe that $T_p$ is a normal subgroup in $M^p$.

\begin{claim} Let $M$ be a torsion-by-$K$ group,  of the form
$$
 1 \to T \to M \to K \to 1,
$$ 
where $T$ is an abelian torsion group.
Then   $t\in T$ acts non-trivially on the boundary of a torsion-by-$K$ component
  if and only if 
$t= \bigoplus_p (t_p)$ and there exists a prime $p$ such that $t_p$ acts non-trivially on the boundary of $(M^p, \mu^p)$ (where $\mu^p$ is a measure with the same projection to $K$ as $\mu$).
\end{claim}

\begin{proof}
Indeed, for each $t\in T$ we know that there exists a finite set of primes $\Omega_t$, such that 
$$
t= \bigoplus_{p \in \Omega_t} (t_p)
$$

Since each $t_p$ is a multiple of $t$, it is easy to see that $t$ acts non-trivially if and only if at least one $t_p$ acts non-trivially. Thus it suffices to check the claim when an element $t_p$ acts non-trivially.

Let $A$ be the set of elements in $T$ of order coprime to $p$ (that is, $A=
T=\bigoplus_{q}T_q$ for primes $q$ not equal to $p$) .  Applying lemma \ref{lem:mainforreduction} for $G=G$, $T=H$, $A=A$, and $K=K$ yields the desired result. 

\end{proof}

So we already made a reduction to torsion-free-by $K$ and $p$-torsion-by-$K$ components.

These ($p$-torsion)-by-$K$ components $M_p$ can be reduced to a  collection of extensions of
elementary abelian $p$-groups (
abelian groups with all elements of order $p$) by $K$.

\begin{defn}[(Elementary-$p$)-by-$K$ components.]
Consider a short exact sequence
$
1 \to T_p \to M^p \to K \to 1,
$
where $T_p$ consists of  elements of $M^p$ which are of order which is a power of $p$.
Consider a (finite or countable) family of groups $T^{j,p}$, for each $j$ the group $T^{j,p}$ is elementary $p$-group  defined as $W^{j+1,p}/W^{j,p}$, where
$W^{j,p}$ are elements of order
$p^j$. Put $M^{j,p}$ to be an extension  
of $T^{j,p}$ by $K$. 
We call these groups  {\it  elementary-$p$-by-$K$}
components of $M_p$.
\end{defn}

\begin{figure}[!htb] 
\centering
\includegraphics[scale=.95]{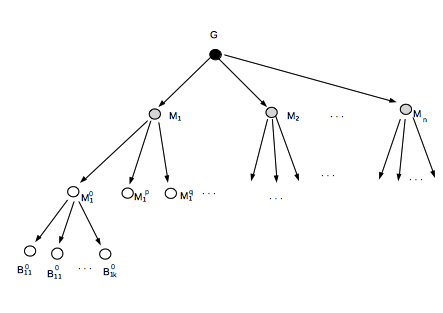}
\caption{Block decomposition of a f.g. nilpotent-by-abelian group. First line of the decomposition are Metabelian
components of $G$. Next line are primary Metabelian components (extensions of $p$ groups or torsion-free groups. Last line are blocks: a collection
of f.g. two-times-two upper triangular  matrices
}
\label{pic:blockdecomposition}
\end{figure}

\begin{lem}[Reduction of extension by $p$-groups to extensions by elementary abelian
$p$-groups] \label{lem:reductiontoElementary}
We consider a short exact sequence
$$
1 \to T_p \to M^p \to K \to 1,
$$
where $T_p$ is $p$-torsion. Let $\mu_p$
be a non-degenerate measure on $M_p$.
We claim  $g\in T_p$ acts non-trivially if and only if 
there exists $h\in \Norm(M^p) g$,
$h \in T^{j+1,p}$, $h \notin T^{j,p}$, and $h$ acts non-trivially on an associated random walk on some primary-by-$K$
component $M^{j,p}$.
\end{lem}
\begin{proof}
We proceed by induction on $n$ where $g$ has order $p^n$. In the case $n=0$, $g=e$ the claim is trivial.  
Suppose we have proven the statement for all elements $g$ of order $p^{n-1}$. Take an element $g$ of order $p^n$. Consider $U=\Norm_{M^p} (g) \cap T^{p^{n-1}}$. 
Applying Lemma \ref{lem:obsforreduction}
for $G=M^p$, $M=T^{p^{n-1}}$, we conclude that $g$ acts non-trivially on the boundary of $(M^p, \mu^p)$
if and only if either some element of $U$ acts non-trivially on the Poisson boundary of $(M^p, \mu^p)$ or $g$ acts non-trivially on the boundary of the induced random walk on $M^p/T^{p^{n-1}}$.

By induction step we now know that  $U$ acts non-trivially if and only if 
$g$ acts non-trivially on the boundary of one of the components.

Now suppose that $g$ acts non-trivially on the boundary of $M^p/T^{p^{n-1}}$.
This shows that we can assume that $n=1$ (because $g$ is of order $p$ in this quotient and since "extension of elementary $p$" components of the quotient are also components of our group).

We apply Lemma \ref{lem:mainforreduction} for
$G_2= M^p/T^{p^{n-1}}$,  $H=\tilde{T}^p$ and
$A=T^{p^n}/T^{p^{n-1}}$ and conclude that
$g$ acts non-trivially on $M^p/T^{p^{n-1}}$
if and only if it acts non-trivially on 
"extension of elementary $p$" components
of $G_2= M^p/T^{p^{n-1}}$.

Finally, observe that if $h \in \Norm_{M^p}(g) $ acts non-trivially on
the boundary of one of the elementary components $M^{j,p}$ (which is an extension of $W^{j+1, p}/W^{j,p}$ by $K$) of $M_p$, then we can apply
Lemma \ref{lem:mainforreduction}
for $G_2=M^p$, $H=\tilde{T^p}$ and $A=T^{p^j}/T^{p^{j-1}}$ and conclude that if $h$ acts non-trivially on the component $h$ then $h$ (and therefore $g$) acts non-trivially on the group. 

 \end{proof}

\begin{defn}\label{defn:singlegeneratedcomp}(Single generated component).
Consider a short exact sequence
$$
1 \to Q \to G \to K \to 1,
$$
where $Q$ is abelian.
Given an element $q\in Q$ define the single generated by $K$ component $G_q$ to be the semi-direct product defined by the short exact sequence 
$$
1 \to \Norm_G (q) \to G_q \to K \to 1
$$
\end{defn}
If $Q$ is either torsion-free abelian or
an elementary $p$ group,
then it is clear that $\Norm_G (q)$ in the definition above is either generated by a single torsion free element, or by a single
element of order $p$.

\begin{lem}[Reduction to single generator].
\label{lem:reductiontosingle}
Consider an extension
$$
1 \to Q \to G \to K \to 1.
$$
Take a finite entropy non-degenerate measure on $G$.
An element $q$ in the center of $Q$ acts non-trivially
on the boundary of $(G,\mu)$ if and only if it acts non-trivially on the boundary of $(G_q, \mu_q)$, where we chose some/all measures non-degenerate  measure $\mu_q$ on $G_q$ to have the same projection to $K$  as $\mu$.
We remind that the group $G_q$ is defined in Definition
\ref{defn:singlegeneratedcomp}.
\end{lem}
\begin{proof}
This follows from Lemma \ref{lem:mainforreduction} applied to
$G=G$, $H=Q$ and $A= \Norm_G (q)$.
Observe that every conjugate of $q$ in $G$ is also in the center, so $A$ is also in the center of $Q$, and thus $A$ satisfies the assumption of the lemma
that $A$ is central in $Q$.
\end{proof}

\subsection{Reduction to blocks}

Combining the reduction lemmas (reduction to abelian-by-$K$ components (Lemma\ref{lem:reductionabelianbyH}), reduction to primary-by-$K$ and torsion-free metabelian components \ref{lem:reductionI} and \ref{lem:reductionII}, reduction to (elementary $p$)-by-$K$ \ref{lem:reductiontoElementary}  and reduction to single) \ref{lem:reductiontosingle}, we obtain

\begin{thm}\label{thm:generalextensions}
Let $1\to N\to G\to K\to 1$ be a short exact sequence where the group $N$ is hypercentral. Let $\mu$ be a non-degenerate finite entropy measure on $G$.
Then there exists a countable family of groups
$C_\alpha$, $1\to A\to C_{\alpha}\to K\to 1$,
where $A$ is abelian,  normally generated by a single element, and  either torsion-free or an elementary $p$-group such that
the following holds. For each $C_\alpha$ fix a finite entropy non-degenerate measure $\mu_\alpha$ with the same projection to $K$ as $\mu$.
\begin{enumerate}
\item $(G,\mu)$ has non-trivial boundary if any only if there exists $\alpha$ such that  $(C_\alpha,\mu_\alpha)$ has non-trivially boundary.


\item
For each $g \in G$ we can associate  $g_\alpha \in C_\alpha$ (not depending on $\mu$)
such that 
$g\in G$ acts non-trivially on the boundary 
$(G,\mu)$ if and only either $g$ acts non-trivially on the projected random walk to $K$ or
if there exists
$\alpha$ such that $g_\alpha$ acts non-trivially on the boundary of $(C_\alpha,\mu_\alpha)$.
\end{enumerate}
\end{thm}

\begin{rem}  If we follow the reduction procedure explained in this section, we can describe elements $g_\alpha$. In the next section, in the case of linear groups, we explain  more explicitly how blocks $C_\alpha$ and $g_\alpha$ can be chosen.
\end{rem}

Now we assume that $K$ is finitely generated and  abelian and that $N$ is nilpotent.
It is well-known that the exit measure to
a finite index subgroup
has the same Poisson boundary(see \cite{furstenberg1970}, Lemma 4.2). If the group is finitely generated, then,
taking the exit measure on an appropriate subgroup, we can assume that $K=\mathbb{Z}^d$. 

In this situation Theorem \ref{thm:generalextensions} 
tells us that we get groups $M$
$$
1 \to N \to M \to K =\mathbb{Z}^d\to 1,
$$
$N$ is generated by a single element, 
either of order $p$ or torsion-free, as a normal subgroup of the abelian subgroup $K$.
In this case $M$ is finitely generated.
This group $M$ is either a torsion-free quotient of $\mathbb{Z} \wr \mathbb{Z}^d$
or a quotient  of $ \mathbb{Z}/p\mathbb{Z}
\wr \mathbb{Z}^d$.

\begin{rem}  As we have already mentioned, the groups above
can be represented by upper-triangular matrices  over a field of characteristic zero in the (finitely generated) torsion-free metabelian case and over a field of characteristic $p$ in ($p$-torsion)-by-abelian case \cite{remeslennikov69}.
\end{rem}

\begin{cor}\label{thm:nilpotentbyAbelian} \label{thm:abelianbynilpotent}
Assume that  $G$ is a finitely generated  nilpotent-by-abelian group. 
There exists a
 collection of groups $B_{\alpha}$ with a fixed surjection to $A$ such that the following holds.
If $\mu$ is a probability measure on $M$ and $\mu_\alpha$ are some probability measures  on $B_\alpha$
with the same projection to $A$ as $\mu$, then the boundary of $(G,\mu)$ is non-trivial if and only
if there exists $\alpha$ such that the boundary of $(B_\alpha, \mu_\alpha)$ is non-trivial.
This claim is for some/equivalently all choices 
of measures $\mu_\alpha$ on $B_\alpha$
(with the same projection to $A$ as $\mu$).
\end{cor}

\begin{rem} This collection in the corollary is not necessarily finite. This can be seen already for (torsion two step nilpotent)-by-$\mathbb{Z}^d$ groups (even for $d=1$).

\end{rem}

\section{Linear 
groups} \label{sec:moreonreductiontoblocks}

In the previous section, we established a reduction from nilpotent-by-abelian groups to metabelian ones.
In the case of linear upper-triangular groups,  
the collection of blocks can be chosen to be finite, and the reduction can be made in a more explicit way. This is the main result of this section.

As we have mentioned in the introduction, by Malcev's theorem any solvable linear group contains a finite index subgroup which is isomorphic to a group consisting of upper-triangular matrices.
And that a linear group is amenable if and only if this group is virtually solvable (as  follows from Tits alternative). 


We consider a partial order $U$ on $\{i,j: 1 \le i < j \le n \}$, defined as

For pairs $(i,j)$ where $1 \le i < j\le n$ we consider the following partial order $U$: $(i,j) \leq_U (i',j')$ if $i \leq i'$ and $j \geq j'$. See  Picture
\ref{pic:partialorderU}.

\begin{figure}[!htb] 
\centering
\includegraphics[scale=.65]{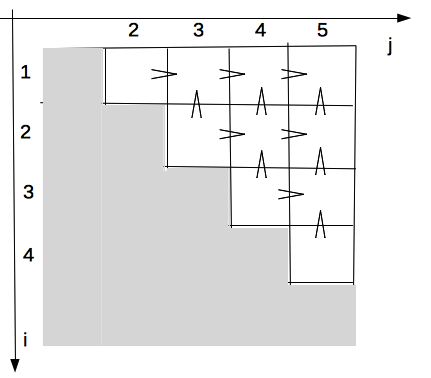}
\caption{Partial order $U$ on pairs $(i,j): 1\le i < j \le n$.
$(i,j) \leq_U (i',j')$ if $i \leq i'$ and $j \geq j'$; e.g. $(1,5)\leq_U (2,4)$
}
\label{pic:partialorderU}
\end{figure}

In the introduction, we gave a definition of basic blocks and valid basic blocks. Now we can formulate the same definition using the terminology of partial orders.
In the definition below we consider the partial order $U$. This point of view is useful having in mind  the third claim of Theorem 6.4, 
where various partial orders are considered.

\begin{defn} \label{def:blocks}[Basic blocks]
Consider a group of  $n\times n$ upper triangular matrices.
Consider matrices of the form
$$
G_{i,j}=
\left( \begin{array}{ccc}
g_{i,i} & 0  \\
0 &  g_{j,j}
 \end{array} \right), 
$$
defined when  in $G$ there is a matrix 
with 
the same diagonal
entries ($g_{i,i}$ and $g_{j,j}$).
If there are no unipotent elements in 
$G$ with non-zero entry at $(i,j)$ and with zero entries in all positions  $(i',j)'$ such that $(i',j)' \ne (i,j)$ and
$(i', j')\leq_U (i,j)$, then we say that $(i,j)$-block of $G$ is trivial (a group consisting of an identity element).
Otherwise, we consider the group generated
by $G_{i,j}$
and by the  matrix $\theta$
%
%

and we call this group the $(i,j)$-basic block of $G$. We denote this subgroup by  $B_{i,j}$.
\end{defn}


\begin{rem}
A close  version of the definition above could also be used: we could also say that a block $\tilde{B}_{i,j}$ is a subgroup of a linear group generated by
$$
\tilde{G}_{i,j}=
\left( \begin{array}{ccc}

g_{i,i} g_{j,j}^{-1} & 0  \\

0 &  1
 \end{array} \right)
$$
and $\theta$.
Observe that $\tilde{B}_{i,j}$ is a quotient of $B_{i,j}$ over a central subgroup (of matrices having a constant on the diagonal), so that
the Poisson boundary of a non-degenerate random walk on
$B_{i,j}$ has a non-trivial boundary if and only if it happens for its projection to $\tilde{B}_{i,j}$.
\end{rem}

Given a a partial (or complete)
order  $T$ on pairs $(i,j)$, $i<j$
a group of upper triangular matrices $G$ has an $(i,j)$-block which is valid with respect to the partial order $T$ if there exists an upper uni-triangular element in $G$, with matrix $m_{x,y}$ with $m_{i,j} 
\ne 0$ and zero in all $i',j' <_T (i,j)$.

Note that any non-trivial block of $G$
is valid with respect to $U$.

Given a probability measure $\mu$ on $G$, on any valid block $B_{i,j}$ we consider any nondegenerate finite entropy measure
$\mu_{i,j}$ with the same projection to the two diagonal entries $i,i$ and $j,j$. As we have mentioned in the introduction, we call any such measure an {\it associated}  measure on this block.

\begin{rem}\label{rem:blockscandepend}
A family of basic blocks of a given group $G$ can depend on the matrix realization of $G$. Assume that $G$ is a subgroup of two-by-two matrices for, and then realize it as a subgroup of $4\times 4$ matrices, of the following form
$$
\left(\begin{array}{cccc}
* & * & 0 & 0\\
0 & * & 0 & 0\\
0 & 0 & * & *\\
0 & 0 & 0 & *\\
 \end{array}\right)
$$ 
Here the two-by-two matrix groups in the upper left corner is isomorphic to a quotient of $G$. This can happen for example if $G= \mathbb{Z}/p \mathbb{Z} \wr \mathbb{Z}^d$ and its quotient is
$\ \mathbb{Z}/p \mathbb{Z} \wr \mathbb{Z}^j
$, for some $1\le j\le d$.
Observe that there exist two valid blocks of $G$ (defined for $i=1$, $j=2$ and for $i=3$, $j=4$), and that they are equal to
$G= \mathbb{Z}/p \mathbb{Z} \wr \mathbb{Z}^d $ and 
$\mathbb{Z}/p \mathbb{Z} \wr \mathbb{Z}^j $.
\end{rem}

\begin{thm} \label{thm:lineargroups}

\begin{enumerate}

\item The Poisson boundary of $(G,\mu)$
is non-trivial if and only if there exist $i,j: 1 \le i<j \le n$ such that for  every associated measures $\mu_{i,j}$ on the $(i,j)$-block $B_{i,j}$ (equivalently: for some associated measure on this block)
the Poisson boundary of $(B_{i,j}, \mu_{i,j)}$ is non-trivial.

\item Moreover, we can characterise the elements that act trivially on the Poisson boundary.
An element $g\in G$ acts non-trivially on the
boundary $(G,\mu)$ if (and only if) there exist
$i,j: 1 \le i<j \le n$ and
$h \in \Norm_G (g) \cap UT_n$ such that
the matrix corresponding to $h$ has
$m_{i,j} \ne 0$ and $m_{i',j'}=0$
for any $(i',j')>_U(i,j)$ and any
associated measure on the block $B_{i,j}$ has a nontrivial boundary.

\item Moreover, fix a partial order $T$ on $\{i,j: 1 \le i <j \le n  \}$, which extends the partial order $U$.
An element $g\in G$ acts non-trivially on the
boundary $(G,\mu)$ if (and only if) there exist
$i,j: 1 \le i<j \le n$
$h \in \Norm_G (g) \cap UT_n$ such that
the matrix corresponding to $h$ has
$m_{i,j} \ne 1$ and $m_{i',j'}=0$
for any $(i',j')>_T (i,j)$ and the associated measure on the block $B_{i,j}$ has a nontrivial boundary. 
\end{enumerate}
\end{thm}


\begin{rem}\label{rem:orderYoung}
Take a complete order $T$, which 
respects the partial order $U$.
Recall that such complete order 
corresponds to Young Tableaux.
See picture for an example of a Young tableau corresponding to the horizontal order. Take note that our picture is a flipped version (with respect to the vertical axis) of the standard convention for Young
Tableaux.
\end{rem}

\begin{figure}[!htb] 
\centering
\includegraphics[scale=.65]{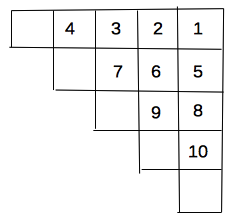}
\caption{Young tableau
}
\label{pic:youngtableau}
\end{figure}

Let $T$ be a partial  order
on pairs $(i,j)$, $1 \le i <j \le n$.
Observe that $T$ defines a  preorder on elements of $UT_n$, saying that $u \le_T w$ if for every non-zero coordinate $(i,j)$ of $u$ there exists $(i',j')$ such that $(i',j')\geq_T (i,j)$ and the $(i',j')$ coordinate of $w$ is non-zero. 

Denote by  
$N^T_{i,j}$ the set of $UT_n$ consisting of all upper-unitriangular elements which are $0$ in all coordinates $i',j'$ where $(i',j')\ge_T (i,j)$ with respect to $T$. By $\theta_{i,j}$ we denote the matrix with $1$ at the entry $(i,j)$ and zero elsewhere.
For the convenience of the reader, we formulate and explain some elementary properties of $N^T_{i,j}$.

\begin{claim}
 \label{claim:normalmatrix}
 Let $T$ be a partial order 
 order which extends  $U$.
\begin{enumerate} 
 
\item $N^T_{i,j}$ is a normal subgroup of the upper triangular matrices $\rm{Tr}_n$.
 
\item Moreover, if we conjugate $I+\theta_{i,j}$ by $u\in UT_n$ we get $I+\theta_{i,j}-k$ where $k \in N^T_{i,j}$.

\item Moreover let $h$ be an element in $U_n$ which is $0$ for all elements $(k,l)>_T (i,j)$. Then the image of $h$ (under the quotient map) is central in $U_n/N^T{i,j}$.

\end{enumerate}
\end{claim}

\begin{proof}

1) Consider upper-triangular matrices $A$ with $0$ on the main diagonal. We recall such matrices form a nilpotent group. Consider a map $t$ that maps a nilpotent uppertriangular matrix $A$ to the pair $(k,l)$, $1 \le k,l \le n$, where $(k,l)$ is the maximal  element, with respect to the order $T$,  with $a_{k,l} \ne 0$. We also use the convention $t(0)<_T (k,l)$ for any $k,l$, $1\le k, l \le n$.
Note that given two non-zero nilpotent  upper triangular  matrices $A$ and $B$
$$
t(AB)<_T t(A),t(B). (*)
$$
Also note that
$$
t(A+B)\leq_T Max_T(t(A),t(B).  (+)
$$

First we prove that $N^T_{i,j}$ is a subgroup. Note that by the definition $N^T_{i,j}$ is the subset consisting of elements of the form $I+A$ where $A$ is a nilpotent upper triangular matrix such that $t(A)<_T(i,j)$.

Also observe that $g^{-1}$
belongs to $N^T_{i,j}$ if $g\in N^T_{i,j}$. Indeed $(1+A)^{-1}=1-A+A^2-A^3\dots=1+B$ Since $B$ is a sum of powers of $A$ and by (*) we know that $t(A^n)<_T t(A)$, thus we conclude by (+) that 
$t((1+A)^{-1}) \leq_T t(A)$.
So $N^T_{i,j}$ is a subgroup.

Now we want to prove that the subgroup is normal in the upper-triangular matrices. First note that given an upper triangular matrix $A$ and an invertible diagonal $D$, we have
that $DAD^-1$ has zero or non-zero entries at the same positions as $A$, and hence
$t(DAD^-1)=t(A)$. So it suffices to prove that $N^T_{i,j}$ is preserved by conjugation by uni-upper triangular matrices. Let us conjugate $1+A$ by $1+B$. Let $(1+B)^{-1}=1+C$. Then $(1+B)(1+A)(1+C)=(1+B+C+BC)+(A+BA+AC+BAC)=(1+B)(1+C)+(A+BA+AC+BAC)=1+A+BA+AC+BAC$ where the last equality holds because $1+B$ and $1+C $ are inverses. Applying (+) and (*) to the terms of  $1+A+BA+AC+BAC$, we see that this element belongs to the subgroup, and the normality follows. 

2) Let $(1+A)$ and $(1+B)$ be inverse uni-upper triangular matrices.  Then $(1+A)\theta_{i,j}(1+B)=\theta_{i,j}+A\theta_{i,j}+\theta_{i,j}B+A\theta_{i,j}B$. By (*) $t(A\theta_{i,j}+\theta_{i,j}B+A\theta_{i,j}B)<_T t(\theta_{i,j})=(i,j)$, which implies the claim.  
 
3) First we prove the following: Let $A$ and $C$ be two nilpotent upper triangular matrices such that $t(A-C)<_T{i,j}$ then $(1+A)(1+C)^{-1}\in N^T_{i,j}$. To see this, rewrite it as $(1+C+(A-C))(1+C)^{-1}$. This is equal to $1+(A-C)(1+C)^{-1}$. Since $t(A-C)<_T (i,j)$, we can use * and conclude that 
$$
t((A-C)(1+C)^{-1})<_T (i,j),
$$
which proves the above-mentioned claim.  

Now let $h=I+A$.  Take $I+B \in U_n$ we want to show that
$(I+A)$ and $I+B$ commute modulo $N^T_{i,j}$.
By the above claim it suffices to prove that $t((I+A)(I+B)-(I+B)(I+A))<_T (i,j)$. However, since $(I+A)(I+B)-(I+B)(I+A)=AB-BA$, so (*) and (+) immediately imply that $AB-BA<_T (i,j)$.

\end{proof}

Now we prove the theorem.
\begin{proof}

(2) follows from (3) taking $T=U$.
And (2) implies (1) since the boundary is trivial if and only if all elements act trivially.

So to prove the theorem, we have to prove Claim (3).
First, assume that $T$ is complete.
Claim \ref{claim:normalmatrix}
states that $N^T_{i,j} \cap G$ is a normal subgroup of $G$, for any total order $T$,
compatible with $U$. 

Suppose that $g$ acts non-trivially on the boundary of $G$. Consider elements  $h$  in
\newline
$\Norm_G (g) \cap UT_n$ that act non-trivially on the boundary of $(G,\mu)$.
For any element $h$ in $\Norm_G (g) \cap U_n$, take the minimal
$(i,j)$ with respect to $T$ such that the entries of $h$ are zero in all $(i', j') >_T (i,j)$. We use notation $i(h)=i$, $j(h)=j$.
Among $i(h),j(h)$ such that
$h$ acts non-trivially on the boundary, we take the minimal one
which we denote by $(i,j)$. We consider one of such $h$, with
 $i(h) =i$, $j(h)=j$. Consider the block $B_{i,j}$. We want to  show that the boundary of an associated measure on  this block is nontrivial. 
 
 By our assumption that $(i,j)$ is minimal with the above-mentioned property, we know that the intersection  $\Norm_G (g) \cap N^T_{i,j}$ and thus $\Norm_G (h) \cap N^T_{i,j}$  acts trivially on the Poisson boundary. Thus, by Lemma \ref{lem:obsforreduction}
 applied to $G=G$ and $M=N^T_{i,j}$,
 we see that $h$ acts non-trivially on the Poisson boundary of $G/N^T_{i,j}$. 
Using the last part of the Claim \ref{claim:normalmatrix}, we see that $h$ is in the center of $Q'=G/N^T_{i,j}\cap U_n/N^T_{i,j}$.  
 
 Apply Lemma \ref{lem:reductiontosingle} to the short exact sequence $1\to
 Q' \to G_2=G/N^T_{i,j}\to D \to 1$ and the image $q$ of $h$ in $Q'$.  
($D$ is the diagonal group).
 Denote by $(G_2)_{q}$ the group obtained in Lemma \ref{lem:reductiontosingle}. Note that the subgroup  $D_{i,j}$ of the diagonal group which is $1$ on coordinates $i,i$ and $j,j$ acts trivially on $(G_2)_q$.  
Observe that $D_{i,j}$ is a central subgroup of $(G_2)_q$. We take the quotient of $(G_2)_q$ over this subgroup. Observe that  the map from
this quotient that sends $q$ to $\theta$ in $B_{i,j}$ and sends $D/D_{i,j}$ to the diagonal elements  $G_{i,j}$
in the definition of the block (see Definition \ref{def:blocks})
induces an isomorphism.

Since central extensions have the same boundary for non-degenerate measures,  we can  conclude that the boundary of (any) associated measure on the 
block $B_{i,j}$ is non-trivial.

Thus we have proved the last claim of the theorem under the assumption that the order $T$ is complete.

Until now we have assumed that $T$ is a complete order.
Now consider $T=U$. 
Choose a complete order $T'$ which respects $T=U$.
Observe that a valid block with respect to a $T'$ is also valid with respect to $U$.
By the claim applied to $T'$ we know that  if the boundary of our random walk on $G$ is non-trivial, there exists a valid block with respect to $T'$, thus a valid block with respect to $U$,
such that any associated measure on this block has non-trivial boundary.

Finally, take any partial order $T$. If we have a valid block with respect to $T$ which has non-trivial boundary, then since it is valid with respect to $U$ also, we know that the boundary of our random walk is non-trivial. 
Now consider a total order $T'$ which respects $T$. If the boundary of the random walk on $G$ is non-trivial, then there exists a valid block with respect to $T'$ (and thus valid with respect to $T$) which has a non-trivial boundary.

\end{proof}

\section{General facts about 
modules, Krull dimension of rings and Krull (Tushev) dimension of groups} \label{sec:section7}

We recall that since the 1954 paper of Hall \cite{hall}, many results about finitely generated metabelian groups $G$ are proven by considering the commutator group $[G,G]$ as a $\mathbb{Z}[H]$ module,
 for $H=G/[G,G]$.
The properties of finitely generated metabelian groups are thus closely related to that of finitely generated modules over Laurent polynomial rings in finitely many variables.
In this section, we formulate basic properties of such modules and related notions.


Given a ring $R$ and a group $G$, let $R[G]$  denote the group ring.  The elements of this ring can be considered  as finitely supported  functions from $G$ to $R$.
If $G=\mathbb{Z}^d$, then
the elements of the group ring can be identified with Laurent polynomials;
elements with support $(\mathbb{Z}^{+})^d$ 
can thus be considered as polynomials in $d$ variables, and
elements of $(\mathbb{Z}^{+})^d$ as monomials.


Given a finitely generated metabelian group $G$ and a short exact sequence 
$$
1 \to M \to G \to \HAKquotient \to 1,
$$
with $\HAKquotient$ and $M$ abelian, we  consider $M$ as a module over $\mathbb{Z}[\HAKquotient]$.

A strictly increasing sequence of proper prime ideals $\rho_0 \subseteq \rho_1 \subseteq \dots \subseteq \rho_n$
is said to be a chain of  proper prime ideals, and its length is $n$.
We recall that for a commutative Noetherian ring $R$ its {\it Krull dimension} is the supremum of the lengths
of chains of prime ideals in the ring.  If $M$ is an $R$-module, its dimension can be defined
as the dimension of the ring $R/({\rm Ann}_R M)$. Here ${\rm Ann}_R M$ denotes the annihilator of $M$
in $R$, that is, the ideal of $R$ consisting of elements   whose  
product  with every  element of the module is zero.

We mention that the Krull dimension can also be defined for non-commutative rings  \cite{RentschlerGabriel}, see also \cite{McConnellRobson}.
This definition  has inspired a definition of the Krull dimension of groups, defined and studied by Tushev in \cite{tushev}.
In the case of metabelian groups, this dimension is known to be equal to 
 the dimension of $[G,G]$
considered as  a $ \mathbb{Z} [G/[G,G]]$-module, see \cite{jacoboni}, Proposition 2.30.
Moreover, this proposition states that for any short exact sequence
$$
1\to M \to G \to \HAKquotient\to 1
$$
where the  subgroup $M$ and the quotient group $\HAKquotient$ are abelian groups, the Krull dimension of $G$ is equal to the dimension of $M$ considered as a $\mathbb{Z} [\HAKquotient]$-module. So if we take  as a definition of
the dimension of 
 $G$ the dimension of the corresponding module $M$, as above, the definition does not depend on the choice of the short exact sequence.

 If the group $M$ in our short exact sequence
 is a $p$-torsion group we can tensor this group by $\mathbb{Z}/p\mathbb{Z}$ and consider this group as a module over $\mathbb{Z}/p\mathbb{Z}[\HAKquotient]$. If $M$ is torsion-free, then, by taking its tensor product with $\mathbb{Q}$, we can consider it as a $\mathbb{Q}[A\HAKquotient]$ module.
If $M$ is torsion-free, 
 we can speak of
 $\dim_{{\mathbb Q}[\HAKquotient]}(M \otimes \mathbb{Q})$
 (or for short   $\dim_{\mathbb{Q}}(M)$)
and
about $\dim_{{\Z/ p \Z} [\HAKquotient]} M $ (or for short   $\dim_{\Z / p \Z} M $)  for $M$ where the elements of $M$ are of exponent $p$.
We consider the  Krull dimension of these modules considered as modules over $\mathbb{Q}[\HAKquotient]$, or respectively over $\Z/p \Z [\HAKquotient]$ modules. This dimension is equal to the 
Krull
dimension of $G$ for ($p$-torsion)-by-abelian
metabelian groups. This dimension   is equal to the Krull
dimension decreased by one for torsion-free  metabelian groups. This is summarized in 
Claim \ref{claim:knowndimension} below.

\begin{claim}\label{claim:knowndimension}
Let $G$ be a finitely generated metabelian group.
Assume that $M$ and $\HAKquotient$ are abelian, such that there is a short exact  sequence
$$
1\to M \to G \to \HAKquotient \to 1
$$

\begin{enumerate}
\item
Assume that $M$ is an elementary abelian $p$-group (any element of $M$ is of exponent $p$).
 The dimension of $M$, considered as a
$\mathbb{Z} [\HAKquotient]$ module, is the same as dimension of $M$ considered as a
$\mathbb{Z}/p\mathbb{Z}[\HAKquotient]$-module:
$$
\Tushev(G)=\dim_{\mathbb{Z} [\HAKquotient]M} = \dim_{\Z/p\Z[\HAKquotient]} M 
$$

\item 

Assume that $G$ is torsion-free. Then 
the dimension of $M$ considered as  a $\mathbb{Z} [\HAKquotient]$ module is the dimension of $M \otimes_{\Z} \mathbb{Q}$ considered as
$\mathbb{Q}[A]$-module increased by $1$:
$$
\Tushev(G)=\dim_{\mathbb{Z} [\HAKquotient]M} = \dim_{\mathbb{Q}[\HAKquotient]} M \otimes_{\Z} \mathbb{Q}+1
$$

\end{enumerate}
\end{claim} 
\begin{proof}
 Let $R= \mathbb{Z}[\HAKquotient]$
and $R'= \mathbb{Z}/p\mathbb{Z}[\HAKquotient]$.  
  We need to prove that the dimensions of 
$R/({\rm Ann}_R M)$ and $R'/({\rm Ann}_R' M)$ are equal.
Observe that these quotient rings are isomorphic.
Indeed, take the natural map from $R$ to $R'$ and its composition with the map from $R'$ to $R'/({\rm Ann}_R' M)$. Observe that
this composition map induces  an isomorphism of  $R/({\rm Ann}_R M)$ and $R'/({\rm Ann}_R' M)$, and hence the dimensions of these rings are equal; we have therefore proven the first claim.

Now we prove the second claim. Consider  a torsion-free Abelian group $M$ which is a module over the polynomial ring 
on $L$ variables $\Z [\Z^L]$.
 Tensor $M$ by $\mathbb{Q}$. It is clear that we obtain a   $\mathbb{Q}[\Z^L]$ module.
The corresponding Krull dimensions differ by $1$:
$$
\dim_{\Z [\Z^L]}  M  = 1+ \dim_{\Q [\Z^L]}  M \otimes_{\Z} \mathbb{Q}.
$$
Indeed, the ideals of  $\Z [\Z^L]$ over the annihilator of   $M$
 are in one-to-one correspondence
  with the ideals of  correspondence with ideals  $\Q [\Z^L]$
over the annihilator 
of $M \otimes_{\Z} \mathbb{Q} $ 
not containing non-zero constant polynomials.

\end{proof}

\begin{rem}
The definition of Krull dimension of groups due to Tushev that we use in this paper and this section should not be confused with
the definition in   Myasnikov and Romanovskiy \cite{MyasnikovRomanovskiy}, where the authors define the notion of Krull dimension in their work on algebraic geometry  of groups project. In their definition 
the dimension of $\mathbb{Z}^d$ is equal to $d$ (see Lemma 3.5 of their paper).
 Note that the dimension of $\mathbb{Z}^d$ in the definition of Tushev is equal to $1$, for any $d\ge 1$.
\end{rem}






An element of an upper-triangular two-by-two matrix is a nontrivial unipotent if it has 1 on the diagonal and not zero in the upper right corner.

\begin{lem}\label{exa:twobytworank}
Let $G$ be a group of two-by-two upper-triangular matrices over a field $k$
 with at least one non-trivial unipotent element.
 Given a matrix $g$

\[
g = \left( \begin{array}{ccc}
a & c \\
0  & b 
 \end{array} \right),
 \] 
let $\phi(g)=b/a$. If the minimal field ${\mathbf K}$ including all $\phi(g)$, $g \in G$ has transcendence degree  $d$,  then the
$\Tushev(G)=d$ if the characteristic of $k$ is positive and 
 $\Tushev(G) = d+1$ otherwise. 
 
 \end{lem}

\begin{proof}
Consider the following homomorphism 

\[
g = \left( \begin{array}{ccc}
a & c \\
0  & b 
 \end{array} \right)     \to   \left( \begin{array}{ccc}
a & 0 \\
0  & b 
 \end{array} \right)
 \] 

Its image is a multiplicative subgroup  of  $k^2$. 
We denote the diagonal group by $D$.
We consider the short exact sequence induced by this homomorphism.
$$
1 \to M \to G \to D \to 1
$$
We know that the dimension of $M$ as a $\mathbb{Z}[D]$ module is the Krull dimension 
of the group $G$.
rem:dimsubew of  Claim \ref{claim:knowndimension} it is sufficient to show that $\dim_{\Z/p \Z} M = d$ if $k$ has positive characteristic  and $\dim_{\mathbb{Q}} M \otimes \mathbb{Q} = d$ if $k$ has characteristic $0$.
Consider the restriction  our  map $\phi$ (defined on two-by-two upper-triangular matrices)   to $D$
$$
\phi_D : D \to k.
$$

Note that the kernel of $\phi_D$ 
 is equal to the annihilator  of $M$ in $\mathbb{Z}[D]$.
Inside $k$ we consider a ring generated by $\phi(g)$, $g\in G$. Since it is a subring of a field, it is an integral domain.
We denote it by ${\mathbf A}$. 
It is clear that  the  fraction field of ${\mathbf A}$ is equal to  ${\mathbf K}$ and the prime subfield of $\mathbf K$ is $\mathbb{P}$. It is clear that $\mathbb{P}=\mathbb{Q}$ if $k$ has characteristic $0$ and 
$\mathbb{P}=\Z /p \mathbb{Z}$ otherwise.  
By Theorem A in \cite{eisenbud} it follows that the Krull dimension of $\mathbf{A}$ (in the case where $\mathbb{P}$ is finite) or $\mathbf{A}\otimes \mathbb{Q}$ (in the case where the prime field is $\mathbb{Q}$) is equal to the transcendence degree of $K$. 

\end{proof}

For example,  for the wreath products $\Z/ p \Z \wr \Z^d$, $\Z \wr \Z^d$, for the free metabelian group
$Met_d$ and for the free $p$-metabelian group $Met_d(\Z/pZ)$, their $\dim_{\Z/p \Z}$ and $\dim_{\mathbb Q}$ dimensions are equal to $d$.
However, while the Krull dimension of $\Z/ p \Z \wr \Z^d$ and  $Met_d(\Z/p \Z)$ is also equal to $d$, the Krull dimension of  $\Z \wr \Z^d$ and  $Met_d$ is equal to $d+1$. 
Lemma \ref{exa:twobytworank} also explains the dimension of some other two-by-two matrix groups, including
Baumslag groups,  described in Section 3.1. We will discuss some further examples in Subsection
\ref{subsection:further}.




\begin{notation} \label{def:rank}
Let $k$ be a field and $\HAKquotient$ be a finitely generated abelian group.
Given a finitely generated $k[\HAKquotient]$ module $M$ and a subgroup $\A'$ of $\HAKquotient$, observe that we can consider $M$ as a $k[\A']$ module.
We say that $M$ is finitely $\A'$ generated if $M$ is finitely generated 
as a $k[\A'] $ module. Let $d$ be the minimal number such that there exists $\A'= \mathbb{Z}^d +C$, where $C$ is a finite abelian group and  where $M$ is finitely generated as an $\A'$ module. We denote this number 
$d_{k[\HAKquotient]}(M)$. If the choice of  $k$ and $\HAKquotient$ is clear, we will omit $k[\HAKquotient]$ in the notation and write $d(M)$
\end{notation}

 Let $d(M)=d$, described in Notation \ref{def:rank}. We know that  there exists a subgroup $\A'= \mathbb{Z}^d +C$, $C$ is a finite abelian group, $\A'$ is a subgroup of $A$,  such that $M$ is finitely generated as an $\A'$ module.  Observe that in this case there exists $A^{''}=\mathbb{Z}^d$, which as a subgroup of $A$, such that $M$ is finitely generated as $A^{''}$ module.

\begin{prop} \label{prop:ourdefinition}
Let $k$ be a field, $A$ be a finitely generated abelian group and $M$ be a $k[A]$-module.
The number $d_{k[\HAKquotient]}$ defined in Notation  \ref{def:rank} satisfies
$$
d_{k[\HAKquotient]} M = \dim_{k[\HAKquotient]} M.
$$
\end{prop}

%

\begin{proof} 


We start with the following lemma

\begin{lem}\label{newclaim:freesubmodule}
Let $M$ be a module over $k[\HAKquotient]$ such that $d(M)$ is equal to  $d$. Let $B$ be a subgroup of $\HAKquotient$, isomorphic to $\mathbb{Z}^d$ such that $M$ is finitely generated as a $k[B]$ module. Then $M$, considered as a $k[B]$ module,  has a nontrivial free submodule.    
\end{lem}

\begin{proof}

Let $W$ be a finite subset of $\mathbb{Z}^d$.  Observe  that there exists a  basis $b_1,\dots, b_d$ of $\mathbb{Z}^d$ such that elements of $W$ have pairwise  distinct first coordinates. 
Indeed, consider the set $W'$ consisting of $w^{(1)} -w^{(2)}$, for $w^{(1)}, w^{(2)} \in W$, $w^{(1)} \ne w^{(2)}$.
Consider a linear form $\mathbb{Q}^d \to \mathbb{Q}$, which does not vanish on any $w'\in W'$.
Multiplying this map by an integer constant, we get a map $\mathbb{Z}^d \to \mathbb{Z}$,
which is obtained by taking a scalar product with some vector $b_1 \in \mathbb{Z}^d$, such that the common
divisor of its coordinates is one and
such that the scalar product does not vanish on elements $w'$, 
$w' \in W'$. It is clear that there exist $b_2$, \dots, $b_d \subset \Z^d$ such that
$b_1$, $b_2$, \dots, $b_d$ form a basis of $\mathbb{Q}^d$.



Consider a  subgroup
$B$
of $A$, isomorphic to $\mathbb{Z}^d$, over which $M$ is finitely generated.
Consider $M$ as a $k[B]$ module. We want to show that  $M$  has a non-trivial free   $k[B]$  submodule.

 Suppose, for the sake of contradiction, that there is no nontrivial free submodule. Let $U = u_1, u_2, ..., u_d$ be a finite generating set of $M$, considered
as a $k[B]$-module. Note that  each $u_k$ has a non-zero annihilator in $k[B]$ (otherwise the module generated by $u_k$ would be free). Let us
choose an element in the annihilator for each $u_i$ and denote these elements  $r_1$,$r_2$, \dots, $r_n$. Then $q= r_1r_2...r_n$ annihilates
$U$ and hence $M$. Observe that the element $q$ is non-trivial because the group ring $k[\Z^d]$ is an
integral domain.

Assume that for some $\alpha_x$  the element $q \ne 0$ is written as a finite sum
$$
q = \sum_{x\in \Omega} \alpha_x x
$$
As mentioned in the beginning of the proof
 we can find a basis of $W$ so that elements in $\Omega$ have pairwise distinct
first coordinates. Consider the subgroup $C$ of $W$, consisting of elements with $0$ in the first
coordinate, it is clear this subgroup is isomorphic to $\Z^{d-1}$. It is easy to see (by applying polynomial long division to $q$), that
$M$ is finitely generated as a $k[C]$ -module.
And this contradicts the assumption that $d$ is
the minimal number such that there exists $\A'= \mathbb{Z}^d +C$, where $C$ is a finite abelian group and  where $M$ is finitely generated as an $\A'$ module.

We have proved that $M$, considered as a  $k[Z^d]$ module, contains a free submodule. 

\end{proof}

Now we need to prove the proposition.
We consider $A$ and $M$ as in the assumption of the proposition. Take $B$
as in the assumption of Lemma \ref{newclaim:freesubmodule}. 
$B$ is a subgroup of $\HAKquotient$, and hence
$k[B]$ is  subring of $k[\HAKquotient]$.
Consider the composition of this map with the map $k[\HAKquotient] \to k[\HAKquotient]/Ann(M)$. Since $M$, considered $k[B]$ module, admits a free submodule, we conclude that  the composition map is injective.
$$
k[B] \to k[\HAKquotient]/Ann (M)
$$
Since $M$ is finitely generated as a $k[B]$ module, $k[\HAKquotient]/Ann(M)$ is finitely generated as a $k[B]$ module. Thus by Axiom D3 in Chapter 8 \cite{eisenbud} applied to $S=k[A]/Ann(M)$ and $R=k[B]$, we know  that the Krull dimensions of these rings are equal. 
By Theorem A in \cite{eisenbud} the Krull dimension of $k[B]$,
and hence of $k[A]/Ann(M)$,
is equal to  $d$.
This concludes the proof of the proposition.


\end{proof}

\begin{cor}
Consider a finitely generated metabelian group $G$,
abelian groups $\HAKquotient$ and $M$ and a  short exact sequence
$$
1 \to M \to G \to \HAKquotient \to 1
$$
Assume that $M$ is either torsion-free or all its elements are of exponent $p$.  Consider the field
$k$ which is equal to $\mathbb{Q}$ in the first case and $\Z/p\Z$ in the second case. We have
$$
d_{k[\HAKquotient]} M = \Tushev(G)
$$
if $M$ is $p$-torsion and 
$$
d_{k[\HAKquotient]} M = \Tushev(G)+1
$$
if $M$ is torsion-free.
\end{cor}

\begin{proof}

Follows by combining Claim \ref{claim:knowndimension} and  Proposition \ref{prop:ourdefinition}
\end{proof}


In the lemma below, we will consider a generating set $S$ of $\Z^m$ and the ball of
radius $r$ in this generating set which we denote $B_{\Z^m,S}(r)$. To simplify the notation
we write
$$
B_r = B_{\Z^m,S}(r).
$$
In the proof of this lemma we will also consider a subgroup of $\Z^m$  isomorphic to $\Z^d$,  with a generating set $T$ and the ball of radius $r$ in this generating set which we
denote $B_{\Z^d,T}$.
To simplify this notation, we will write  $B'(r)=B_{\Z^d,T}$.

\begin{lem}[The dimension of obtainable vector spaces]\label{prop:mainsection6}
We consider a short exact sequence 
$$
1 \to M \to G \to \mathbb{Z}^m \to 1,
$$
$G$ is  finitely generated and $M$ is a vector space over a field $k$ (which is embedded as an abelian subgroup in $G$).
Put 
$$
d= \dim_{k [\mathbb{Z}^m]} M
$$

Let $U$ be a finite generating set of $M$.
Let $
V_{(B_r,U)}
$ be the
vector space over $k$ spanned by the products $z u$ for $z\in B_r$, $u\in U$.
Then
there exist positive numbers $C_1$ and $C_2$ such that 
all $r\ge 1$
it holds
$$
     C_1 r^d \le  \dim(V_{(B_r,U)}) \le C_2 r^d.
$$
     
\end{lem}

\begin{proof}

First observe that it suffices to prove the statement of the lemma for some choice of generating set $U$ for $M$ considered as a $k[\mathbb{Z}^m]$-module.
Indeed, the dimension as a vector space of
$V_{(B_r,U)}$
for other choices of the generators of the module $M$
will only differ by multiplication by a positive constant.

By Proposition \ref{prop:ourdefinition}
 we know  that there exists a subgroup $B=\mathbb{Z}^d \subset \mathbb{Z}^m$ such that the module $M$ is finitely generated as a $k[B]$ module. We denote such a generating set as $U$.
Choose a sufficiently large $q$ such that the following holds.
For each of the $m$ generators
$e_1$, $e_2$, \dots, $e_m$ of $\mathbb{Z}^m$ and each of the elements of $u\in U$
$$
e_iu \in V_{(B'_q,U)}
$$

Then it is easy to see that any element of the form $bu$ where $b\in B_r, u\in U$ is in $V_{B'_{rq}}$.
This implies the upper bound in the claim of our proposition. 


The lower bound in the claim of the lemma  follows from 
Lemma \ref{newclaim:freesubmodule}
since the lower bound holds true for a free module and since  $\dim(V_{(B_r,U)} \leq \dim(W_{(B_r,U)}$  if $W$ contains $V$ and the finite generating set $T$ for $W$ contains $U$.  

\end{proof}

\subsection{Further examples}\label{subsection:further}

We have mentioned  examples of evaluation of dimension of some metabelian groups
after Lemma \ref{exa:twobytworank}. Now we explain some more examples. 
The  claims and the lemma of  this subsection are not  used in the proofs of our main results. But they  can be used to explain examples mentioned in the beginning of the paper (which illustrate the application of our theorems).

\begin{claim} \label{rem:modulesubmodule}
Let $M$ be a finitely generated module over $k[\mathbb{Z}^r]$ and $M'$ be a submodule.
Then the dimension $d'$ of $M'$ satisfies $d'\le d$ where $d$ is the dimension of $M$.
\end{claim}

\begin{proof}
Follows from the fact that the Krull dimension of rings is monotone under quotients.
\end{proof}

\begin{claim}\label{rem:rankforchangeofring}
Let $M$ be $k[\mathbb{Z}^n]$ module. Take $m<n$,
we have $\mathbb{Z}^m \subset \mathbb{Z}^n$,
and consider $M$ a $\mathbb{Z}^m$-module.
Assume that $M$, considered as a $k[\mathbb{Z}^m]$-module, has dimension $d$. Then its dimension as $k[\mathbb{Z}^n]$-module is also $d$.
\end{claim}
\begin{proof}
The upper bound is immediate from the definition of the dimension. The lower bound follows from Lemma \ref{prop:mainsection6}
about the dimension of obtainable vector spaces.
\end{proof}

The following lemma is in particular useful for evaluating dimension for quotients of wreath products.

\begin{lem} \label{lem:rankdrops}[Dimension drops]
\begin{enumerate}
\item
Let $k$ be a field.  Let $M$ be a proper quotient of the  free module $k[\mathbb{Z}^d]$. Then
its Krull dimension  $\dim_{k[\Z^d]} M$ is
 strictly less than $d$. 
\item
If $I$ is a nontrivial principle ideal then $M/I$ has dimension precisely  $ d-1$.
\end{enumerate}
\end{lem}
\begin{proof}

Given a Laurent polynomial $p'$ 
 we consider  its  monomials   $c_1^{j_1}  \dots c_d^{j_d}$, put  $M=\max{j_1}$, and $m=\min{j_1}$,
where the maximum and the minimum are taken over the monomials with non-zero coefficients.
We say that a Laurent  polynomial $p'$ satisfies *-condition
if  $p'$ has a single monomial with $j_1=M$ (with non-zero coefficient) and a single monomial where $j_1=m$ (with non-zero coefficient).

By our assumption of the lemma there is at least one  non-zero element $r\in k[\mathbb{Z}^d]$ that belongs to the annihilator of  $M$. Thus, there exists a polynomial $p$
 in $d$ variables $b_1$, \dots, $b_d$ that also belongs to this annihilator.

Arguing as in the beginning of the proof of Lemma \ref{newclaim:freesubmodule},
we can choose a basis $c_1,c_2, c_d$ of $\mathbb{Z}^d$ such that the following holds. Any polynomial in $b_1, \dots, b_d$ can be rewritten 
as a  polynomial $p'$ in $c_1, \dots, c_d$. 
$$
p(b_1, \dots, b_d) = p'(c_1, \dots, c_d)
$$
such that $p'$ satisfies *-condition.


Remove $c_1$ from the basis. We claim that our quotient module is a finite dimensional module over  $c_2,\dots,c_d$, so the dimension of this module is at most $d-1$. To see that it is finite dimensional over $c_1, \dots, c_d$, we use a version of  polynomial long division.
Namely,  if $u$ and $v$ are Laurent polynomials in $c_1, \dots, c_d$, and $v$ satisfies $*$-condition
with values $\min$ and $\max$, then there exists a Laurent polynomial $w$ with $c_1$ coefficients bounded between $\max$ and $\min$ in monomials of the corresponding polynomial
and a Laurent polynomial $t$ so that $u=vt+w$.
Since the coefficient of $w$ in $c_1$ is bounded, we conclude that the module is finite dimensional over $k[\mathbb{Z}^{d-1}]$ (where $\mathbb{Z}^{d-1}$ corresponds to the group generated by $c_2, \dots, c_d$.)

Now we prove the second claim of the lemma. Let $p$ be a generator of the ideal.
Let $x_1$,$x_2$, \dots, $x_d$ be free generators of $\mathbb{Z}^d$.
The ideal $I\neq k[\mathbb{Z}^d]$, $p$ has at least two non-trivial monomials. Thus there is a coordinate $c_i$ such that the monomials of $p$ have at least two distinct powers of $c_i$.
Renaming the generators, we can assume that it happens for $c_1$.
Observe that a product of $p$ and a non-zero Laurent polynomial $q$ has at least two distinct powers of $c_1$ (a largest and a smallest).  Therefore, every Laurent polynomial in $c_2, \dots, c_d$ is not a multiple of $p$ and thus there is a free submodule of dimension $d-1$. Note that the rank of a free module is $d$; the dimension is not increasing by taking a submodule. 
Thus, taking into account Claim
 \ref{rem:rankforchangeofring} we conclude that the dimension is $\ge d-1$.
\end{proof}

\section{The cautiousness criterion for triviality of the boundary}\label{sec:cautious}
In this section  we will prove the cautiousness criterion
for triviality of the Poisson boundary.
It will imply, in particular, that (under sufficient moment conditions) random walks on a linear group in characteristic $p$ of dimension $d \le 2$ or on a linear group in characteristic $0$ of dimension $d \le 1$ have trivial boundary.

We  start by explaining the idea of our cautiousness criterion. Consider
the basic and well-known example of   $
\mathbb{Z}/2\mathbb{Z} \wr \mathbb{Z}^2$. There are several known arguments to explain why simple random walks on this group have trivial boundary.
The original argument of Kaimanovich and Vershik \cite{kaimanovichvershik}
uses the fact that the commutator subgroup is recurrent, considers the exit measures to this subgroup and then concludes by observing that this exit measure has trivial boundary, since the subgroup is abelian.

We note that this argument cannot be applied to many of the groups we are interested in. For example, in the two-dimensional Baumslag group, the quotient over the commutator is $\mathbb{Z}^4$, and this subgroup is transient.

Another way to prove the triviality of the boundary for the wreath products is to observe that the entropy function is closely related to the drift of the random walk on $\mathbb{Z}^2$, and use that this drift is sublinear for recurrent random walks.
Such estimates in general require a detailed understanding
of the typical elements  
visited at time $t$, and  we do not know in what generality a detailed description of such elements  can be obtained for general group extensions.

Now we explain one more way to see that simple random walks on the two-dimensional Lamplighter group have trivial boundary.
Observe that for any $\varepsilon>0$  the random walk on the projection to $\mathbb{Z}^2$
stays in the ball of radius $\varepsilon \sqrt{t}$ for all
  time  instants between $1$ and $t$, with positive probability  
not depending on $t$.  See Figure \ref{pic:cautiousz2}.
 \begin{figure}[!htb] 
\centering
\includegraphics[scale=.65]{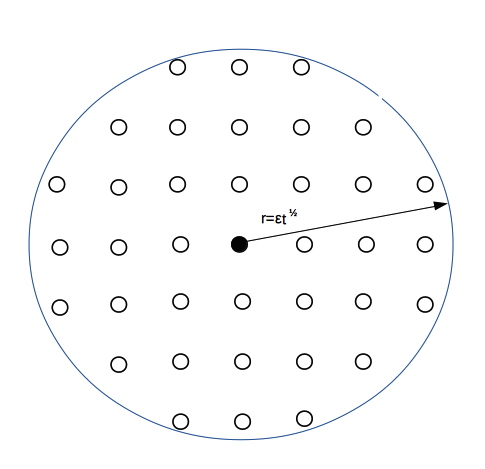}
\caption{With positive probability the random walk on $\mathbb{Z}^2$ stays until time $t$ in the ball of radius $r=\varepsilon \sqrt{t}$.
}
\label{pic:cautiousz2}
\end{figure} 
(Here this property holds for all $t$. When it holds for a subsequence of time instants,   the property is called cautiousness in \cite{erschlerozawa}, see also Definition \ref{def:cautious} for a more general version of this definition).
Observe that with positive probability the position of the random walk at time $t$ belongs to a set of cardinality at most $t^2 \exp(C \varepsilon^2 t)$. By the Shannon -- McMillan -- Breiman type theorem, one concludes that the asymptotic entropy of the random walk is zero.

In this section we  will
show that this latter argument can be generalized to many group extensions. 

\begin{defn}[$f$-cautious random walk]\label{def:cautious}
Let $f:\N \to \R$ be an increasing function.
A random walk $X_n$ is  $f$-cautious along some subsequence if 
$$
\limsup_{n} \mathbb{P} [\max_{1 \le m \le n}(|X_m|/f(n))<\varepsilon]
$$
is positive for every $\varepsilon>0$.
\end{defn}

In other words, the definition says that for each $\varepsilon$ there exists a $p_{\varepsilon}>0$ 
and a subsequence $n_i$ such that the following holds:
The probability that during the time interval between $1$ and $n_i$ the random walk stays inside the ball of radius $\varepsilon f(n_i)$
is at least $p_\varepsilon$.

\begin{rem}[$f$-cautious random walks with condition for all $n$.]
Let $f:\N \to \R$ be an increasing function.
A random walk is $f$-cautious if for all $\varepsilon >0$
there exists $p_{\varepsilon}>0$ such that 
for all sufficiently large $n$
 trajectories $X_n$ satisfy the following. 
$$ 
 \mathbb{P} [\forall m: 1 \le m \le n ,|X_m|/f(n)<\varepsilon] \ge p_{\varepsilon}
$$ 
\end{rem}

\begin{rem}\label{rem:abeliancautious} The Central Limit Theorem shows that symmetric
finite second moment random walks on $\mathbb{Z}^d$ (and more generally any finitely generated abelian group)
are $f(n)$ -cautious for $f(n)=\sqrt{n}$.
\end{rem}

\begin{rem}
Let $G$ be a group and $\mu$ be a finite first moment measure. In this situation, the drift  of the random walk is zero if and only if the r.w. is $f(n)$ -cautious for $f(n)=n$.
\end{rem}

\begin{rem} If $f(n)=\sqrt{n}$, our notion of $f(n)$-cautious 
is the same as cautiousness defined in \cite{erschlerozawa}.
\end{rem}

\begin{rem}
We mention that there exists $\sqrt{n}$-cautious random walks where
it is important that the condition holds
along a subsequence. Lacunary hyperbolic examples with this property are studied in Theorem 5.1 of \cite{erschlerzheng}.
\end{rem}

\begin{defn}\label{def:span}[Span function for actions of $G$.]
Suppose that a group $G$ acts by automorphisms on a group $H$. Let $S$ be a finite subset of $G$, and $T$ be a finite subset of $H$. Denote by $T_r$ the set of all elements $h$ in $H$ for which there exists some $g\in G$ of word length at most $r$ (with respect to $S$) and some $t\in T$ such that $g(t)=h$. Let $T_{r,n}$ be the set of all elements which are a product of at most $n$ elements in $T_r$. Define $\Span_{S,T,G \acts H}(r,n)$ to be the cardinality of $T_{r,n}$.  We also call the set $T^{S,T, G \acts H }_{r,n}=T_{r,n}$ obtainable.
\end{defn}

\begin{rem}
The asymptotics of $\Span_{S,T,G \acts H}(r,n)$ 
do not depend on the choice of $S$
and $T$ when $S$ is a generating set for $G$, and the union of $g(t)$ for all $g\in G$ and $t\in T$ is a generating set for $H$. 
\end{rem}

We state below a criterion for the Liouville property for group extensions using cautiousness  of the projected random walk. 
 
\begin{thm}[Cautiousness criterion for extensions]
\label{prop:cautious}
Let $f:\mathbb{N} \to \mathbb{R}$  be an increasing function.
Let $1\to \HNMsubgroup \to G\to \HAKquotient\to 1$ be a short exact sequence,  where $\HNMsubgroup$ is abelian. 
Let $\mu$ be a non-degenerate finite entropy  probability measure on $G$ whose projection
$\mu_K$ to $\HAKquotient$ is $f$-cautious and Liouville. 
Fix some  finite generating set  $S\subset \HAKquotient$ and a  finite set   $T\subset \HNMsubgroup$. Assume that  for every $\varepsilon>0$ there exists a $\theta>0$, $C_\varepsilon>0$ such that
$$
\Span_{S,T,G \acts \HNMsubgroup}(\delta f(n),n) \le C_\varepsilon (1+\varepsilon)^n  (*)
$$ for all $n$.

\begin{enumerate}
    \item 

If the set $T$ generates $\HNMsubgroup$, then the random walk $(G,\mu)$ is also Liouville. 

\item

More generally, take  $h\in H$ and consider $T=T_h=\{ h \}$.
Then if the assumption $(*)$ holds for $T_h$, then
$h$ acts trivially on the boundary $(T,\mu)$.
\end{enumerate}
\end{thm}

\begin{rem}\label{rem:twomimpliesone}
The second claim of the theorem is indeed more general than the first one. Indeed,  in view of Lemma \ref{lem:obsforreduction} if the boundary of $(G,\mu)$ is non-trivial, then there exists $h\in H$ acting non-trivially on the boundary. The second claim also applies in the case where $H$ is not normally finitely generated. 
\end{rem}

\begin{rem}
We have mentioned that for $f(n)=n$ cautiousness implies the Liouville property, so that if $f(n)\leq n$ we can omit mentioning the Liouville property in the assumption of the criterion. But the criterion makes sense also for  $f(n)>n$
(in particular this applies for some non-symmetric finite first moment random walks with positive drift). 
We mention also that for a symmetric finite first moment measure $\nu$ the Liouville property of   $(G,\nu)$  is equivalent to the fact that  $(G, \nu)$ has $0$ drift, see \cite{karlssonledrappier}. 
\end{rem}

The assumption of the theorem above
applies in particular for $\HAKquotient=\mathbb{Z}^d$ and
$\HNMsubgroup=(\mathbb{Z}/2\mathbb{Z})^{\mathbb{Z}^d}$.

In this case, we can take $\phi(n) = C n^d$ and $f(n)=\sqrt{n}$. We have $\phi(f(n)) \sim n^{d/2}$. So we control the span, assuming that $d=1$ and the first moment is finite  or $d=2$ and the second moment is finite,
and then we recover the well-known fact that in these examples the Poisson boundary on the wreath product is trivial.

Before we prove the theorem, we first explain 
a straightforward observation about the normal form of (the product of increments in)  a semi-direct product.
In the claim and its proof below we
use the same notation for subgroup $\HNMsubgroup$, quotient group $\HAKquotient$ and their image  for the canonical embedding to the  semi-direct product $\HNMsubgroup
\rtimes  \HAKquotient$.

\begin{claim}
\label{lem:Normalform}
Let $G$ be a finitely
generated semi-direct product  $\HAKquotient\rtimes \HNMsubgroup$.
Consider a probability measure on $G$, with the support $U$. Let $V$ be the subset of $\HNMsubgroup$,
consisting of $h \in H$ such that there exists $u\in U$, and $k\in \HAKquotient$ such that $u=hk$.
Let $X_n$ be a trajectory of the random walk on $G$ and 
$Y_n$ be this trajectory's quotient  to $\HAKquotient$.
Let $l_{\HAKquotient}(Y_m)<C$ for
$1\leq m\leq n$.  Then $X_n$ can be written in the form $h'k'$ where $h'\in \HNMsubgroup$,
$k'\in \HAKquotient$ 
and where $h'$ is a product of at most $n$ elements of the form $kvk^{-1}$, where $k\in K:l_{\HAKquotient}(k)<C$, and $v\in V$. 
\end{claim}
\begin{proof}
We prove the statement by induction on $m$. The base $m=1$ is immediate. Now suppose the statement  is true for $X_{m-1}$. Consider  $X_m$. By induction we know that  $X_{m-1}$ can be written in the form $h_{m-1}k_{m-1}$. Here $l_{\HAKquotient}(k_{n-1})<C$ and $h_{n-1}$ can be written as the product of at most $n-1$ elements of the form $kuk^{-1}$,  $k\in \HAKquotient:|k|<C$,  $u\in U$. We have $h_{m-1}k_{m-1}hk=h_{m-1}(k_{m-1}h{k_{m-1}}^{-1})k_{m-1}g$. Since $l_K (k_{m-1})<C$  it follows that $X_m$ can be written in the form $hk$ where $h$ is a product of at most $n$ elements of the form $kuk^{-1}$ where $k\in K:|k|<C$, and $u\in U$.
\end{proof}

\begin{lem} [Cautiousness criterion for semi-direct products]\label{lem:cautiousnesssemidirect}
Assume that $G$ is a semi-direct product corresponding to the short exact sequence, 
$$ 
1 \to \HNMsubgroup \to G \to \HAKquotient\to 1,
$$
where $\HAKquotient$ is finitely generated.
Let $\mu$ be a finite entropy probability measure on $G$, $\mu_K$
its projection to $\HAKquotient$. We assume that $\mu_{\HAKquotient}$  is $f$-cautious and Liouville.
Let $S$ be a finite generating set for $\HAKquotient$.
Assume that the set  $T$ consisting of  $k\in \HAKquotient$ such that there exists at least one $h\in H$ where $hk$ is in the support, is finite.
Assume also that  for every $\varepsilon>0$ there exists  $\delta>0$, 
$C_\varepsilon>0$ such that
$$
\Span_{S,T,G \acts \HNMsubgroup}(\delta f(n),n) \le C_\varepsilon (1+\varepsilon)^n  (*)
$$ for all $n$.
Then the random walk on $G$ is Liouville. 
\end{lem}

\begin{proof}

We want to prove  for all $\varepsilon>0$ that the entropy of $\mu^{*n}$ is at most $\varepsilon  n+o(n)$  which will imply that the entropy is $o(n)$ and thus the  boundary is trivial. 

Observe that by the cautiousness of $(\HAKquotient,\mu_{\HAKquotient})$   
there exists a subsequence  $n_i$ such that the following holds. For each $n_i$ with probability bounded below by some $p_\varepsilon>0$, the
projected random walk $Y_{n_i}$ has word length  less than $\varepsilon f(n_i)$ for $0<n\leq n_i$. 

Thus, by Claim \ref{lem:Normalform} with positive probability bounded below by $p_\varepsilon$, at times $n_i$ our random walk on $G$ is contained in $(T_{(\varepsilon f(n),n)} B_K (\varepsilon f(n))$. Here $T_{(\varepsilon f(n),n)}$ is
the obtainable set $T_{r,n}$ for $r=f(n)$
from the definition of the Span function
for $\HAKquotient\acts H$, $S$,$T$. We recall
$B_K(r)$ denotes the ball of radius $r$ in the group $\HAKquotient$ and we omit in this notation the generating set (if we have fixed some generating set).

By assumption of the lemma, this set is of size at most  $\exp (\varepsilon n+o(n))$ and thus by convergence of the entropy with probability $1$ to its limiting value ( Shannon-McMillan-Breiman type theorem,  see \cite{kaimanovichvershik}, Thm $2.1$ or \cite{derriennic}), 
we know that the entropy function of our random walk is 
$\le  \varepsilon  n+o(n)$ for all $\varepsilon$ and thus is sublinear, which implies the triviality of the boundary.
\end{proof}

We now prove the 
theorem.

\begin{proof}

As we have mentioned in the Remark \ref{rem:twomimpliesone},
it is enough to prove  the second claim of the Theorem. We have   $h\in H$ and 
we assume that the assumption (*) of the theorem  holds for the set $T_h=\{h\}$. That is, that for a generating set $S$ of $\HAKquotient$ and for every $\varepsilon>0$ there exists  $\delta$,  $C_\varepsilon$ such that
$\Span_{S,h,G\acts \HNMsubgroup}(\delta f(n),n)\leq 
C_\varepsilon(1+\varepsilon)^n$ for all $n$. We want to prove that 
$h$ acts trivially on the boundary $(G,\mu)$.

Apply Lemma \ref{lem:reductiontosingle} to $G$, our short exact sequence and $h$.
Consider the obtained group  $G_h$.
By the above-mentioned Lemma, it is sufficient to prove that $h$ acts trivially on the boundary of $G_h$, for some measure on $G_h$ with the same projection to $K$.
Observe that we can choose the measure on $G_h$, with the same projection to $\HAKquotient$ as $\mu$, in such a way that it satisfies the assumption of  Lemma \ref{lem:cautiousnesssemidirect}, in particular  about the finiteness of the set $T$.
This completes the proof of the theorem.
\end{proof}

\subsection{Applications to the triviality of the boundary for metabelian groups}

In the proof of the corollary below, we will need upper bounds for the size of obtainable sets. In the case of extensions of torsion groups, obtainable sets are subsets of finite dimension vector subspaces, and their dimension was estimated in the previous section,
Lemma \ref{prop:mainsection6}(this we will use in the proof of the  first claim of the corollary). In the case of torsion-free abelian groups, the obtainable sets are subsets of a vector space over $Q$, and so to estimate the cardinality we need not only the dimension of an ambient vector space, but also a bound on  coefficients (for elements in this set) in a fixed basis.
This is done in the following Lemma.



\begin{lem}[The cardinality of obtainable sets] \label{lem:obtainableset}
Let $\HAKquotient$ be an abelian group and $\HNMsubgroup$ be a torsion-free abelian group.  $G$ is such that we have an exact sequence
$$
1\to \HNMsubgroup \to G \to \HAKquotient\to 1,
$$
and we assume that the Krull dimension of $G$ is $d$.
Let $S$ be a finite subset of $G$ and $T$ be a finite subset of $\HNMsubgroup$.
Consider the obtainable set $T_{r,n}$ defined in 
Definition \ref{def:span}.
There exists a constant $C>0$ such that
its cardinality  $\Span_{S,T,\HAKquotient \acts \HNMsubgroup}(r,n)$    satisfies
$$
\Span_{S,T,\HAKquotient \acts \HNMsubgroup}(r,n)  \le ((n+2)C^r)^{Cr^d}.
$$

\end{lem}


\begin{proof}
From Claim 7.2 and Proposition 7.5, we know that 
there is a subgroup $K'$ isomorphic to $\mathbb{Z}^d$ of $K$ and a finite set $T'$ of elements of $\HNMsubgroup\otimes \mathbb{Q}$ 
(we can assume that they belong to $H$) such that the set of elements of the form $k(t')$ where $k\in \HAKquotient, t'\in T'$ spans $A\otimes \mathbb{Q}$. Thus every element in $T$ is the span (over $\mathbb{Q}$) of finitely many conjugates of $k(t)$. 

We now will show that there exists $C$,$D$ and $m$ so that every element in $\bar{k}(t)$ where $\bar{k}\in B_r(\HAKquotient)$ and $t\in T$ can be written as a sum of elements of the form $qk(t')$ where $k\in B_{mr}$, $t\in T'$ and $q$ is a rational number with denominator $D^{r+1}$ and numerator with absolute value at most $C^{r+1}$.
Since a sum of $n$ elements of this form can be written in the same way (replacing the maximal numerator absolute value with $nD^{r+1}$) the lemma will  follow.

To see the existence of such $C,D,m$ note that  there exists a finite subset $P\subset K$ 
such that for all $s\in S$ and $t'\in T'$  it holds $s(t')=\sum_{p\in P,t\in T'}
q_{p,t}p(t)$, where $q_{k,t}$ are rational numbers.
Let $D$ be a common denominator of all the $p{k,t}$ for all $s\in S$ and let $m$ be large enough so that $P\subset B_m(\HAKquotient)$. Let $C'$ be the maximal numerator of $q_{k,t}$ and $C$ be $C'|P||T'|D$. Apply an element $s\in S$ to a rational linear combination of $k(t')$ with common denominator $D^{r+1}$ and maximal numerators $C^{r+1}$ , where $k\in B_{rm}, t'\in T'$.
We obtain  a rational linear combination of $k(t')$ where $k\in B_{(r+1)m}$ have common denominator $D^{r+2}$ and numerators bounded above by $C^{r+2}$
(since we are summing up at most $|P|T'|$ different rationals, multiplying the denominator by $D$ and each numerator by at most $C'$).  This yields the desired result. 

\end{proof}

\begin{cor}\label{cor:rank2rank1}
Let $G$ be a metabelian group of Krull 
dimension at most $\le 2$.


Then for any centered finite second moment
measure $\mu$ on $G$ the Poisson boundary of the random walk $(G,\mu)$ is trivial.
Moreover, if the dimension $\le 1$, then any centered first moment measure $\mu$ on $G$
has trivial Poisson boundary.

\end{cor}

\begin{proof}

Assume that we have a short exact sequence
$$
1\to \HNMsubgroup \to G \to \HAKquotient \to 1
$$
where $\HNMsubgroup$ and $\HAKquotient$ are abelian. And we first deal with the case when $\HNMsubgroup$ is either $p$-torsion, where $p$ is prime, or it is torsion-free.

Observe that the Central Limit Theorem implies that any centered
finite second moment random walk on $\mathbb{Z}^d$ and, more generally, on any f.g. abelian group,  is $\sqrt{n}$-cautious
(see Remark \ref{rem:abeliancautious}).
We assume that the defining measure $\mu$ has finite second moment.
Apply Theorem \ref{prop:cautious}
for $f(n)=\sqrt{n}$. 
First, we consider the case when the subgroup $\HNMsubgroup$ is $p$-torsion.
Observe that
the obtainable set $T_{r,n}$ for $r=\delta f(n)$
is contained in a vector space,  which in view of Lemma \ref{prop:mainsection6}
has dimension at most $C (\delta(f(n))^2)=
C \delta^2 n$.
Thus, the cardinality of these obtainable sets is at most $\exp(\ln p C \delta^2 n)$. 

Now we consider the case when the subgroup $\HNMsubgroup$ is torsion-free.
Here we also assume that $\mu$ has finite second moment.
By Lemma \ref{lem:obtainableset}
we know that 
$$
\Span_{S,T, \HAKquotient \acts \HNMsubgroup}(r,n)  \le ((n+2)C^r)^{Cr}.
$$
We put $r= \delta \sqrt{n}$, 
 observe that $\Span_{S,T, \HAKquotient \acts \HNMsubgroup}(\delta \sqrt{n},n) \le (n+2)^{C \delta \sqrt{n}}C^{C \delta^2 n}$   and conclude that for every $\varepsilon>0$ there exists some $\delta$ such that
 $$
 \Span_{S,T,\HAKquotient\acts A}(\delta \sqrt{n},n)<C_\varepsilon (1+\varepsilon)^n.
 $$
Therefore, the assumption of Theorem \ref{prop:cautious} is satisfied, and we conclude that our random walk has trivial boundary.

Now we assume that the Krull dimenison is one and  the defining measure $\mu$ is centered and has 
 finite first moment. 
We apply Theorem \ref{prop:cautious} for $f(n)=n$.  
By the Law of Large Numbers, any finite first moment centered measure on $\mathbb{Z}^d$ (and more generally, on any finitely generated abelian group) has zero drift, and 
$l(X_n)/n \to 0$ almost surely. Thus we have $f(n)$-cautiousness for $f(n)=n$.
 In the $p$-torsion case we observe the following.
 We use  Lemma \ref{prop:mainsection6} (the lemma about the dimension of obtainable vector spaces). 
 The fact that the dimension of our extension $\le 1$ implies, in view of this Lemma,  that for all $\delta>0$
$\Span(\delta n, n) 
\le c^{\delta n}$ for some $c>0$ and all $n$. 

In the torsion-free case we observe the following.
It holds $
\Span_{S,T,\HAKquotient \acts \HNMsubgroup}(\delta n,n)  \le ((n+2)C^{{\delta n}^{C}}$,
and again we can conclude that for every $\varepsilon>0$ there exists some $\delta$ such that $\Span_{S,T, \HAKquotient \acts \HNMsubgroup}(\delta \sqrt{n},n)<C_\varepsilon (1+\varepsilon)^n$.  Which implies that the bounds of Proposition $\ref{prop:cautious}$ are satisfied proving the second  part of the claim. 

First we deal with the dimension 

And then we consider the general case. We have
$$
1\to \HNMsubgroup \to G \to \HAKquotient \to 1
$$
and $\HNMsubgroup$ can have  torsion elements of various order  and elements of infinite order.
Let ${\rm Ker}$ be elements acting trivially on the Poisson boundary of our random walk on $G$.
Consider the quotient $G/{\rm Ker}$. It is sufficient to show that this group has trivial boundary. If the abelian normal subgroup of this group does not have torsion, then we have proven the claim already. For the sake of contradiction, suppose that the boundary is non-trivial. 
Using Lemma 5.15 we observe that there is at least one prime $p$ and an element $q$ of order $p$ such that the extension of the subgroup of $\HNMsubgroup$ generated by $q$ by $\HAKquotient$ has a non-trivial boundary. 
Note that the Krull dimension of this extension is at most two, since the Krull dimension is monotone for taking quotients, taking subgroups and does not change for replacing extensions by those with the same action on the Abelian subgroup \ref{claim:changingextensionKrull}. Thus we get a contradiction with the claim that we have already  proven for $p$-torsion by abelian extensions.

\end{proof}

\begin{claim} \label{claim:changingextensionKrull}
Consider finitely generated groups $G$ and $G_2$ and abelian groups $\HNMsubgroup$ and $\HAKquotient$.
Assume that we have short exact sequences
$$
1 \to \HNMsubgroup \to G \to \HAKquotient \to 1
$$
and 
$$
1 \to \HNMsubgroup \to G_2 \to \HAKquotient \to 1
$$
and the actions of $\HAKquotient$ on $\HNMsubgroup$
are the same. Then the Krull dimension of $G$ is equal to the Krull dimension of $G_2$.
\end{claim}
\begin{proof}
Follows from the definition of the Krull dimension, since the corresponding modules are isomorphic.
\end{proof}

As we have already explained in Section \ref{section:nilpotentbyabelian}, the general metabelian case can be reduced to ($p$-torsion)-by-abelian case and (torsion-free)-by-abelian case.  Claim 1) of Corollary \ref{cor:rank2rank1} gives a sufficient condition for boundary triviality for ($p$-torsion)-by-abelian group. For second
moment centered random walks, we will see in Section \ref{sec:nontrivial}
that this condition is not only sufficient but necessary. Namely, in  the (torsion)-by-abelian case we will show that dimension $\ge 3$ implies non-triviality of the boundary for any finite entropy measure.



\begin{rem}
It is clear that in claim 1) of the corollary the condition of finite second moment is essential.
Indeed, we recall again the well-known example of $ \mathbb{Z}/p\mathbb{Z} \wr \mathbb{Z}^2$. We know that finite first moment random walks can have transient projection to $\mathbb{Z}^2$, and that some of them (for example, those whose support belongs to the union of the base group and the lamp) have non-trivial boundary.
\end{rem}

 In Claim 2) of the corollary, the assumption of a finite second moment is also important, as we explain in the following example.

\begin{rem} The Lamplighter-Baumslag-Solitar groups on $1$ variable are torsion-free metabelian groups of the Krull dimension $2$, but (in contrast with $\mathbb{Z}\wr \mathbb{Z}$) a finite first moment random walk can have a non-trivial boundary, as we will see in Example \ref{exa:m1m2mxmy}. In these examples, finite second moment measures have trivial boundary.
\end{rem}


\begin{rem}[$2$ -dimensional (torsion-free) case with non-trivial boundary].

Lamplighter-Baumslag-Solitar group on $2$ variables will be shown to have a non-trivial boundary for any non-degenerate  finite entropy measure, as we will see in Example \ref{exa:m1m2mxmy}. And a similar, more general family of examples will be discussed in Proposition \ref{prop:Galpha}. 
\end{rem}

The following corollary explains the implication of corollary \ref{cor:rank2rank1} to more general linear groups.
\begin{cor} \label{cor:variablesfield}
Let $F$ be either a function field over  at most $2$ variables of positive characteristic or $F$ is a function field over at most $1$ variable in characteristic $0$.
Then any virtually solvable group, linear over $F$ is Liouville: any finitely supported
(and more generally any finite second moment) centered measure on this group has a trivial
boundary.
\end{cor}

\begin{proof}
By Theorem \ref{thm:lineargroups}, it suffices to prove for $2
\times 2$ basic blocks that the corresponding random walks are Liouville. Observe that these blocks have dimension at most $2$
if the characteristic is positive, and at most $1$ if the characteristic $0$. Indeed, this follows from 
Lemma \ref{exa:twobytworank}.
By Corollary \ref{cor:rank2rank1} the corresponding random walks have a trivial Poisson boundary, proving the corollary.
\end{proof}

\begin{exa}\label{ex:Baumslag}[Baumslag groups
$B_{d,p}$ revisited] 
\newline
We recall that 
$B_{d,p}\in GL_2((\Z/p\Z)(X_1,\dots, X_D)$
is the group generated by the matrix $\theta $ and  $2d$ matrices of the form 
$$
%
%
\left( \begin{array}{ccc}
1 & 0  \\
0 & X_i\\
 \end{array} \right), 
\medspace \medspace
 \left( \begin{array}{ccc}
1 & 0  \\
0 & X_i+1\\
 \end{array} \right).
$$

Here $1 \le i \le d$.
If we consider only the generators of the first two forms, then what we have is the wreath product  $ {\mathbb{Z}/p\Z} \wr 
\mathbb{Z}^d $
If $d=1$ the Baumslag group is a quotient of the $2$ dimensional lamplighter. For $d=3$ the non-triviality of the boundary is proven in \cite{erschlerliouv}. For $d=2$ this group is a particular case of  $2$ -dimensional extension of a torsion group 
and so the boundary is trivial for  any centered random walk with a finite second moment (by Corollary \ref{cor:rank2rank1}).
\end{exa}

\begin{exa} \label{exa:g23x}
Consider the subgroup $G_{2,3,x}$ of $Gl_2(\Q(X))$ generated by 
$$
\left( \begin{array}{ccc}

1 & 1  \\

0 & 1
 \end{array} \right), 
\medspace \medspace
\left( \begin{array}{ccc}

1 & 0  \\

0 & 2\\
 \end{array} \right), 
\medspace \medspace
\left( \begin{array}{ccc}

1 & 0  \\

0 & 3\\
 \end{array} \right),
\medspace \medspace
\left( \begin{array}{ccc}
1 & 0  \\

0 & X\\
 \end{array} \right).
$$

The dimension is $1$, so any finite second moment centered random walk on this group has a trivial boundary.
\end{exa}

To see why the dimension is $1$ note that the 1-dimensional lamplighter is a subgroup of our group, and the corresponding module over $\mathbb{Q}$ is the same for our group and for the lamplighter.







\section{Non-triviality of the boundary criteron}
\label{sec:nontrivial}

In this section we consider 
a  group $\HAKquotient$ acting on an abelian group $\HNMsubgroup$.
A particular 
case of the statement we will prove is that  if
$\HNMsubgroup$ is a torsion group,  
of dimension $\ge 3$ or in the general case when the dimension of $H$ is $\ge 4$,  
then the Poisson boundary is non-trivial for any finite entropy non-degenerate 
measure $\mu$ on $G$.
Observe that for torsion abelian ($\HNMsubgroup$) by a virtually abelian group ($\HAKquotient$) and simple
random walks (and for all non-degenerate finite second moment random walks) we therefore have  a complete classification (if the dimension is $\le 2$ then the boundary is trivial and if the dimension is $\ge 3$ the boundary is non-trivial).

\subsection{A criterion for the non-triviality of the boundary}

When we have an abelian group $\HNMsubgroup$, we say that a set of elements of $\HNMsubgroup$ is linearly
 independent over $k=\mathbb{Z}/p\mathbb{Z}$  or over $k=\mathbb{Q}$
if their images in $\HNMsubgroup \otimes k$ are linearly independent.

\begin{thm}\label{prop:rank3}
Suppose we have an extension $ 1\to \HNMsubgroup \to G\to \HAKquotient\to 1$.
$\HNMsubgroup$ is abelian. 
Let $\A'=\mathbb{Z}^3$ be
a central subgroup of $\HAKquotient$.
Assume that for some $b\in \HNMsubgroup$
the following holds : the elements of $\A' b \in \HNMsubgroup$
are 
linearly independent, either over $k=\mathbb{Z}/p \mathbb{Z}$
or over $k=\mathbb{Q}$.
Then for any non-degenerate finite entropy measure $\mu$ on  $G$, the Poisson boundary is non-trivial, and
moreover $b$ acts non-trivially on the boundary.
\end{thm}

Since $\HAKquotient$ acts on $\HNMsubgroup$, its subgroup $\A'$ also acts on $\HNMsubgroup$.
In the formulation of the theorem above, $\A' b$ denotes the orbit of $b$ for the action of $\A'$ on $\HNMsubgroup$.


Before we start the proof of the theorem, we formulate some auxiliary statements.
The lemma below is formulated for any field $k$; we will apply it for $k$ being a field of $p$ elements or $k=
\mathbb{Q}$.

Let $\HNMsubgroup$ be a vector space over a field $k$. Assume that
$1 \to \HNMsubgroup \to G \to \HAKquotient \to 1$ is a short exact sequence of groups, and let $\A'$ be a subgroup of $\HAKquotient$.
We consider the action of $\A'$ by conjugation on $\HNMsubgroup$.
We say that an element $b$ in $\HNMsubgroup$ is {\it small with respect to $\A'$} (and $\HAKquotient$), if  
the elements of the orbit $\A' b$ are not linearly independent over $k$. 
Observe that $b$ being small is equivalent to the annihilator of $b$ in $k[\A']$ being non-trivial.

If the short exact sequence is clear from the context, we will say that $G$ admits
small elements with respect to $\A'$ if there exists an element of $\HNMsubgroup$, not equal to $e_{\HNMsubgroup}$, which is small with respect to $\A'$.

\begin{lem} [Reduction to $\HNMsubgroup$ without small elements]\label{lem:small_elements}

Let $\A'=\mathbb{Z}^d$.
$\HNMsubgroup$ is a vector space over a field $k$. Assume that
$1 \to \HNMsubgroup \to G \to \HAKquotient \to 1$ is a short exact sequence of groups, let  $\A'$ be a  normal subgroup in $\HAKquotient$.

\begin{enumerate}
\item
The small elements of $\HNMsubgroup$ with respect to $\A'$ form a normal subgroup in $G$, which we denote by $\HNMsubgroup_{\rm small}$. 
\item
The quotient of $G/{\HNMsubgroup_{\rm small}}$ does not admit small elements
\end{enumerate}
\end{lem}

\begin{proof}

1) Observe that $\HNMsubgroup$ is a $k[\A']$-module.
Below we use additive notation for the group multiplication in our abelian group $\HNMsubgroup$. 
We first prove that the small elements form a group. Consider $b, c \in \HNMsubgroup$ with $b$ and $c$ both small with respect to $\A'$. 

Observe that 
the annihilator in $k[\A']$ of $b\in \HNMsubgroup$ is equal to the annihilator of $-b$. Thus the set of small elements is closed under taking inverses. 

Now we prove that the sum of small elements is small.
Observe that the
annihilator of $b+c$ contains the intersection of the  annihilator of $b$ with the annihilator of $c$. By the smallness of $b$ and $c$ we know that annihilators of $b$ and $c$ are non-zero.
Since $k[\A']$ is an integral domain, we see that this intersection is also a non-zero ideal,
and thus $b+c$ is small.
Suppose $r$ is an element of the annihilator of $b$. Then $frf^{-1}$, for $f\in \HAKquotient$,
(which belongs to  $k[\A']$ since $\A'$ is normal in $\HAKquotient$)
is in the annihilator of $fbf^{-1}$. 
Hence a conjugate of a small element is small and thus the subgroup is normal.

2) We now prove that if an element $b\in \HNMsubgroup$ admits a non-zero element $r\in k[\A']$ such that $rb$ is small in $G$, then $b$ itself must be small in $G$. (And this will imply the second claim of the lemma).
If $rb$ is small, then it is annihilated by some non-trivial element $s\in k[\A']$. Thus $srb=0$ and so $sr$ is an annihilator of $b$. Since $s,r$ are non-trivial and $k[\A']$ is an integral domain, $sr$ is non-trivial and so $b$ is small.
\end{proof}

Given a total preorder $w$, we say that elements $b_1$ and $b_2$ are equivalent if
\newline
 $b_1 \le_{w} b_2$ and $b_2 \le_w b_1$.

\begin{lem}[If no small elements, then there is an order with three properties] \label{lem:order123} Let $k$ be a field. Consider a short exact sequence of groups
$$
1 \to \HNMsubgroup \to G \to \HAKquotient \to 1.
$$
$\A'=\mathbb{Z}^d$ is a normal subgroup of $\HAKquotient$.
 We assume that $\HNMsubgroup$ does not have non-zero small elements with respect to $\A'$.
There exists a total preorder $w$ on $\HNMsubgroup$ satisfying
\begin{enumerate}
\item
Any $\A'$ orbit of $\HNMsubgroup$ does not contain equivalent elements with respect to this preorder.
\item
The total preorder $w$  is $k [\A']$ invariant.
\item 
If for some subset of $\HNMsubgroup$ its elements are pairwise not equivalent with respect to $w$, then they are linearly
independent over $k$.
\end{enumerate}
\end{lem}

\begin{proof}
Fix a left invariant order $w$ on $\A'$. This induces a left invariant preorder on $k[\A']$ where we say $r\ge r'$ if the non-zero  monomial of the largest degree
with respect to $\A'$ of $r$ has larger degree than the corresponding monomial for $r'$.
We slightly abuse the notation  and also call this order $w$. We also will use the same notation $w$ for the orders on $\HNMsubgroup$ and $\HNMsubgroup_S$ which will be defined below.

Note again that $\HNMsubgroup$ is a $k[\A']$ module. 
Let $S$ be a maximal $k[\A']$ linearly independent subset of $\HNMsubgroup$. Choose an arbitrary total order
on $S$. Let $\HNMsubgroup_S$ be the $k [\A']$ module generated by $S$.  Consider two elements $x=\sum_{i} r^x_i s^x_i $,
$y = \sum_{i} r^y_i s^y_i$, where $r_i^x$, $r_i^y$ are non-zero elements in $k [\A']$, and where $s_i^x, s_i^y \in S$. 
In $\HNMsubgroup_S$ we say $x\leq_w y$ if 
either the largest element  of $s_i^x$ with non-zero coefficient is strictly larger than any element of
$s_i^y$ with non-zero coefficient; or the largest elements of
$s_i^x$ and $s_i^y$ are equal, and the following holds.
 Let $s_j$ be this largest element of $S$.  We require that $r_j^x \ge r_j^y$ with respect to the order on $k[\A']$. 
We have defined $w$ on $\HNMsubgroup_S$, now we want to extend $w$ to $\HNMsubgroup$.
To do this, we first check that $w$ defined on $\HNMsubgroup_S$ has 1), 2), 3) of 
the lemma (for this we only need, as in the definition of $w$ on $\HNMsubgroup_S$, that $A$ is orderable).

To prove 1) for $w$ on $\HNMsubgroup_S$ observe, that if 
$x=\sum_i r^x_i s^x_i$, then
$hx= h \sum_i r^x_i s^x_i= \sum_i hr^x_i s^x_i $, hence these two elements $x$
and $hx$ have the same maximal element $s_j$ with at least one non-zero coefficient  in this decomposition, and the maximal coefficients of this $s_j$ are not equal.

To see 2) observe the following. Let
$x=\sum_{i} r^x_i s^x_i$ and $y=\sum_{i} r^y_i s^y_i$. Assume that $x \ge_w y $. Then either the largest element with non-zero coefficient of $x$ is strictly larger than that of $y$, in which case, as we have just mentioned, these largest elements do not change by the action of $h$, and so $hx \ge _w hy$.
Or, $x$ and $y$ have the same largest elements $s_j^x= s_j^y$, and we have $r^x_j
\ge_w r_j^y$, then $h r^x_j
\ge_w h r_j^y$ (since $w$
is a left invariant order on $k[\A'])$. Conversely, if $hx > hy$ then either $hx$ must have a strictly larger element with non-zero coefficient than $hy$, in which case the same is true of $x$ and $y$, or $hx$ and $hy$ have the same largest element, in which case $h r^x_j\ge_w h r_j^y$, which implies $r^x_j\ge_w r_j^y$.  

Finally, 3) holds for $w$ on $\HNMsubgroup_S$:
let $T$ be a set of elements in $\HNMsubgroup_S$ which have a $k$-linear relation, with non-zero coefficients for each $t\in T$. It suffices to show that there must exist  $t_1 , t_2 \in T$ which are $w$-equivalent. 
Indeed, since there is a linear relation for elements of $T$,
note that at least two elements of $t_1, t_2 \in T$ have the same maximal  elements $s=s^{t_1}_i = s^{t_2}_i$. Also note that among those with this maximal element equal to $s$, there must be two ($t_1$ and $t_2$) with equivalent $r^{t}_i$ in $k[A]$, $t=t_1$ and $t_2$.

We now extend the preorder from $\HNMsubgroup_S$ to $\HNMsubgroup$.
Given two elements $b_1,b_2 \in \HNMsubgroup$ we say that $b_1\leq_w b_2$ if there exists some non-zero $r\in k[\A']$ such that $rb_1,rb_2\in \HNMsubgroup_S$ and $rb_1<_w rb_2$. 
First note that by the maximality of $S$, for every $b\in \HNMsubgroup$ there exists some $r$ such that $rb\in \HNMsubgroup_S$. 

We now argue that this induced relation is defined for every pair $b_1,b_2$. Let $r_1,r_2$ be such that $r_1b_1,r_2b_2\in \HNMsubgroup_S$ then $r_1r_2b_1,r_1r_2b_2$ are also in $H_S$ and, since $k[\A']$ is an integral domain, $r_1r_2$ is nonzero.  From the totalness of the preorder on $\HNMsubgroup_S$, it follows that 
$r_1r_2b_1,r_1r_2b_2$ are comparable. Hence, the relation is total on $\HNMsubgroup$.

We now explain that this relation is indeed a preorder.  Note that if $b_1<_w b_2<_w b_3$, then there exist $r_1,r_2$ so that $r_1b_1,r_1b_2\in \HNMsubgroup_S$ and $r_2b_2,r_2b_3$ are in $\HNMsubgroup_S$
In this case the elements 
$r_1r_2b_1$ ,$r_1r_2b_2$, $r_1r_2b_3$ are all in $\HNMsubgroup_S$ and by the multiplication invariance of the order on $\HNMsubgroup_S$ it holds $r_1r_2b_1<_w r_1r_2b_2<r_1r_2b_3$ and thus $b_1<_w b_3$. Therefore, the induced relation is a preorder.

We now note that the restriction of this new preorder to $\HNMsubgroup_S$ induces the original order $w$ on $\HNMsubgroup_S$. This follows from property 2) of the order on $\HNMsubgroup_S$. First note that if $b_1<b_2$ in the original pre-order it is obviously still so in the new preorder. 
Also observe that if for $b_1,b_2\in \HNMsubgroup_S$ there exist $r$ such that $rb_1<_{w}rb_2$, then by property 2 above we have $b_1<_{w}b_2$ and thus, again by property 2, for all $r'$ we have $r'b_1<_{w} r'b_2$.


Now let us show that since the order on $\HNMsubgroup_S$ has 1), 2), 3), then the order on $\HNMsubgroup$ also has these properties.
First we observe that 2) holds:
2) If $b_1, b_2 \in \HNMsubgroup$ and $b_1 \le_w b_2$, then  $rb_1 \le_w rb_2$. Indeed,  by definition $b_1 \le_w b_2$  means that there exists $r_1$ such that $r_1b_1,r_1b_2\in \HNMsubgroup_S$ and $r_1b_1<_w r_1b_2$. Then by the property 2) on the order on $\HNMsubgroup_S$ we have that $r_1rb_1\le_w r_1rb_2$. If $b_1\nleq b_2$ then, by definition, no non-zero multiple of them is. 

Now we check properties 1) and 3).
1): if we have two equivalent elements of an $\A'=\mathbb{Z}^d$ orbit, 
then multiplying them from the left we get (by property 2) two equivalent elements in $\HNMsubgroup_S$ which are also in an $\A'$ orbit. 

3) 
If we have a subset of elements of $\HNMsubgroup$ which are pairwise not equivalent with respect to $w$, then they are linearly
independent over $k$. Otherwise, we multiply them on the left and get a linear dependence in $\HNMsubgroup_S$. 
Observe that by the construction of the extended order, if $rb_1,rb_2$ are not equivalent,  for a non-zero element $r$  in $k[\A']$,  then $b_1,b_2$ are not equivalent. Hence, such linear dependence in $\HNMsubgroup_S$ would be a contradiction.


\end{proof}

\begin{defn}\label{def:recurrentsequences} [Recurrent/transient sequences]

Given a group $\Gamma$, a measure $\mu$ on $\Gamma$. Fix a  sequence $\gamma_1, \gamma_2 \dots \gamma_n \dots$ and  consider the following event. Let $X_1$, $X_2$, \dots , $X_n$, \dots be a  trajectory of the random walk $(\Gamma, \mu)$. We consider the probability that 
there  exist infinitely many  $i$ such that $\gamma_i=X_i$. If this happens with probability $1$,  we say that the sequence $(\gamma_i, i)$ is a recurrent sequence for $(\Gamma, \mu)$. If all sequences are transient for the random walk $(\Gamma,\mu)$, we say the random walk $(\Gamma, \mu)$ is strongly transient. 
\end{defn}

\begin{defn}\label{def:uniformlysr}[Uniform strong transience] 
We say that a random walk 
on $\Gamma$ is {\it uniformly strongly transient}, 
if 
for any  sequence $\gamma_i \in \Gamma$, $i \in \mathbb{N}$,
the expected number of {\it prescribed} hits ($X_i=\gamma_i$) is finite.

\end{defn}

\begin{rem}
If a random walk is uniformly strongly transient, then there exists a positive constant $C$ (not depending on the sequence $\gamma_i$)
such that this expectation is bounded by $C$.
\end{rem}

Indeed, this expectation is the sum of expectations $E_i$ to have a prescribed hit at time $i$. For each
$i$ consider the supremum of $\mu^{* i}(\gamma)$. 
(This is the expectation to have a prescribed hit at time $i$ if $\gamma_i= \gamma$). Observe that this supremum is indeed a maximum, and choose $\gamma_i$ to be an element that realizes this maximum.
If the sum is not summable, then for this sequence $\gamma_i$ the expectation would not be finite.

\begin{lem}[Main lemma] \label{lem:mainlemma}
Let $G$ be a semi-direct product, corresponding to a short exact sequence
$1\to \HNMsubgroup \to G\to \HAKquotient\to 1$. Let $\A'$ be a torsion-free abelian group which is central in $\HAKquotient$. Assume that $\HNMsubgroup$ is a vector space over a field $k$.
Let $b\in \HNMsubgroup$ be a non $\A'$-small element.
Suppose $\mu$ is a finite entropy
measure on $G$.
Consider the normalized restriction (of its projection to $\HAKquotient$)
to $A$ :
$
\mu_{\A'}= \frac{\mu_{\HAKquotient}|\A'}{\mu_K(\A')}$.
We  assume that 
the random walks $(A, \mu_{\A'})$ are uniformly strongly transient. Then  $b$ acts non-trivially on the Poisson boundary. 
\end{lem}
\begin{proof} 
First observe that we can assume that
$b \in S^{+}S^{-}$, where $S$ is the sub-semigroup generated by the support of $\mu$, and $S^-= \{s^{-1}, s\in S\}$.
Indeed, for any measure violating this assumption, on any group, the element $b$ acts non-trivially on the boundary (see e.g. \cite{erschlerkaimanovich2019}, Remark $5.4$; for more general statements see \cite{kaimanovich92}, Lemma $2.9$).
We can therefore assume that there exists $g \in G$, such that $g \in \supp \mu^{*k}$ and $gb \in \supp \mu^{*l}$ for some $k, l\ge 1$.

Observe also, that taking, if necessary, a convex combination with the atomic measure supported on $e$ (this change does not change the Poisson boundary), we can assume that $e \in \supp \mu$.
We conclude in this case that there exists $g\in G$ and $k\ge 1$ such that
$g, gb \in \supp \mu^{*k}$. Replace, if necessary, $\mu$ by $\mu^{*k}$ and assume that $g, gb \in \supp \mu$.

Now observe that we  can assume that $B\HNMsubgroup$  does not have $\A'$-small elements (if it is not the case, we know by Lemma \ref{lem:small_elements} that small elements form a normal subgroup, and that if we quotient $G$ over this subgroup the image of not $A$-small element $b$ remains not $\A'$-small. And this quotient does not admit $\A'$-small element.
Observe that then if $b$ acts non-trivially on the Poisson on the quotient group, then it acts non-trivially on the Poisson boundary
of $(G,\mu)$.

So now we assume that $\HNMsubgroup$ does not have  $\A'$-small elements.
Then on 
 $\HNMsubgroup$ there is an order with 1), 2), 3), which we denote $w$ (by Lemma \ref{lem:order123}).

Put $\Delta = \{g, gb\}$. We want to prove that $\Delta$-restriction entropy (see Definition \ref{def:deltarestrictionentropy})
is positive.

Consider a random walk $X_1$, \dots, $X_n$ on $G$
and its projection $A_1$, \dots, $A_n$ to $\HAKquotient$.
We fix all increments that are not in $A$ (this leaves a positive proportion of elements, since $g,gb \in A$ and these elements belong to the support of $\mu$).
We fix then also all increments $\ne g, gb$.
Consider corresponding times instants $t_1$, \dots, $t_k$ where multiplication by elements of $\Delta$ occurs.
Let $Z_i$ be the product of all increments not in $\A'$ until time $t_i$ and $W_i$ be the product of all increments in $H$ until time $t_i$.
Consider an equivalence relation on $\A' \times \mathbb{N}$, saying that $(h,k)$ is equivalent to 
$(f,l)$ if $Z_k  h g b g^{-1} h^{-1} Z_{k^{-1}}$ is $w$-equivalent to
$Z_l  f g b g^{-1} f^{-1} Z_{l^{-1}}$. Observe that property 1) of the order $w$ implies that the equivalence relation satisfies the condition $(x, k)$ is not equivalent to  $(y,k)$ whenever $x\ne y$, $x,y \in A$.
(this condition will be called "admissible" in the Appendix).
By the Law of Large Numbers, we know that the number of increments $T_n$ by elements of $\A'$
until time instant $n$ is approximately ${\rm Const} n$.


By Lemma  \ref{lem:linearrange} proven in the Appendix, we know that the range function with respect to the equivalence relation $E$ is linear in $T_n$ (since this relation is admissible), and hence in $n$. 

We conclude that with positive probability there are at least $p n$ linearly independent elements of $\HNMsubgroup$, corresponding to possible increments $gb$ (rather than $g$) in the places where multiplication by these two elements occurs.
This shows that $\Delta$-restriction entropy is linear. 
By Lemma \ref{lem:easydirection} we conclude that the action of $b$ on the boundary is non-trivial.
\end{proof}

\begin{cor} \label{cor: Z3central}
Take an extension as in Lemma \ref{lem:mainlemma}
and assume that the central subgroup of $A$ is $\mathbb{Z}^3$. Then the claim of this lemma holds for any non-degenerate finite entropy measure on $G$.

\end{cor}
\begin{proof}
By Claim $1$ of Theorem \ref{thm: admisrecurrent} (proven in the Appendix) we know that any non-degenerate random walk on $\mathbb{Z}^3$ is uniformly strongly transient. Hence, replacing if necessary the measure by its convolution power (so that its restriction to  $A$ generates $A$),  we can apply Lemma \ref{lem:mainlemma}.
\end{proof}

Now we are ready to prove the theorem of this section.

\begin{proof}
By the assumptions of the theorem we have a short exact sequence
$ 1\to \HNMsubgroup\to G\to \HAKquotient\to 1$,
where $\HNMsubgroup$ is abelian. 
$\A'=\mathbb{Z}^3$ is
a central subgroup of $\HAKquotient$.
By assumption of the Theorem we also know that 
$b\in \HNMsubgroup$
 is such that
the elements of $\A' b $
are linearly independent, either over $k=\mathbb{Z}/p\mathbb{Z}$  or over $k=\mathbb{Q}$.  By definitions of smallness this says that 
that $b$ is not small.
We need to prove that $b$ acts non-trivially on the Poisson boundary of $(G,\mu)$ for any finite entropy non-degenerate measure $\mu$.

Replacing if necessary $G$ by a semi-direct product $G_2$ with the same action of $\HAKquotient$ on $\HNMsubgroup$, since $\mu$ is non-degenerate, 
we can use Theorem \ref{thm:compcrit} and see that it is sufficient to prove our theorem
for semi-direct products.
Observe that if $b$ is small for $G$, then it is small for $G_2$ (since the definition depends only on the action of $\HAKquotient$ on $\HNMsubgroup$). We use Lemma \ref{lem:reductiontosingle}
for $Q=\HNMsubgroup$, $G=G$, $K=\HAKquotient$, $q=b$ and conclude that it is sufficient to assume
that $\HNMsubgroup$ is generated as a normal subgroup by $b$ in our group
$G=G_q$.
First assume that  $b$ is $p$-torsion, $p$ is prime. Then $\HNMsubgroup$ is a vector space over $\mathbb{Z}/p \mathbb{Z}$.
Then $G$, $b$ and $\mu$ satisfy the assumption of Lemma \ref{lem:mainlemma} and Corollary \ref{cor: Z3central}. Thus by the latter corollary, we conclude that $b$
acts non-trivially on the boundary.

Now assume as in the previous case that $\A' b$ are linearly independent over $\mathbb{Z}/p\mathbb{Z}$. Consider $B'=\HNMsubgroup/p \HNMsubgroup$ and a short exact sequence
$$
1 \to B' \to G_2 \to \HAKquotient \to 1,
$$
where $G_2= G/p \HNMsubgroup$. 
If the image of $b$ acts non-trivially on the boundary of the projected r.w. to $G_2$, then $b$ acts non-trivially on the boundary of $G$. The image of $b$ is of order $p$, so by the previous case we know that the action of $b'$ (and hence of $b$) on the boundary is non-trivially.

Now we assume that the elements $\A' b$ are linearly independent over $\mathbb{Q}$.
Let $T_{\HNMsubgroup}$ be the torsion subgroup of $\HNMsubgroup$.
Consider 
$$
1 \to \HNMsubgroup/T_{\HNMsubgroup} \to G/T_{\HNMsubgroup} \to \HAKquotient \to 1
$$
Let $b'$ be the image of $b$ in $T_{\HNMsubgroup}$.
Observe that  $\A' b'$ are linearly independent over $\mathbb{Q}$, so without loss of generality we can assume that $T_{\HNMsubgroup}= \{e\}$.
Consider $B_2 = B \otimes \mathbb{Q}$, $\HNMsubgroup$ is a subgroup of $B_2$, $b \in B \in B'$.
There exists a semi-direct product $G_2$, 
$$
1 \to B_2 \to G_2 \to K \to 1 
$$
where the action of $K $ on $B_2$, restricted to $\HNMsubgroup$ is the action we had for our previous short exact sequence.
Consider a non-degenerate measure $\mu_2$ on $G_2$, with the same projection to $K$ as $\mu$. Applying Lemma \ref{lem:mainlemma} for $G_2$, $\mu_2$ and $b$, we conclude that $b$ acts non-trivially on the boundary of $(G_2, \mu_2)$.
Applying the reduction to single lemma (Lemma \ref{lem:reductiontosingle}) twice: once to $G$ and also to $G_2$, we observe that the reduction groups $G_q = (G_2)_q$ ($q=b$).
Thus we conclude that $b$ acts non-trivially on the boundary of $(G,\mu)$.

\end{proof}

\begin{rem} \label{rem:notsimplytrans}
The assumption of (uniform) strong transience (and not simply transience) is essential in Lemma 
\ref{lem:mainlemma}. Under the assumption of transience
there are examples where the action on the boundary is trivial.
\end{rem}

\begin{proof}
Indeed, take $G= \left( ( \mathbb{Z}/2\mathbb{Z} )  \wr \mathbb{Z} \right)  \times \mathbb{Z}$.
We denote by $r$ the generator of the
cyclic factor above.
We denote by $z$ the generator of the infinite cyclic subgroup of the lamplighter.
The subgroup $B$ is $\mathbb{Z}/2\mathbb{Z}^{\mathbb{Z}}$. 
Let $K$ be the subgroup generated by $z,r$
Let $A$ be the subgroup generated by $z$. 
Consider the measure  on $\mathbb{Z}^2$ with a charge $1/3$ on $z$,  $(1/6)$ on $z^{-1}$,  $1/3$ on $zr$ and $1/6$ on $z^{-1}r^{-1}$ (or any other finitely supported measure on $\mathbb{Z}^2$, which has a symmetric projection on $\mathbb{Z}=\mathbb{Z}^2/<r>$ and such that its restriction on $A= <z>$ has a non-zero mean). 
It is clear that the Poisson boundary of
(the central extension)
$(  \mathbb{Z}/2\mathbb{Z}  \wr \mathbb{Z} ) \times \mathbb{Z},\mu)$
is equal to the boundary of the projection of the random walk to 
$\mathbb{Z}/2\mathbb{Z} \wr \mathbb{Z}$, and any element of the configuration acts trivially on the boundary.
But $(A,\mu_A)$ is transient. 

\end{proof}

Consider a short exact sequence
$$
1\to \HNMsubgroup\to M\to \HAKquotient\to 1,
$$
$\HNMsubgroup$ and $\HAKquotient$ are abelian, $b\in B$.
Consider $\HNMsubgroup$ as a $\HAKquotient [\HNMsubgroup]$ module.
Consider its submodule  generated by $b$. We can speak about the dimension of $b$ defined as the dimension of this submodule
(considered as $\HAKquotient[B_2]$-module, where  $B_2$ is a subgroup of $\HNMsubgroup$ generated by $b$).

\begin{cor} \label{cor:metablockrank3}
\begin{enumerate}
\item
Consider a metabelian basic block $M$ of a nilpotent-by-abelian linear group a
with  short exact sequence 
$$
1\to \HNMsubgroup\to M\to \HAKquotient\to 1
$$
By our convention  such basic blocks are either torsion-free or ($p$-torsion)-by-abelian.
If  $b\in \HNMsubgroup$ has dimension $\ge 4$ or if $b$ has dimension $\ge 3$ and $\HNMsubgroup$ is $p$-torsion, 
then $b$ acts non-trivially on the Poisson boundary. 


In particular, if $M$ has dimension  $\ge 4$,
 or if $M$ has dimension $\ge 3$ and $\HNMsubgroup$ is a $p$-torsion group,
then the Poisson boundary of any non-degenerate finite entropy measure is non-trivial.

\item 
Moreover, let $G$ be an upper triangular linear group with a valid block of dimension $\ge 4$, or of dimension $\ge 3$ in case when $H$ is $p$-torsion. Then the Poisson boundary of any non-degenerate finite entropy measure is non-trivial. 
\end{enumerate}
\end{cor}
\begin{proof}




By Lemma \ref{newclaim:freesubmodule}
we know that if $M$ or $b$ has dimension 
$\ge 3$, then there exists $b'$ and a subgroup $A=\mathbb{Z}^3$ of $K$ such that the module generated by $b'$ is freely generated as a $k[A]$ module . The first statement follows therefore
from Theorem \ref{prop:rank3}. The second statement then follows from \ref{thm:lineargroups}.

\end{proof}



\subsection{Examples with non-trivial boundary and further applications to linear  and metabelian groups}
First we recall that in Example \ref{exa:twobytworank} we gave a sufficient condition for two-by-two upper triangular matrices to have dimension at least $3$ , and thus it follows from Corollary \ref{cor:metablockrank3}   that the non-degenerate finite entropy random walks on groups in Example \ref{exa:twobytworank}
have non-trivial Poisson boundary.

Below we discuss examples of dimension $3$ torsion-free metabelian groups  and measures of non-trivial boundary, which on some of such groups can be chosen finitely supported.

\begin{exa}[Lamplighter Baumslag-Solitar groups $d=1$ and $d=2$.]\label{exa:m1m2mxmy}
Recall our notation for matrices
\[ \theta= \left( \begin{array}{ccc}
1 & 1 \\
0  & 1
 \end{array} \right), \medspace \medspace 
 M_2= \left( \begin{array}{ccc}
1 & 0 \\
0  & 2
 \end{array} \right), 
 \] 
 
 \[
 M_x= \left( \begin{array}{ccc}
1 & 0 \\
0  & X
 \end{array} \right), 
 \medspace \medspace 
 M_y= \left( \begin{array}{ccc}
1 & 0 \\
0  & Y
 \end{array} \right) 
 \]
The group $G_{x,2}$ (d=1) generated by $M_x$, $M_2$ and $\theta$ has non-trivial boundary for some finite first moment symmetric random walks.

The group $G_{x,y,2}$ (d=2), generated by $M_x$, $M_y$, $M_2$ and $\theta$
has non-trivial boundary for any finite entropy non-degenerate measure.
\end{exa}
\begin{proof}
Possibly by taking a convolutional power we may assume that $\mu$ contains $\theta$ and the identity $e$ in its support. Note that we have a canonical homomorphism  from $G_{x,2}$
to $\mathbb{Z}^2$ and from $G_{x,y,2}$
to $\mathbb{Z}^3$ given by mapping $M_x, M_y, M_2$ to distinct generators and sending $\theta$ to the identity.
Put $\Delta= \{e, \theta\}$.
Observe that $\Delta$-restriction entropy \ref{def:deltarestrictionentropy} grows linearly.

Indeed, to see this for $d=1$ observe 
$$
\sum_{j,k} \varepsilon_{j,k} 2^k x^j \ne 0,
$$
and that for $d=2$ observe that
$$
\sum_{i,j,k} \varepsilon_{i,j,k} 2^k x^i y^j \ne 0,
$$
if $\varepsilon_{j,k},\varepsilon_{i,j,k}$ take value $0$, $1$ and $-1$ and at least one of the coefficients is $\ne 0$. Here the sum is taken over integers $i$, $j$
 and $k$.
 
For our application for the group with $d=1$ we use  
the existence of finite first moment  symmetric random walks with transient projection to $\mathbb{Z}^2$.
Indeed, it is well-known that such measures exist
and can be chosen to have finite $2-\varepsilon$
moment (for example, we can choose them of the form $\nu \times \nu$, where $\nu$ is a symmetric measure in the domain of the attraction of a Stable Law, and apply the Local Limit theorem for $\nu$, see 
e.g. \cite{christoph}.

For our application for $d=2$ we remind that
the projected random walk on $\mathbb{Z}^3$ is transient
for any non-degenerate random walk (see e.g. \cite{spitzer}).  

Hence
the  number of visited points in $\Z^2$ (correspondingly in $\mathbb{Z}^3$)
grows linearly (in the number of steps, in other words the range function grows linearly,  see e.g.  Theorem $1.4.1$ Spitzer \cite{spitzer}). 
If we condition on all $\Delta$ increments except for the final multiplication at $X_i$ with a projection to given points in
$\mathbb{Z}^2$ or correspondingly  $\mathbb{Z}^3$,
we see that the $\Delta$-restriction entropy is linear and thus the boundary is nontrivial by \ref{lem:easydirection}.

Also observe that the dimension of our groups is equal to $d+1$.
Indeed, tensor by $\mathbb{Q}$ and observe that the group becomes an extension of an abelian group by $\mathbb{Z}$,  correspondingly by  $\mathbb{Z}^2$ (generated by the images of $M_x$; and by the images of $M_x$ and $M_y$).  In view of monotonicity of the dimension for taking  subgroups,
the lower bound follows from the existence of $1$-dimensional and  $2$-dimensional lamplighters in the groups.


\end{proof}.

More generally, we consider
\[
M_\alpha= \left( \begin{array}{ccc}
1 & 0 \\
0  & \alpha 
 \end{array} \right) 
 \] 
and the group $G_\alpha$, generated by $M_1$, $M_\alpha$, $M_x$, $M_y$.

Claim $2$ and $3$ of  Proposition \ref{prop:Galpha} below classify $G_\alpha$ in terms of triviality/non-triviality of the boundary,  in particular providing further examples of $3$-dimensional torsion-free groups with non-trivial boundary. This happens whenever $\alpha$ is algebraic and not a root of unity.

\begin{prop}\label{prop:Galpha}
\begin{enumerate}

\item If $\alpha$ is algebraic, the group $G_\alpha$ has dimension $2$. If not, the group has dimension $3$.

\item 

If $\alpha$ is a root of unity, then $G_\alpha$
has a finite index subgroup which is a quotient of $\Z^l \wr \Z^2$.
(In this case  by \cite{kaimanovichvershik} we know the boundary of any centered second moment random walk is trivial).

\item  If $\alpha$ is not the root of unity, then
any finite entropy irreducible random walk on
$M_\alpha$ has non-trivial boundary.

\end{enumerate}
\end{prop}

\begin{proof}
1) Follows from Lemma \ref{exa:twobytworank}



2) Now assume that $\alpha $ is a root of unity of degree $d$, $\alpha^d=1$. Consider the projection to $\Z \times \Z \times \Z/d\Z$. Consider the finite index subgroup of $\Z \times \Z \times \Z/d\Z$
with the third coordinate equal to $0$. Its preimage is a finite index subgroup in $G_\alpha$. Observe that the numbers $\alpha^s$, where $1\le s <d$ are integers relatively prime with $d$, are linearly independent (and other powers $\alpha^s$ are expressed as their linear combination).
We deduce from this that this finite index subgroup is isomorphic to $\Z^{\phi(d)} \wr \Z^2$, where $\phi(d)$ is Euler's totient function.

3) Observe that we can assume that $\alpha$ is algebraic
since
otherwise
$G_{\alpha}$ has dimension 3. 
Since $\alpha$ is not a root of unity, there is some Galois conjugation $\alpha'$ of $\alpha$ which has absolute value not equal to $1$. Observe that $G_\alpha$ is isomorphic to $G_{\alpha'}$. Note also that $G_{\alpha^{-1}}$ is isomorphic to $G_\alpha$. So we can assume that the absolute value of $\alpha>1$.
Take $m$: $|\alpha^m|>2$.
We recall that for any group
if $(G,\mu)$ does not have $S^+S^{-}$ property the Poisson boundary of $(G,\mu)$ is non-trivial, see e.g. \cite{erschlerkaimanovich2019}, Remark $5.4$.
Hence, we can assume that $(G,\mu)$ satisfy $S^+S^-$ property.
First observe that if $M$ is non-degenerate
, there is $k$
such that $M_{\alpha^m}$, $e$ are in the support of $\mu^{*k}$. We put $\Delta=\{e, M_1 \}$ and similarly to Example \ref{exa:m1m2mxmy} observe that
$$
\sum_{i,j, k} \epsilon_{i,j,k} (\alpha^m)^k x^i y^j \ne 0,
$$  
if $\epsilon_{i,j,k}$ takes value $0$, $1$ and $-1$ and at least one of the coefficients is $\ne 0$.
Observe that points whose projection to $Z^3$ are of the form $(mk, \cdot , \cdot)$ (for $k\in \mathbb{Z}$)  form a finite index subgroup (a lattice)
in $\mathbb{Z}^3$. Observe that then the number of points after $n$ steps visited in this lattice grows linearly. (Indeed, with positive probability this number of points is at least ${\rm Const} R_n $, where $R_n$ is the number of distinct points in $\mathbb{Z}^3$, visited up to time instant $n$. And $R_n$ is linear in $n$ since the random walk on $\mathbb{Z}^3$
is transient (for this basic fact see e.g. Theorem 1.4.1 in \cite{spitzer}).
Then we condition 
 on all $\Delta$ increments except for the final multiplication at $X_i$ with a projection to given points of our lattice,
we can conclude that $\Delta$-restriction entropy is linear.

This implies that $\Delta$-restriction entropy is linear, and therefore the entropy function grows linearly, and the boundary is non-trivial.

In general, without loss of generality, we can assume that $e \in \supp \mu$.
Since $\mu$ satisfies $S^+S^-$ property, there is $r$ such that $g_1, g_2 \in \supp \mu^{*r}$, and $g_1 = M_{\alpha^m} g_2$.
We can put $\Delta =\{g_1, g_2\}$ and show that
$\Delta$ -restriction entropy grows linearly.

\end{proof}



\begin{rem}
In Corollary \ref{cor:variablesfield} we have seen 
that for a virtually solvable group, linear over 
a function field over at most $1$ variable, any finite second moment centered measure defines a random walk with trivial
boundary.
The Example \ref{exa:m1m2mxmy} shows that in the case of characteristic  $0$ the assumption of the corollary that the transcendence degree  $\le 1$ cannot be replaced by $\le 2$. 
\end{rem}

\begin{rem}
However, observe that some groups, for example $\mathbb{Z} \wr \mathbb{Z}^2$,
have trivial boundary for any centered second moment random walk \cite{kaimanovichvershik}.  But this group had Krull dimension $3$ and thus cannot be embedded into the linear group of transcendence degree $1$ in view of Lemma \ref{exa:twobytworank}.


\end{rem}
As we have mentioned in the introduction, we conjecture that this is essentially the only  $2$-dimensional example of a linear basic block with trivial boundary.

\subsection{More on groups of characteristic $p$}

Now we are ready to prove the Corollary $D$ formulated in the introduction:

\begin{cor} \label{cor:1234charp}
Let $G$ be an amenable subgroup of $GL(n,k)$, $k$ is a field of characteristic $p$.
The following properties are equivalent
\begin{enumerate}
\item
$G$ is commensurable to a group with a basic block $B_{i,j}$ which contains
$ \mathbb{Z}/p\mathbb{Z}\wr \mathbb{Z}^3$ as a subgroup.

\item The dimension of some block of 
the group is $\ge 3$.

\item
There exist a simple
random walk on $G$ with non-trivial Poisson boundary.
\item
All non-degenerate random walks 
of finite entropy on $G$ have non-trivial Poisson boundary.

\end{enumerate}
\end{cor}

\begin{proof}
Passing to a finite index subgroup, we can assume that $G$ is upper-triangular, as we will discuss below.
First let us show that (2), (3) and (4) are equivalent. We know

(4) obviously implies (3).
Suppose we have (3). Passing to a finite index subgroup and considering exit measure to this subgroup, we observe that this subgroup of upper-triangular matrices admits a non-degenerate symmetric measure, with exponential moments and with non-trivial Poisson boundary. We will only use that the second moment is finite.
By Theorem \ref{thm:lineargroups}, we know that in this case the random walk on at least one block has non-trivial boundary. By the first part of Lemma \ref{cor:rank2rank1} we know that if the dimension is smaller than $3$, then any symmetric finite second moment measure has a trivial Poisson boundary. Thus we conclude that the dimension is  $\ge 3$, in other words (2) holds.
Finally, assume (2).  Using again Theorem \ref{thm:lineargroups} and the definition of the dimension, we know that it is enough to consider the case when the group is a metabelian block (of dimension $\ge 3$).
By Corollary \ref{cor:metablockrank3} we know in this case that any non-degenerate finite entropy measure has non-trivial boundary.
Thus (iv) holds.

It remains to show that (1) (wreath products in the blocks condition) is equivalent to (2) (there exists a block of dimension $\ge 3)$.
First observe that (1) implies that the dimension is $\ge 3$.  We know that the dimension of $ \mathbb{Z}/p \mathbb{Z} \wr \mathbb{Z}^d$ is $d$.
Hence by  Remark \ref{rem:modulesubmodule} about monotonicity 
the dimension of our metabelian block has rank $\ge 3$.

Now observe that Lemma \ref{newclaim:freesubmodule}
implies that if the dimension of a metabelian block is $\ge 3$, then it has $ {\mathbb{Z}/p\mathbb{Z}}
\wr \mathbb{Z}^3$ as a subgroup.  
\end{proof}
 
\begin{rem} \label{rem:withoutwr} Consider a free metabelian group $\rm{Met}_d$ on $d$ generators, and let $\rm{Met}_d(p)$ the free $p$-metabelian group. 
The group $M_3(p)$ does not contain $ \mathbb{Z}/p \mathbb{Z} \wr \mathbb{Z}^3$ as a subgroup. In fact, any abelian subgroup of $M_3$ or $M_3(4)$ belongs to its commutator group, in particular, $M_3(p)$ does not contain $\mathbb{Z}^2$ as a subgroup.
For the background about subgroups of metabelian and free metabelian  groups see \cite{baumslagetal}
and references therein.
Observe also that a proper quotient of $M_3(p)$
can not contain a three dimensional wreath product, since by Lemma \ref{lem:rankdrops} the dimension of a proper quotient is strictly less than $3$.
Thus while  by our results a non-Liouville metabelian group, which is linear in characteristics $p$, admits blocks that contain three dimensional wreath products as a subgroup, the group itself does need to admit such wreath products as a
subgroup or as a quotient.
\end{rem}

\appendix 

\section{Transience and uniform strong transience}

We consider a group $\Gamma$, a measure $\mu$ on $\Gamma$ with trajectories
of the random walk $(\Gamma, \mu)$ denoted by  $X_1$, $X_2$, \dots , $X_n$, \dots.  

Given a sequence $\gamma_i$ of elements of $\Gamma$, we say that it is is a recurrent sequence for $(\Gamma, \mu)$ if with probability one
there  exist infinitely many $i$ such that $X_i=\gamma_i$. If $\gamma_i$ is not recurrent we call it transient. 
If all sequences are transient for the random walk $(\Gamma,\mu)$ we say the random walk $(\Gamma, \mu)$ is strongly transient (as we defined in Definition \ref{def:recurrentsequences}).

We give a note on the terminology: one should not confuse the fact that the sequence $(\gamma_i, i)$ is recurrent with the fact that the set $\{\gamma_i, i \in \mathbb{N} \}$
is recurrent. The recurrence of the sequence corresponds to the fact that the set $\{\gamma_i, i  \})$ is recurrent for the time extended random walk $(X_i,i)$.

We recall that a random walk $(\Gamma,\mu)$ is  uniformly strongly transient (see Definition
\ref{def:uniformlysr}), 
if 
for any  sequence $\gamma_i \in \Gamma$, $i \in \mathbb{N}$, the
expected number of $i$ such that ($X_i=\gamma_i$) is finite.

\begin{rem}\label{rem:unif}
Consider $\delta_i \ge 0$ such that $\sum_{i=1}^\infty \delta_i <\infty$. Assume that
$\mu^{*i}(g) \le \delta_i$ for all $g$. Then
the random walk $(G,\mu)$ is uniformly strongly
transient.
\end{rem}
\begin{proof}
Fix a sequence $\gamma_i \in G$.
We need to prove that the expected number of prescribed hits $X_i=\gamma_i$ is finite. This expectation is the sum of the expectation that a prescribed hit happens at the time instant $i$. This sum is equal to
$$
\sum_i P[X_i=\gamma_i] =\sum_i \mu^{*i}(\gamma_i)\le
\sum_i \delta_i <\infty
$$
\end{proof}

We recall a well-known corollary of the Cauchy Schwarz inequality
\begin{rem} \label{rem:CauchySch}
For a symmetric measure $\mu$ it holds
$\mu^{*2n}(g) \le \mu^{*2n}(e)$
\end{rem}

\begin{lem}\label{rem:oldCauchySch}
If $\nu$ is a symmetric  non-degenerate measure such that the random walk $(\Gamma, \nu)$ is transient then the random walk  $(\Gamma, \nu)$ is uniformly strongly transient. 
\end{lem}

\begin{proof}
Consider  $\mu = \nu^{*2}$. 
The random walk $(\Gamma,\mu)$ is recurrent if and only if for some/all sequences $\gamma_i$ the sequence $\gamma_i$ is recurrent. Indeed, we recalled in the previous remark
that for any symmetric $\nu$ and any $y$ $\nu^{*2n}(y) \le \nu^{*2n}(e)$. Thus we also have that
$$
\nu^{2n+1}(y) =\sum_z(\nu^{*2n}(z)\mu(z^{-1}y)) \le \nu^{2n}(e).
$$
In view of Remark \ref{rem:unif} we conclude that $\nu$ is  uniformly strongly transient.
\end{proof}

\begin{exa}\label{exe:drift}[A transient walk can admit recurrent sequences]
Take a random walk on $\mathbb{Z}$, supported on $-1$ and $1$,
$\mu(1) =\alpha > \mu(-1)=\beta$. The random walk in  transient. Consider $\gamma_n = [(\alpha-\beta) n]$. It is clear that infinitely many times $X_n\ge \gamma_n$, infinitely many times $X_n \le \gamma_n$, and hence since each jump of our random walk has absolute value one, we can not have for some $n$ $X_n < \gamma_n$ and then for $n+1$
$X_{n+1} > \gamma_{n+1}$.
Hence for 
infinitely many times $X_n=\gamma_n$
\end{exa}


\begin{rem}\label{rem:unifinverse}
It is clear that $(\Gamma, \mu)$ is uniformly strongly transient if and only if $(\Gamma, \mu^{-1})$ 
is uniformly strongly transient. Indeed, the probability that $X_i \ne \gamma_i$ for the random walk 
$(G,\mu)$ is equal to the probability that $X_i \ne \gamma'_i= \gamma_i^{-1}$ for the random walk $(G,\mu^{-1})$.
\end{rem}

Uniformly strongly transience and the lemma below were used in an essential way for our criterion of boundary non-triviality (in particular the application about blocks of dimension $\ge 3$). 


\begin{defn}\label{def:admissible}[Admissible equivalence relations].
Given an equivalence relations on $X  \times \mathbb{N}$, we say that this relation is admissible, if for any $x \ne  y$ and any $i \in \mathbb{N}$ it holds $(x,i)$
is not equivalent to $(y,i)$.
\end{defn}

\begin{defn}\label{def:range}
[Generalized range with respect to an equivalence relation $E$]
Let $X$ be a space equipped with a Markov kernel, and let $E$ be an equivalence relation on $X$. Given an $n$ step trajectory of the random walk $X_1$, \dots, $X_n$, we say that the generalized range of $R_E(X_1, \dots, X_n)$
with respect of $E$ is the number of equivalence classes that $(X_1,1)$, \dots, $(X_n,n)$ intersect.
We define the range function to be the expectation of the range:
$$
R_E(n) =E [R_E(X_1), \dots, X_n)]
$$
\end{defn}

\begin{rem}
Let the equivalence relation $E$ is defined by $(x,i) \sim (y,j)$ if and only if $x=y$. Then the range with respect to $E$ is the same as the range of the random walk, that is, the number of distinct points (in $X$) visited up to the time instant $n$.
\end{rem}

For a random walk on a group, the transience of a random walk is equivalent to
linearity of the range function (see Theorem $1.4.1$ Spitzer \cite{spitzer} for the case of $\mathbb{Z}^d$, and a similar argument works for an arbitrary not necessarily commutative group, see e.g. Lemma $1$ in \cite{dyubina}; for a more general case when instead of random walks on groups we consider the walks on Schreier graphs, there is an analogous statement where we have to consider inverted orbits of random walks, see Lemma 3.1 \cite{bartholdierschler}). In a similar way, linearity of the generalized range function with respect to an equivalence relation is related to uniform strong transience.

\begin{lem}\label{lem:linearrange} [Generalised range for equivalence relations] Let $G$ be a group and $\mu$ be a probability measure on $G$.
Consider an admissible equivalence relation $E$  on $G \times \mathbb{N}$. 
Assume that $\mu$ is uniformly strongly transient. Then the range function satisfies $R_E(n) \ge q n$, for $q>0$ and all $n$.
\end{lem}
\begin{proof}
We know that $\mu$, and hence also $\mu^{-1}$ are uniformly strongly transient.
Consider the probability that $X_n$ is equivalent to none of $X_i$, for $0 \le i \le n-1$.  Consider the $n$ step trajectory of the random walk,
started from $X_n$ and  inverting the time (putting $X_0'=X_n$, $X_1'=X_{n-1}$, $X_{n-1}'=X_1$, $X'_n= X_0 =e$).  Observe that the  probability we consider is an expectation that this inverted r.w. trajectory never returns to an element equivalent to its origin element $X'_0=X_n$ up to the time instant $n$. Here we consider an "inverted" equivalence relation $E'$ $(x,k) \sim_{E'} (y,m)$ if $(x, n-k) \sim_E' (y, n-m)$. This auxiliary equivalence relation depends on $n$, and if we move the base-point of the inverted random walk to $e$, then it also depends on $g_n=X_n$.

However, the uniform strong transience allows us to bound from below all considered probabilities (and hence their expectation), implying that the probability to visit at time $n$ an element, not equivalent to the previous ones, is $\ge p >0$.
Finally, observe that the range function with respect to $E$ is equal to the sum of random variables $\chi_i$, $\chi_i$ takes  value $1$ if $i$-th element is not equivalent to the previous ones, and $0$ otherwise. Hence we obtain a linear lower bound for the range functions, and this concludes the proof of the lemma.
\end{proof}

By a result of Varopoulos \cite{varopolos325} any non-degenerate random walk on an infinite group which is not a finite extension of $\mathbb{Z}$ or $\mathbb{Z}^2$ is transient. See e.g. Chapter I, Section 3B, Lemma 3.12 in \cite{woess} for a proof based on Nash-Williams criterion and
Chapter III, Section 14 \cite{woess} for a more general argument based on   Coulhon Saloff-Coste isoperimetric inequality. A similar argument is used in 1) of the theorem below to affirm the uniform strong transience.

A result of Dudley \cite{dudley} states that a countable abelian group admits a non-degenerate transient random walk if and only if this group does not contain $\mathbb{Z}^3$ as a subgroup. 
(Spitzer refers  to this result as the most interesting general results concerning random walks on groups (!), see final remark of section $8$ in  \cite{spitzer}). 
The theorem below generalises his result to not necessarily abelian groups.

\begin{thm}\label{thm: admisrecurrent}

\begin{enumerate} Let $G$ be a countable group.
\item If $G$ admits an infinite finitely generated subgroup which is not virtually $\mathbb{Z}$ or $\mathbb{Z}^2$, then any 
non-degenerate
random walk on $G$ is uniformly strongly transient. 
In particular, given an admissible equivalence relation $E$, the range of $(G,\mu)$ with respect to $E$ is $\ge pn$, for some $p>0$ and all $n$.

\item If all finitely generated subgroups of
$G$ have at most quadratic growth,
then $G$ admits a non-degenerate recurrent measure. This measure can be chosen to be symmetric and with $\supp \mu= G$.

    
\end{enumerate}
\end{thm}
\begin{rem} The assumption in the first claim of the theorem
 that the measure is non-degenerate
cannot be replaced
by the condition of being adapted (nor by the $S^+S^-$
property). Let $e_1$, $e_2$, \dots,  $e_d$
be standard generators of $\mathbb{Z}^d$, consider the equidistribution measure $\mu$ on this set.
Random walks $(\mathbb{Z}^d,\mu)$ are essentially $d-1$ dimensional, the $n$ step positions of the random walk belong to $d-1$ subspace of $\mathbb{Z}^d$ (depending on $n$).
Put  $\gamma_n=(n/d, n/d, n/d)$. Observe that the  probability to hit $\gamma_n$ is
equal to $1$ for $d=1$  and also equivalent to $\sim 1/n^{(d-1)/2}$ for all $d\ge 1$. Thus for $d=3$ the expected number of hitting $\gamma_n$ at time instant $n$ is
$\sim \sum_{i=1}^\infty  1/i = \infty$, and the sequence $(i,\gamma_i)$ is recurrent.
\end{rem}

\begin{proof}
1)
To prove the first claim of the theorem, observe that Gromov's theorem on polynomial growth  \cite{gromov} implies that any group satisfying the assumption of this claim has at least cubic growth: its growth function
$v_{G,S}(n) \ge A n^3$, for some positive constant
$A$, depending on the generating set $S$ and all $n$ (since any nilpotent group has growth $\sim n^d$, and the only groups of (sub)-quadratic growth are finite extensions of $\mathbb{Z}$, $\mathbb{Z}^2$ or finite ones, see e.g. Corollary $3.18$ in
\cite{woess}; for an exposition of basic facts about growth of groups see also
\cite{mannhowgroups}
). Consider a simple random walk on $G$, we denote its symmetric finitely supported measure by $\nu$, and we assume that $\nu= \nu_0^{*2}$. Recall that by Coulhon Saloff-Coste inequality
the F{\o}lner function of $G$ satisfies
$Fol_{G,S}(n) \ge v_{G,S}(C n) \ge \A' n^3$,
where $v_{G,S}(n)$ is the growth function of $G$ with respect to $S$.
Therefore, the return probability of the random walk $(G,\nu)$ satisfies $\nu^{*n}(e) \le C' \frac{1}{n^{3/2}}$. Indeed, this (well-known) fact follows from the inequality between isoperimetry and return probability, see e.g. Proposition 14.1 and Theorem 14.3 and for this particular claim Corollary 14.5 in Chapter III \cite{woess}.

Now recall  an argument that goes back to Baldi et al \cite{baldietal}. By
$(2.1)$  in Section $2$ of   \cite{varopoulos85} by Varopoulos the following holds. Take a symmetric random walk
defined by a measure $\mu_1$ and a not-necessarily symmetric non-degenerate  random walk defined by $\mu_2$. Then there exists $\alpha>0$ such that   
for any function $f \in l^2(G)$
$$
\langle P_{\mu_2} f, f \rangle \le \alpha \langle P_{\mu_1} f, f \rangle. 
$$
Indeed, since $\mu_2$ defines a non-degenerate
random walk,
we 
replacing if necessarily  $\mu_2$ by a convolution power, we can assume that
the 
Markov kernels satisfy $P_{\mu_2} \ge K P_{\mu_1}$. Then the inequality above is a particular case of the corresponding inequality ($(2.1)$ in \cite{varopoulos85}) for Markov kernels.
Applying this inequality to $f$ defined by $f(e)= f(g)=1$ and $f(x)=0$ otherwise, we conclude
that $P_{\mu}^{*n}(e,e) + P_{\mu}^{*n}(e,g)
\le C(P_{\nu}^{*n}(e,e) + P_{\nu}^{*n}(e,g)),
$
for a positive constant $C$ not depending on $g$. Since $\nu= \nu_{0}^{*2}$ is symmetric, we know that $P_{\mu}^{*n}(e,g) \le  P_{\nu}^{*n}(e,e)$ (see Remark \ref{rem:CauchySch}), and hence
$$
\mu^{*n}(g)=P_{\mu}^{*n}(e,g) \le 2 C P_{\nu}^{*n}(e,e)\le
\frac{C''}{n^{3/2}} 
$$
for a positive $C''$ not depending on $g$.
Thus in view of Remark \ref{rem:unif} the random walk is strongly transient.

Now since we have proved the uniform strong transience, the claim about the range function follows from Lemma \ref{lem:linearrange}.

2) Choose finite generated subgroups $G_i \subset G$
such that  $G_i \subset G_{i+1}$ and $\cup_{i \in \mathbb{N}} G_i =G$.
Choose a measure $\mu_i$ on $G_i$.
We will construct the measure $\mu$ as a convex combination of $\mu_i$. More precisely, we will choose in a recursive way 
a sequence $a_1=1$, $a_i>0$. On each step we construct $a_1$, \dots, $a_n$
and a number $b_n$ such that the numbers $a_i, i >n$, constructed in the sequel will satisfy $a_{n+1}+a_{n+2}+a_{n+3} + \dots \le b_n$ (and hence $\frac{ a_{n+1}+a_{n+2}+a_{n+3}+ \dots} {a_{1}+a_{2}+a_{3}+\dots} \le b_n$). 
We will check the claim for the measure
$$
\mu = \sum_{i=1}^{\infty} \frac{a_i}{a_1+a_2 +a_3 + \dots} \mu_i.
$$

First choose  all  measures $\mu_n$ to be finitely supported. 
Indeed, if the group of growth $\sim n^d$, then 
 return probability of a simple  random walk satisfies $\sim \frac{C}{n^{d/2}}$, for all even $n$, (see e.g. Theorem 15.8 in \cite{woess}); Here
(in the proof of 2)  we use only the case for groups of at most quadratic growth, which, taking into account the polynomial growth theorem,  can also be deduced from the local limit theorem for the random walks on $\mathbb{Z}^2$ and $\mathbb{Z}$).
Indeed $G_i$ is either finite, or virtually cyclic, or contains $\mathbb{Z}^2$ as a finite index subgroup. 

Observe that the probability that one of the terms corresponding to $\mu_i$, $i>n$, occurs at least once among the first $N_n$ increments of the random walk $(G,\mu)$, is at most $d_n=\frac{N_n b_n}{\sum_{i=1}^\infty a_i} \le N_n b_n$.
We  choose a sequence $N_n$ which increases quickly enough and the sequence $b_i$ decreasing  quickly enough (in terms of $N_n$)
so that $N_n b_n \le 1/2$.
(If we assume moreover that 
$\sum N_n b_n <\infty$, then by Borel-Cantelli lemma we know the following. The event: for some $n$, there is at least one increment corresponding to $\mu_i$, $i>n$, among first $n$
increments of the random walk $(G,\mu)$ occurs only for finitely many $n$. We show in this case that return probability for $\mu$ is estimated by these probabilities for its truncations).

Having chosen $a_1$, $a_2$ \dots $a_n$, consider the probability measure $\bar{\mu}_n$ which is the normalised sum $a_1 \mu_1 + a_2 \mu_2 + \dots + a_n \mu_n$.
Consider $C_n$ such that the return probability of the random walk $\bar{\mu}_n$ satisfies ${\bar{\mu}_n}^{t} \ge \frac{C_n}{n}$, for all even $t \ge 0$.
We choose $C_n$ in such a way that this sequence is non-increasing $C_1 \ge C_2 \ge C_3 \dots$.

Observe that
$$
\mu^{*t}(e) \ge (1-d_n) \mu_n^{*t}(e) \ge \frac{1}{2}\mu_n^{*t}(e) \ge \frac{C_n}{2n} 
$$
for any $t \le N_n$,
since with probability $\ge 1-d_n$ no increments of the $\mu_i$, $i>n$ occur during first $N_n$ steps of the r.w.
Having fixed $a_1$, $a_2$, \dots $a_n$ and $C_n$, we choose $N_n$ large enough so that  
$$
\sum_{t=1}^{N_n} \frac{C_n}{2n} \ge n,
$$
here $t$ in the sum above is assumed to be even.
We conclude therefore that for all $n\ge 1$
$$
\sum_{t=1}^{\infty} \mu^{*t}(e) \ge n.
$$
Therefore, this sum is infinite, and
thus the random walk $(G,\mu)$ is recurrent.
This proves the main claim of 2). 

Now
if we want to assure that the measure on $G$ has full support, we choose $\mu_i$ such that $\supp \mu_i=G_i$ (rather than having finite support).
We assume that  $\mu_i$ is symmetric and has finite second moment. In this
case the truncated measure $\bar{\mu}_i$ are also symmetric measure of finite second moment.
Then we still have $\bar{\mu}_i^{*n}(i) \ge C_i/n$, for some $C_i>0$, and for all even $n$ (see Corollary $1.5$ in \cite{pittetsaloffstability}). Thus we can  proceed as as in the argument explained above.

\end{proof}

\end{document}